\documentclass{iopjournal}
\usepackage{amsmath}
\usepackage{comment}
\usepackage{algorithm}
\usepackage{algpseudocode}%,caption}
\usepackage{array}  
\usepackage{braket,amsfonts,amsopn}
\usepackage{subfigure}
\usepackage{epstopdf}
\usepackage{graphicx}
\usepackage{lipsum}
\usepackage{tikz}
\usepackage{multicol}

\usepackage{mathrsfs}

\def\jmpu{{\lbrack\!\lbrack \widetilde{\gamma}> \\
  <g,\widetilde{\Balpha}> \\
<g,\widetilde{\Bbeta}> \end{bmatrix}$, whereu\rbrack\!\rbrack}}

\def\sC{{ C}}
\def\sS{{\mathcal S}}

\def\sG{{\mathcal G}}

\def\sJ{{\mathcal J}}

\def\sS{{\mathcal S}}

\def\sD{{\mathcal D}}
\def\sG{{\mathcal G}}

\def\Beta{\mbox{\boldmath$\eta$}}

\def\Balpha{\mbox{\boldmath$\alpha$}}
\def\Bbeta{\mbox{\boldmath$\beta$}}

\def\Beta{\mbox{\boldmath$\eta$}}

\def\Blambda{\mbox{\boldmath$\lambda$}}

\def\bd{\mathbf{d}}

\def\bg{\mathbf{g}}

\def\bu{\mathbf{u}}
\def\bv{\mathbf{v}}

\newcommand{\half}{\frac{1}{2}}

\newcommand{\LRp}[1]{\left( #1 \right)}
\newcommand{\LRs}[1]{\left[ #1 \right]}

\newcommand{\GM}[2]{\mc{N}\left( #1, #2 \right)}

\newcommand{\db}{\mathbf{d}}
\newcommand{\ub}{\mathbf{u}}

\newcommand{\umap}{\ub^{MAP}}

\newcommand{\nor}[1]{\left\| #1 \right\|}
\newcommand{\norinf}[1]{\left\| #1 \right\|_\infty}
\newcommand{\snor}[1]{\left| #1 \right|}
\newcommand{\norm}[1]{\left \Vert #1 \right \Vert}
\newcommand{\grad}{\nabla}

\newcommand{\mc}[1]{\mathcal{#1}}

\newcommand{\R}{\mathbb{R}}
\renewcommand{\L}{{{\Sigma}}}
\newcommand{\A}{A}
\newcommand{\G}{\mc{G}}

\newcommand{\N}{N}
\newcommand{\M}{M}
\newcommand{\n}{n}

\newcommand{\J}{\mc{J}}
\newcommand{\D}{\mc{D}}
\newcommand{\La}{\mc{L}}

\newcommand{\vb}{\boldsymbol{v}}

\newcommand{\sigb}{{\boldsymbol{\sigma}}}

\newcommand{\Dstoch}{\tilde\D}

\newcommand{\Dexpect}{\mathscr{D}}

\newcommand{\DexpectSA}{\Dexpect_{\N}}
\newcommand{\DexpectSi}{\Dexpect_{i}}

\def\rand{\mbox{\boldmath$\xi$}}

\newcommand{\randP}{\mathcal{\pi}}

\newcommand{\del}{{\boldsymbol{\delta}}}

\newcommand{\taub}{{\boldsymbol{\tau}}}

\newcommand{\delP}{\randP_{\del}}

\newcommand{\sigP}{\randP_{\sigb}}

\newcommand{\subjectTo}{{subject \ \  to: }}

\newcommand{\eb}{\boldsymbol{e}}

\renewcommand{\Pr}[1]{\mathbb{P}\LRs{#1}}
\newcommand{\B}{{B}}
\renewcommand{\b}{\boldsymbol{b}}

\newcommand{\E}{\mc{E}}

\newcommand{\Ex}{{\mathbb{E}}}
\newcommand{\expect}{\Ex}

\algblock{ParFor}{EndParFor}
\algnewcommand\algorithmicparfor{\textbf{For all}}
\algnewcommand\algorithmicpardo{\textbf{do in parallel}}
\algnewcommand\algorithmicendparfor{\textbf{end\ For all}}
\algrenewtext{ParFor}[1]{\algorithmicparfor\ #1\ \algorithmicpardo}
\algrenewtext{EndParFor}{\algorithmicendparfor}

\newenvironment{proof}{\paragraph{\textit{Proof:}}}{}

\newtheorem{proposition}{Proposition}
\usepackage{comment}
\newtheorem{remark}{Remark}
\newtheorem{theorem}{Theorem}
\newtheorem{corollary}{Corollary}
\usepackage{algorithm} %for writing pseudoalgorithms
\begin{document}

\articletype{Paper} %	 e.g. Paper, Letter, Topical Review...

\title{A New Look at the Ensemble Kalman Filter for Inverse Problems: Duality, Non-Asymptotic Analysis and Convergence Acceleration}

\author{C G Krishnanunni$^1$, Jonathan Wittmer$^*$,  Tan Bui-Thanh$^{1,*}$, Quoc P. Nguyen$^2$}

\affil{$^1$Department of Aerospace Engineering \& Engineering Mechanics, UT Austin}

\affil{$^2$Hildebrand Department of Petroleum and Geosystems Engineering}

\affil{$^*$Oden Institute for Computational Engineering and Sciences,  UT Austin: now at Facebook research.}

\email{krishnanunni@utexas.edu}


\begin{abstract}
This work presents new results and understanding of the Ensemble Kalman filter (EnKF) for inverse problems.   In particular, using a Lagrangian dual perspective we show that EnKF can be derived from the sample average approximation (SAA) of the  Lagrangian dual function.
The beauty of this new duality perspective is that it facilitates us to prove and numerically verify a novel non-asymptotic convergence result for the EnKF. Motivated by the new perspective, we also present a new convergence improvement strategy for the Ensemble Kalman Inversion Algorithm (EnKI), which is an iterative version of the EnKF for inverse problems.
In particular, we propose an adaptive multiplicative correction to the sample covariance matrix at each iteration and we call this new algorithm as EnKI-MC (I). Based on the new duality perspective, we derive an expression for the optimal correction factor at each iteration of the EnKI algorithm to accelerate the convergence. In addition, we also consider an ensemble specific multiplicative covariance correction strategy (EnKI-MC (II)) where a different  correction is employed for each ensemble.
By viewing EnKI through the lens of fixed-point iteration, we also provide theoretical results that guarantees the convergence of EnKI-MC (I) algorithm. 
 Numerical investigations for the deconvolution problem, initial condition inversion in advection-convection problem, initial condition inversion in a Lorenz 96 model, and inverse problem constrained by elliptic partial differential equation are conducted to verify the non-asymptotic results for EnKF and to assess the performance of convergence improvement strategies for EnKI. The numerical results suggest that the proposed strategies for EnKI not only led to faster convergence in comparison to the currently employed techniques but also  better quality solutions at termination of the algorithm.
\end{abstract}

\section{Introduction}

The field of data assimilation seeks to find the best estimate for the unknown states or parameters in a dynamical system by combining appropriate mathematical models with observations and balancing the uncertainties \cite{kalnay2003atmospheric, oliver2008inverse}.  The  Kalman filter, with extensions and generalizations such as the Extended Kalman filter,  Unscented Kalman filter and Ensemble Kalman filter, has emerged as one of the most popular data assimilation tool over the past few decades \cite{kalman1960new, julier2004unscented, evensen2003ensemble}. In particular, the Ensemble Kalman filter (EnKF) has been developed as a recursive filter suitable for state-parameter estimation problems  involving high dimensional PDEs \cite{evensen2003ensemble}.

The EnKF addresses the issue of expensive covariance matrix computation by replacing it with the sample covariance matrix of an ensemble. The EnKF may also be viewed as a derivative-free optimization technique, with the ensemble used as a proxy for derivative information.  The EnKF propagates $N$ ensembles through the dynamics, and subsequently update
this prior forecast ensemble into an updated analysis ensemble that assimilates the new observation. Multiple variants of EnKF for state-parameter estimation problems in dynamical systems are available in the literature \cite{evensen2009data, kalnay2003atmospheric}. Among them covariance inflation and localization are two techniques that attempts to avoid filter divergence due to sampling errors associated with the use of a small ensemble size \cite{anderson1999monte, hamill2005accounting}. Albers et al. \cite{albers2019ensemble} have developed an Ensemble Kalman filter capable of imposing equality or inequality constraints on ensemble based parameter and state estimation problems respectively. The popularity of EnKF as a data assimilation tool  may be attributed to its effectiveness even when used with small ensemble sizes \cite{bergemann2010mollified}. However, much of the efforts in the research community has been devoted to proving the asymptotic convergence in the large ensemble limit \cite{li2008numerical, le2009large, kwiatkowski2015convergence}.  Several
works  \cite{bergemann2010localization,kelly2014well,majda2018performance} have noted that large ensemble limit analysis fail to explain empirical
results that report good performance with a moderately sized ensemble even in high-dimensional problems.  This highlights the need for a non-asymptotic convergence analysis of the EnKF that establishes how the ensemble size influences the error in approximating the solution (i.e the posterior mean and covariance). 

The possibility of using the Ensemble Kalman method as an inverse problem solver has been explored recently by various researchers  where it is commonly referred to as the Ensemble Kalman Inversion (EnKI) \cite{schillings2017analysis, iglesias2013ensemble}. Let $\G:\ \R^n\mapsto \R^m$ be a continuous mapping.  The inverse problem aims to find the unknown parameters $\ub \in \R^\n$ given measurements $\db$ of the form:
\begin{equation}
\db=\G(\ub)+\Beta,
\label{forward}
\end{equation}
where  $\db \in \R^m$, $\Beta$ is  the observational noise vector. In the context of the Ensemble Kalman Inversion (EnKI) algorithm, one considers a set of initial guess estimates for the solution, referred to as ensembles, denoted by $\{\ub^{(i)}\}_{i=1}^N$, where $N$ is the number of ensemble members. These ensembles are typically generated based on prior knowledge, often modeled by a prior distribution. Given a forward mapping $\G$ and an observation vector $\db$, the EnKI algorithm transforms each ensemble member $\ub^{(i)}$ to produce an  estimate of the inverse solution, denoted by $\LRp{\ub^{(i)}}^*$.
Schillings et al. \cite{schillings2017analysis} has analyzed the method with fixed ensemble size based on deriving a continuous time limit of the iterative scheme. The EnKI has an interesting property that each ensemble estimate (inverse solution $\LRp{\ub^{(i)}}^*$) it produces lies within the linear span of the initial ensembles $\{\ub^{(i)}\}_{i=1}^N$. This property is commonly referred to as the “subspace property” of EnKI \cite{schillings2017analysis}. Chada et al. \cite{chada2018analysis} have attempted to break this subspace property by incorporating various hierarchical approaches which allows to learn  information from
the data. Iglesias et al. \cite{iglesias2016regularizing} have merged ideas from iterative regularization with ensemble Kalman inversion
to develop a derivative-free stable method. Blömker et al. \cite{blomker2019well} used a continuous time limit perspective of the EnKI to derive a system of stochastic differential equations that allows one to understand the convergence properties of the method. More recently, Chada et al. \cite{chada2020tikhonov} has demonstrated how further regularization in the form of Tikhonov-like Sobolev penalties can be imposed in the Ensemble Kalman inversion framework.  An adaptive regularization strategy for the EnKI has been  proposed  by Iglesias et al. \cite{iglesias2021adaptive} which is based on the interpretation of EnKI as a Gaussian approximation in the Bayesian tempering setting for inverse problems. Solving inverse problems often requires imposition of additional constraints on the parameters. In order to deal with this issue,  Chada et al. \cite{chada2019incorporation} proposed an EnKI motivated by the theory of projected preconditioned gradient flows. EnKI algorithm finds extensive application in scientific computing, where it has been successfully deployed in a variety of inverse problems and data assimilation tasks \cite{kovachki2019ensemble, schneider2020learning, blomker2018strongly}.
More detailed literature review on the theoretical aspects of EnKI is presented in section \ref{from_enkf_enki}. One of the main practical challenges in deploying the EnKI iterative scheme is the computational cost arising from the need to perform $N$
expensive evaluations of the forward map $\sG$ at each iteration. This makes it essential to develop convergence-acceleration techniques for the EnKI algorithm that can reduce the total number of forward evaluations and thereby improve overall computational efficiency.

%EnKI algorithm finds extensive application in scientific computing, where it has been successfully deployed in a variety of inverse problems and data assimilation tasks. Notably, it has been applied to empirical risk minimization in machine learning which is otherwise traditionally solved by stochastic gradient descent (SGD) and variants \cite{kovachki2019ensemble}. Schneider et al. \cite{schneider2020learning} employed EnKI to infer unknown parameters in a stochastic differential equation (SDE) model from data. Furthermore, Blömker et al. \cite{blomker2018strongly} developed a strongly convergent numerical scheme for certain SDEs based on the EnKI algorithm. 

In this work, we present new results on the understanding of the Ensemble
Kalman filter (EnKF) for inverse problems. To that end, let us constrain the discussion to the case where the map $\G$ in (\ref{forward}) is linear and we denote it by $\sG=\A$.
It is interesting to note that a conventional derivation of the Kalman filter for inverse problem requires an application of Sherman-Morrison-Woodbury formula \cite{deng2011generalization} to rewrite the least-squares solution as follows \cite{hager1989updating}: 
\begin{align}
   \umap =& \arg\min_\ub \half\norm{\db -
    \A\ub}^2_{\L^{-1}} + \half \norm{\ub - \ub_0}^2_{\sC^{-1}}\label{unco}\\
    =&\LRp{\A^T\L^{-1}\A+\sC^{-1}}^{-1}\LRp{\A^T\L^{-1}\db+\sC^{-1}\ub_0}
    \label{unco_us}
\end{align}    
\[ \quad \quad \quad \quad \quad \quad \quad \quad\Updownarrow \mathrm{(Sherman-Morrison-Woodbury \ formula)}\]
\begin{equation}
  \hspace{-0.5 cm}   \umap = \ub_0 +\sC\A^T\LRp{\L + \A\sC\A^T}^{-1}\LRp{\db -\A\ub_0}
\label{kal_der}
\end{equation}
where $\sC \in \R^{\n\times \n}$ is the prior covariance matrix, $\ub_0\in \R^n$ is the prior mean and $\L \in\R^{m\times m}$ is the data covariance matrix. Note that \eqref{kal_der} is nothing but the state update equation in the context of Kalman filters (single assimilation step) and the quantity $K=\sC\A^T\LRp{\L + \A\sC\A^T}^{-1}$ is commonly known as the Kalman gain. Instead of following this conventional derivation, in this work we derive the Kalman filter equation \eqref{kal_der} as the solution to a Lagrangian dual problem. 
Via this viewpoint and by dealing with a randomized Lagrangian dual function, we show that a novel non-asymptotic convergence result can be derived for the ensemble Kalman filter. Section \ref{dual_1} and \ref{randomized_sec} deals with the derivation of the Ensemble Kalman Filter as a Randomization of the Lagrangian Dual Problem, and Theorem \ref{non_asymp} provides the novel non-asymptotic error estimator for EnKF. Further, in section \ref{conv_imp_1}, \ref{conv_imp_2} we show how the new duality perspective may be used to develop a convergence improvement strategy for the Ensemble Kalman Inversion Algorithm (EnKI), which is an iterative version of the EnKF for inverse problems. In particular, we propose an adaptive multiplicative correction to the sample covariance matrix at each iteration of EnKI, and we call this new algorithm as EnKI-MC (I). Such a strategy for convergence acceleration of the iterative scheme has been previously employed by Chada et al. \cite{chada2019convergence} where a heuristic multiplicative correction factor of the form $\alpha(k)=h_0k^{\beta}$ ($0\leq \beta\leq 0.8$) has been considered to accelerate the convergence, where $k$ is the iteration number of EnKI. Our work is a significant departure from Chada et al. \cite{chada2019convergence}  by using the duality perspective to derive an expression for the optimal correction factor $\alpha(\ub_n,\ n)$, where $\ub_k=[\ub_k^{(1)},\dots \ub_k^{(N)}]$ denotes the solution  at the $k^{th}$ iteration, and $\ub_k^{(i)}$ denote the $i^{th}$ ensemble member. Using the duality perspective, an objective function is constructed  to derive the optimal $\alpha(\ub_k,\ k)$ to accelerate the convergence of EnKI while also promoting convergence to the true solution provided that the true solution lies in the linear span of initial ensembles (see section \ref{design}). Note that the correction factor $\alpha(\ub_k,\ k)$ in our case not only depends on $k$ but also on the  solution $\ub_k$. 
Further in section \ref{conv_imp_2}, we also consider an ensemble-specific multiplicative covariance correction strategy (EnKI-MC (II)) where a different correction factor $\alpha^{(i)}(\ub_k,\ k)$ is employed for the $i^{th}$ member of the ensemble.
By viewing EnKI through the lens of fixed-point iteration, Theorem \ref{mult_conv} discusses the convergence behaviour of EnKI-MC (I) algorithm. In section \ref{numeric_section} we provide numerical experiments for prototype linear and non-linear inverse problems  to verify the non-asymptotic results for EnKF and to assess the performance of convergence improvement strategies for EnKI. The numerical results suggest that the proposed strategies for EnKI not only led to faster convergence in comparison to the currently employed techniques but also led to better quality solutions.

\section{Derivation of the Kalman Filter state update equation from Duality}
\label{dual_1}
In this work, all vectors are denoted in boldface and matrices are represented in 
capital letters. We also use the notation $\ub_k^{(i)}$ to denote the solution for the $i^{th}$ ensemble member at the $k^{th}$ iteration of the EnKI algorithm, $\ub_k=[\ub_k^{(1)},\dots \ub_k^{(N)}]$ denotes the collection of all ensembles at the $k^{th}$ iteration, $\sC_k:= \sC(\ub_k)$ denotes the sample covariance matrix formed using the $N$ ensembles.

We begin by looking at the state update equation \eqref{kal_der} which serves as the backbone of the Kalman filter. In this section, we will derive  \eqref{kal_der} as the solution to a Lagrangian dual problem. The duality view not only provides new insight into the Kalman filter but also (as we shall show) leads to a new derivation of the ensemble Kalman filter as a randomization of the dual problem.

{\em Our departure from the mainstream of Kalman filter and ensemble Kalman filter literatures is to cast the unconstrained optimization problem \eqref{unco} into a constrained optimization problem} by introducing an artificial constraint as follows:
\begin{equation}
\begin{aligned}
     \label{constrainedOpt}
            &\underset{\ub,\bv}{\min}\ \J(\ub;\bv;\ub_0, \db):=\half\norm{\db -
    \bv}^2_{\L^{-1}} + \half \norm{\ub - \ub_0}^2_{\sC^{-1}}, \\
           &  \subjectTo  \quad \quad \ \A\ub=\bv,
\end{aligned}    
\end{equation}
where $\bv\in \R^m$. The Lagrangian function $\La: \R^n\times \R^m\times  \R^m\rightarrow \R$ for the constrained problem  (\ref{constrainedOpt}) reads:
      \begin{equation}
        \La (\ub,\bv,\Blambda)=\half\norm{\db -
    \bv}^2_{\L^{-1}} + \half \norm{\ub - \ub_0}^2_{\sC^{-1}}+\Blambda^T \LRp{\A\ub-\bv},
    \label{Lag_function}
    \end{equation}
where $\Blambda \in \R^m$ is the Lagrangian multiplier. The Lagrange dual function $\D:  \R^m\rightarrow \R $ for (\ref{constrainedOpt}) is defined as:
     
     \begin{equation}
         \D(\Blambda) := \inf_{\ub,\bv} \La (\ub,\bv,\Blambda).
         \label{dual_function}
     \end{equation}
    By setting the derivatives of $\La$ with respect to $\ub$ and $\bv$ to zeros, and then solving for $\ub,\bv$ as a function of $\Blambda$ we obtain the following dual problem:

      \begin{eqnarray}
      \max_{\Blambda}\D(\Blambda)%&=\half \Blambda^T \L \Blambda +\half \Blambda^T \A \sC \A^T \Blambda- \Blambda^T\A \sC \A^T \Blambda +\Blambda^T \A \ub_0-\Blambda^T \bd-\Blambda^T\L \Blambda \\
       := -\half \Blambda^T \LRp{\L+\A\sC\A^T}\Blambda+\Blambda^T\LRp{\A\ub_0-\bd}.
    \label{dual_function_new}
     \end{eqnarray}
The first order optimality condition for (\ref{dual_function_new}) reads: %is given as:
 \begin{equation}
     \Blambda^*=\LRp{\L+\A\sC\A^T}^{-1}\LRp{\A\ub_0-\bd}  
     \label{lambda_optimal}
  \end{equation}
On the other hand, the KKT conditions for (\ref{constrainedOpt}) reads:
\begin{equation}
     \begin{aligned}
      \label{optimal_one}
 & \frac{\partial \La}{\partial \bv}=\L^{-1}\LRp{\bv-\bd}-\Blambda=0 && \Rightarrow \bv=\bd+\L\Blambda,\\
  & \frac{\partial \La}{\partial \ub}=\sC^{-1}\LRp{\ub-\ub_0}+\A^T\Blambda=0 && \Rightarrow \ub=\ub_0-\sC\A^T\Blambda,\\
 &  \frac{\partial \La}{\partial \Blambda} = \A\ub-\bv = 0  && \Rightarrow \A\ub=\bv.    
 \end{aligned}
\end{equation}
 We conclude that the optimal solution for (\ref{constrainedOpt})\textemdash the MAP point\textemdash is given by:  
\begin{equation}
      \umap=\ub^*=\ub_0-\sC\A^T\Blambda^* = \ub_0+\sC\A^T\LRp{\L+\A\sC\A^T}^{-1}\LRp{\bd-\A\ub_0}.
      \label{optimal_MAP}
  \end{equation}
\begin{remark}
Note that the dual optimality condition (\ref{lambda_optimal}) is part of the KKT conditions \eqref{optimal_one}. Furthermore, the invertibility of $\sC^{-1}$ and $\L^{-1}$ is necessary for validity of solution (\ref{optimal_MAP}) and is consistent with the conditions imposed by Sherman-Morrison-Woodbury formula used in the conventional derivation of (\ref{optimal_MAP}). In our derivation, the conditions arise naturally via the KKT conditions 
%while deriving the KKT conditions 
\eqref{optimal_one} for \eqref{constrainedOpt}.% and Eq. \eqnref{optimal_two}.
\end{remark}
\section{Derivation of the Ensemble Kalman Filter as a Randomization of the Dual Problem}
\label{randomized_sec}

  Let us define:
  
    \begin{equation}
    \label{rand_variables}
    \sigb \sim \sigP := \GM{0}{\L}, \quad
\del \sim \delP := \GM{0}{\sC}.
    \end{equation}
Thus, $\randP\LRp{\sigb,\del} := \sigP \times  \delP $ is the probability density of:
    \begin{equation*}
        \rand := \LRs{\sigb,  \del}. %\in \randSpace
    \end{equation*}
Consider the following stochastic Lagrangian dual function as:
        \begin{equation}
        \Dstoch\LRp{\Blambda; \ub_0, \db, \rand}:= 
        -\half \Blambda^T \LRp{\L+\A \del \del^T\A^T}\Blambda+\Blambda^T\LRp{\A\LRp{\ub_0+\del}-\bd-\sigb}.
         \label{stochastic_cost_D}
        \end{equation}   
Clearly,
    \begin{equation}
      \label{expect_cost}
%        \Dexpect \LRp{\Blambda; \ub_0, \db} := 
                \Dexpect \LRp{\Blambda} = 
        \expect_\randP \LRs{\Dstoch \LRp{\Blambda;\ub_0, \db, \rand}},
    \end{equation}
that is, (\ref{stochastic_cost_D}) is an unbiased estimator of the dual function (\ref{dual_function_new}).
Let $\LRs{\del^{(1)},\del^{(2)}...\del^{(N)}}$ and $\LRs{\sigb^{(1)},\sigb^{(2)}...\sigb^{(N)}}$ be ensemble pairs from $\randP\LRp{\sigb,\del}$. We can approximate the expectation on the right hand side of (\ref{expect_cost}) via the sample average approximation (SAA), i.e.,
\begin{equation}
  \DexpectSA := -\frac{1}{2N}\sum_{i=1}^N \Blambda^T \LRp{\L+\A \del^{(i)}{\del^{(i)}}^T\A^T}\Blambda+\frac{\Blambda^T}{N}\LRp{\A\LRp{\ub_0+\del^{(i)}}-\bd-\sigb^{(i)}}.  
  \label{sample_average_dual}
\end{equation}
If we define %$\Omega=\frac{1}{\sqrt{N}}\LRs{\del_1,\del_2,...,\del_N}$,
% where, $\Omega$ is a $n\times N$ matrix given as:   
    \[\Omega=\frac{1}{\sqrt{N}}\LRs{\del^{(1)},\del^{(2)}...\del^{(N)}} ,\quad \overline{\del}=\frac{1}{N}\sum_{i=1}^N \del^{(i)}, \quad  \overline{\sigb}=\frac{1}{N}\sum_{i=1}^N \sigb^{(i)},
    \]
    then the SAA optimal solution for the dual problem reads:
  %  \[ \overline{\sigb}=\frac{1}{N}\sum_{i=1}^N \sigb_i \quad \quad \sigb_i \sim \GM{0}{\L}\]
%    Note that the solution to maximizing the sample averaged cost function Eq. \eqnref{sample_average} considering finite sample size $N$ is obtained from the first order optimality condition as:
 \begin{equation}
        \label{optimal_1}
            \overline{\Blambda}^*:= \arg\max_{\Blambda} \DexpectSA =\LRp{\L+\A\Omega\Omega^T\A^T}^{-1}\LRp{\A\LRp{\ub_0+\overline{\del}}-\LRp{\bd+\overline{\sigb}}},
    \end{equation}
    and the optimal dual function $ \DexpectSA^*$ is given by:
    \begin{equation}
        \DexpectSA^* := -\frac{1}{2}(\overline{\Blambda}^*)^T \LRp{\L+\A \Omega\Omega^T\A^T}\overline{\Blambda}^*+(\overline{\Blambda}^*)^T\LRp{\A\LRp{\ub_0+\overline{\del}}-\LRp{\bd+\overline{\sigb}}}. 
        \label{optimal_dual}
    \end{equation}
Consequently the induced SAA optimal solution for the primal problem (\ref{constrainedOpt}) is given as:

 \begin{equation}
    \hat{\ub}^* := \ub_0-\sC\A^T \overline{\Blambda}^* = \ub_0+\sC\A^T \LRp{\L+\A\Omega\Omega^T\A^T}^{-1}\LRp{\bd+\overline{\sigb} - \A\LRp{\ub_0+\overline{\del}}}.
     \label{optimal_solution_random_kalman}
  \end{equation}
 If we further randomize $\ub_0$ and $\sC$ with $\ub_0 + \overline{\del}$ and $\Omega\Omega^T$ in (\ref{optimal_solution_random_kalman}) we obtain
    \begin{equation}
    \overline{\ub}^* :=  \ub_0 + \overline{\del}+\Omega\Omega^T\A^T \LRp{\L+\A\Omega\Omega^T\A^T}^{-1}\LRp{\bd+\overline{\sigb} - \A\LRp{\ub_0+\overline{\del}}},
   % \eqnlab{ensemble_kalman_inversion}
    \label{optimal_solution_random}
  \end{equation}
    which is exactly the well-known ensemble Kalman filter (EnKF).
Now unrolling the sums and defining
   \begin{equation}
       \ub^{{EnKF}}_{\N} = \frac{1}{N}\sum_{i=1}^N{{\ub}^{(i)}}^*,
       \label{ensemble_kalman_inversion}
   \end{equation}
we have:
     \begin{equation}
    {{\ub}^{(i)}}^* :=  \ub_0 + \del^{(i)}+\Omega\LRp{\A \Omega}^T \LRp{\L+\A\Omega\LRp{\A\Omega}^T}^{-1}\LRp{\bd+\sigb^{(i)} - \A\LRp{\ub_0+\del^{(i)}}}
     \label{each_ensemble_kalman_inversion},
  \end{equation}
  which is nothing more but the ensemble update equation in Ensemble Kalman Filter (EnKF) framework. That is, we have rediscovered the EnKF from randomizing the dual problem and the relationship between the primal and dual optimal solutions. 
 \begin{remark}
     The importance of \eqref{ensemble_kalman_inversion} is that it allows one to compute an estimate of the inverse solution $\umap$ at a much cheaper cost than if one were to use the actual prior covariance matrix $C$. To understand this, let us consider the main computational cost associated with computing $\sC\A^T\LRp{\L+\A\sC\A^T}^{-1}$ in \eqref{optimal_MAP} vs. the cost associated with $\Omega\LRp{\A \Omega}^T \LRp{\L+\A\Omega\LRp{\A\Omega}^T}^{-1}$ in \eqref{each_ensemble_kalman_inversion}.
     \begin{enumerate}
         \item With $\sC=LL^T$ (Cholesky decomposition), let's look at cost  for  forming $L\LRp{\A L}^T \LRp{\L+\A L\LRp{\A L}^T}^{-1}$: The main cost involves computing the multiplications $AL$, $L(AL)^T$, $(AL)(AL)^T$ and the cost associated with inversion $\LRp{\L+\A L\LRp{\A L}^T}^{-1}$. The total cost is  $\mathcal{O}\LRp{mn^2+m^2n+m^3}$\footnote{Here $\mathcal{O}$ denotes the order of computation. In the estimate we considered the standard cost of inversion rather than a reduced cost if one uses the Woodbury identity.}.
        \item For  forming $\Omega\LRp{\A \Omega}^T \LRp{\L+\A\Omega\LRp{\A\Omega}^T}^{-1}$: The main cost involves computing the multiplications $A\Omega$, $\Omega(A\Omega)^T$, $(A\Omega)(A\Omega)^T$ and the cost associated with inversion $\LRp{\L+\A \Omega\LRp{\A \Omega}^T}^{-1}$. The total cost is $\mathcal{O}\LRp{mnN+m^2N+m^3}$.
     \end{enumerate}
 \end{remark}
Therefore, for high dimensional problems it is clear that when $N<<n$, the cost associated with computing \eqref{ensemble_kalman_inversion} is much lower than \eqref{optimal_MAP}. Further, one only needs to store $n\times N$ elements of $\Omega$ instead of storing  $\frac{n(n+1)}{2}$ elements for $L$. Now the main question that remains is to characterize the error $\nor{\umap - \ub^{{EnKF}}_{\N}}_2 $ when one uses $N$ ensembles, and how the error decreases as $N$ increases. Using the new duality perspective, in section \ref{non_asym_en}, we provide answer to this question via 
non-asymptotic convergence analysis\footnote{Our focus is on analyzing a single ensemble
Kalman update as given in \eqref{optimal_solution_random} rather than on investigating the propagation of error across multiple updates in the context of a dynamical system.}. 
  
\section{Non-Asymptotic Convergence results for linear inverse problem}
\label{non_asym_en}
Based on the duality viewpoint presented in section \ref{randomized_sec}, we will derive a novel non-asymptotic convergence result for the EnKF. Note that the asymptotic convergence of  EnKF, i.e. 
\[\ub^{{EnKF}}_{\N} \xrightarrow[]{a.s} \umap,\quad \mathrm{as}\ N\rightarrow \infty,\]
is a well-known result in literature \cite{wittmer2023unifying}. Wittmer et al. \cite{wittmer2023unifying}  provided a unified framework for non-asymptotic convergence analysis of several randomized methods obtained by randomizing the inverse solution \eqref{unco_us}. However, their framework does not address the non-asymptotic convergence of the Ensemble Kalman filter \eqref{ensemble_kalman_inversion}, which is a randomized counterpart of \eqref{kal_der} and requires a fundamentally different analytical treatment.
In this section we derive a novel non-asymptotic convergence result for EnKF. 
 In particular, we wish to estimate the error between the actual MAP point $\umap$ in (\ref{optimal_MAP}) and the EnKF solution $\ub^{{EnKF}}_{\N}$ in (\ref{ensemble_kalman_inversion}). We begin by decomposing the error as:
\begin{equation}
    \eb_{\ub} := \norinf{\umap - \ub^{{EnKF}}_{\N}} \le \underbrace{\norinf{\overline{\del}}}_{=:\eb_\del} + \norinf {\A^T} \underbrace{\norinf{\Omega\Omega^T}\norinf{\overline{\Blambda}^* - \Blambda^*}}_{=:\eb_{\Blambda}} + \norinf {\A^T} \underbrace{\norinf{\Omega\Omega^T - \sC}}_{=:\eb_{\Omega}} \norinf{\Blambda^*}
    \label{decomp}
\end{equation}
Proposition \ref{eDelta}, Proposition \ref{eOmega}, and Proposition \ref{elam} provides non-asymptotic error bounds for  $\eb_\del$, $\eb_{\Omega}$, and $\eb_{\Blambda}$ respectively. The results in Proposition \ref{eDelta}, \ref{eOmega}, and \ref{elam} will be used in Theorem \ref{non_asymp} to finally derive the desired error bounds for $ \eb_{\ub}$.

\begin{proposition}[Estimation for $\eb_\del$]
\label{eDelta}

Let $\del^{(i)} \sim \GM{0}{\sC}$, $i = 1,\dots,N$. For any $\varepsilon_1 > 0$, there holds:
\[
\norinf{\overline{\del}} \le \varepsilon_1
\]
with probability at least $1 - \exp\LRp{-\frac{1}{4c_1\ \mathrm{dim}(C)}\LRp{\frac{\varepsilon_1}{\norinf{\sC^\half}}\sqrt{N}}^2} =: 1 - \E_1^{c_1}\LRp{\varepsilon_1,\sC}$,  where $c_1$ is an absolute positive constant, and $\mathrm{dim}(C)$ denotes the dimension of $C$.
\end{proposition}
 \begin{proof}
Let us define $\del^{(i)} = \sC^\half \taub^{(i)}$, where $\taub^{(i)} \sim \GM{0}{I}$, where $I$ is the identity matrix. Therefore, one has $\overline{\taub} = \frac{1}{N}\sum_{i=1}^N\taub^{(i)} \sim \GM{0}{\frac{I}{N}}$.

  \begin{enumerate}
      \item From \cite[Theorem 1]{gao2022tail} we have:
      \[
      \Pr{\nor{\taub}_{{\infty}} > \varepsilon }\leq \Pr{\nor{\taub}_{1} > \varepsilon }\le \exp\LRp{-\frac{\varepsilon^2}{4c_1\ \mathrm{dim}(C)}},
      \]
      where $c_1$ is an absolute positive constant.
        \label{na_1}
      \item Using result in item \eqref{na_1}, since $\overline{\taub}\sqrt{N}$ has identity covariance, we can conclude that:
    \[
      \Pr{\norinf{\overline{\taub}} > \frac{\varepsilon}{\sqrt{N}} }\le \exp\LRp{-\frac{\varepsilon^2}{4c_1\ \mathrm{dim}(C)}},
      \]
      \item Finally, by setting $\varepsilon_1 = \frac{\varepsilon}{\sqrt{N}} $ and noting that $\norinf{\overline{\del}} \leq  \norinf{\sC^\half} \norinf{\overline{\taub}}$ it is easy to see that:
      \[
      \Pr{\norinf{\overline{\del}} \le \varepsilon_1} \ge  \Pr{\norinf{\overline{\taub}} \le \frac{\varepsilon_1}{\norinf{\sC^\half}}}
      \ge 1 - \exp\LRp{-\frac{1}{4c_1\ \mathrm{dim}(C)}\LRp{\frac{\varepsilon_1}{\norinf{\sC^\half}}\sqrt{N}}^2}
      \]
thereby concluding the proof.
\end{enumerate}
\end{proof}

\begin{proposition}[Estimation for $\eb_\Omega$]
\label{eOmega}
Let $\del^{(i)} \sim \GM{0}{\sC}$, $i = 1,\dots,N$. For any $\varepsilon_2 > 0$, there holds:
\[
\norinf{\Omega\Omega^T - \sC} \le \varepsilon_2
\]
with probability at least $1  -2\exp\LRp{-c_2\frac{N}{\norinf{\sC}^2}\varepsilon^2_2} =: 1 - 2\E_2^{c_2}\LRp{\varepsilon_2,\sC}$,  where $c_2$ is an absolute positive constant. 
\end{proposition}

\begin{proof}

Note that:

  \begin{enumerate}
      \item From \cite[Theorem 4.6.1]{vershynin_2018} we have:
      \[
        \norinf{\Omega\Omega^T - \sC} \le \varepsilon \norinf{\sC}
      \]
      with probability at least $1 -2\exp\LRp{-c_2N\varepsilon^2}$
      where $c_2$ is an absolute positive constant.
      \item By  setting $\varepsilon_2 = \varepsilon \nor{\sC}$ we get the desired result thereby  concluding the proof.
  \end{enumerate}
\end{proof}

\begin{proposition}[Estimation for $\eb_{\Blambda}$]
\label{elam}
Let $\del^{(i)} \sim \GM{0}{\sC}$, and $\sigb^{(i)} \sim \GM{0}{\L}$, $i = 1,\dots,N$. For any $0< \varepsilon_4 <  \LRp{\frac{\eta\norinf{\L+\A\sC\A^T}}{\kappa\LRp{\B}\norinf{\A}^2\norinf{\sC}}}\LRp{\frac{\eta\norinf{\L+\A\sC\A^T}}{\kappa\LRp{\B}\norinf{\A}^2\norinf{\sC}}+\norinf{C}} $, there holds:
\[
 \norinf{\Omega\Omega^T} \norinf{\overline{\Blambda}^* - \Blambda^*} \le c_4\varepsilon_4
\]
with probability at least 
$1 -6\E_3^{c_2}\LRp{\varepsilon_4,C}- \E_4^{c_1}\LRp{\varepsilon_4,\sC} - \E_5^{c_1}\LRp{\varepsilon_4,\sC,\Sigma}$
%$1 - \exp\LRp{-c\LRp{\frac{\varepsilon_3}{\nor{\sC^\half}}\sqrt{N} - \sqrt{\n}}^2} - \exp\LRp{-c\LRp{\frac{\varepsilon_3}{\nor{\L^\half}}\sqrt{N} - \sqrt{\m}}^2}  -4\exp\LRp{-cN\varepsilon^2_3}$,  
where $c_4$ is a constant that depends only on $\sC, \bd, \A, \ub_0, \L$ given by:
\[ c_4=\frac{\kappa^2(B)\norinf{A\ub_0-\bd}}{(1-\eta)\norinf{B}}\LRp{\frac{\norinf{A}^2\norinf{C}}{\norinf{B}}+\frac{1+\norinf{A}}{\norinf{A\ub_0-\bd}}},\]
$\kappa(\B)$ is the condition number of $B=\LRp{\L+\A\sC\A^T}$ and we defined the following quantities:
\begin{equation}
    \begin{aligned}
       &\E_3^{c_2}\LRp{\varepsilon_4,\sC}= \exp\LRp{-c_2\frac{N}{\norinf{\sC}^2}\LRp{\frac{-\norinf{C}+\sqrt{\norinf{C}^2+4\epsilon_4}}{2}}^2},\\
      & \E_4^{c_1}\LRp{\varepsilon_4,\sC}= \exp\LRp{-\frac{N}{4c_1\ \mathrm{dim}(C)\ \norinf{\sC^\half}^2}\LRp{\frac{-\norinf{C}+\sqrt{\norinf{C}^2+4\epsilon_4}}{2}}^2}, \\
      & \E_5^{c_1}\LRp{\varepsilon_4,\sC,\Sigma}=\exp\LRp{-\frac{N}{4c_1\ \mathrm{dim}(\Sigma)\ \norinf{\Sigma^\half}^2}\LRp{\frac{-\norinf{C}+\sqrt{\norinf{C}^2+4\epsilon_4}}{2}}^2}.
    \end{aligned}
    \label{new_constants}
\end{equation}
\end{proposition}

\begin{proof}
  Recall that $\Blambda^*$  in \eqref{lambda_optimal} and $\overline{\Blambda}^*$ in \eqref{optimal_1} satisfy:
  \begin{eqnarray*}
       \underbrace{\LRp{\L+\A\sC\A^T}}_{\B}\Blambda^* &=\underbrace{\LRp{\A\ub_0-\bd}}_{\b} \\
       \LRp{\L+\A\Omega\Omega^T\A^T}\overline{\Blambda}^* &= \LRp{\A\LRp{\ub_0+\overline{\del}}-\LRp{\bd+\overline{\sigb}}}
  \end{eqnarray*}
  Thus, $\overline{\Blambda}^*$ is the solution of the perturbed equation:
  \[\LRp{\B + \Delta\B}\overline{\Blambda}^* = \b + \Delta\b,
  \]
  where $\Delta\B = \A\LRp{ \Omega\Omega^T-\sC}\A^T$ and $\Delta\b =  \A\overline{\del}-\overline{\sigb} $.  Now using Proposition \ref{eOmega} there holds:
  \begin{equation}
  \label{eDeltaB}
  \frac{\norinf{\Delta \B}}{\norinf{\B}} \le \varepsilon\frac{\norinf{\A}^2\norinf{\sC}}{\norinf{\L+\A\sC\A^T}}
  ,
  \end{equation}
  with probability at least $1 -2\E_2^{c_2}\LRp{\varepsilon,C}$.
Similarly using Proposition \ref{eDelta} and invoking the union bound there holds:
\begin{equation}
     \frac{\norinf{\Delta \b}}{\norinf{\b}} \le \frac{\norinf{\overline{\sigb}} + \norinf{\A}\norinf{\overline{\del}}}{\norinf{\A\ub_0-\bd}} \le \frac{1 + \norinf{\A}}{\norinf{\A\ub_0-\bd}}\varepsilon,
     \label{pre_b_cond}
\end{equation}
with probability at least 
$1 - \E_1^{c_1}\LRp{\varepsilon,\sC} - \E_1^{c_1}\LRp{\varepsilon,\L}$. Consider a factor $\eta<1$. Then, using \eqref{eDeltaB}, for $0< \varepsilon \leq \frac{\eta\norinf{\L+\A\sC\A^T}}{\kappa\LRp{\B}\norinf{\A}^2\norinf{\sC}} $, there holds:
\begin{equation}
    \norinf{\B^{-1}}\norinf{\Delta\B}= \kappa(B) \frac{\norinf{\Delta B}}{\norinf{B}}\leq \eta<1,
    \label{pre_cond}
\end{equation}
 with probability at least $1 -2\E_2^{c_2}\LRp{\varepsilon,C}$.
Using result  \eqref{pre_cond},  standard perturbation theory (equation 2.4 in \cite{demmel1997applied}) gives:  for $0< \varepsilon <  \frac{\eta\norinf{\L+\A\sC\A^T}}{\kappa\LRp{\B}\norinf{\A}^2\norinf{\sC}} $, there holds:
\begin{equation}
    {\norinf{\Blambda^* -\overline{\Blambda}^*}} \le \frac{\norinf{\Blambda^*}\kappa\LRp{\B}}{1-\kappa\LRp{\B}\frac{\norinf{\Delta \B}}{\norinf{\B}}}\LRp{\frac{\norinf{\Delta \B}}{\norinf{\B}} + \frac{\norinf{\Delta \b}}{\norinf{\b}}},
    \label{perturbation}
\end{equation}
with probability at least $1 -2\E_2^{c_2}\LRp{\varepsilon,C}$. Now note that using definition of $\Blambda^*$  in \eqref{lambda_optimal} we have:
\begin{equation}
    \norinf{\Blambda^*}\leq \norinf{B^{-1}}\norinf{A\ub_0-\bd}\leq \frac{\kappa(B)\norinf{A\ub_0-\bd}}{\norinf{B}}.
    \label{lambvda_star_bound}
\end{equation}
Using the above bound on $\norinf{\Blambda^*}$ and  using the results \eqref{eDeltaB}, and \eqref{pre_b_cond}  on \eqref{perturbation}, the union bound gives: for any $0< \varepsilon_3 <  \frac{\eta\norinf{\L+\A\sC\A^T}}{\kappa\LRp{\B}\norinf{\A}^2\norinf{\sC}} $, there holds:
%However, from (\ref{eDeltaB}) we have,  with probability at least
  %$1 -2\exp\LRp{-cN\varepsilon^2}$,
%  $1-2\E_2^{c_2}\LRp{\varepsilon,\ident}$,
  %\[
 % \norinf{\B^{-1}}\norinf{\Delta\B} \le 
%  \varepsilon\kappa\LRp{\B} \frac{\norinf{\A}^2\norinf{\sC}}{\norinf{\L+\A\sC\A^T}} < 1
 % \]
 % for $\varepsilon$ sufficiently small. 
  \begin{equation}
      \norinf{\overline{\Blambda}^* - \Blambda^*} \le c'\varepsilon_3,
      \label{lam_11}
  \end{equation}
with probability at least   $1 -4\E_2^{c_2}\LRp{\varepsilon_3,C}- \E_1^{c_1}\LRp{\varepsilon_3,\sC} - \E_1^{c_1}\LRp{\varepsilon_3,\L}$, where $c'$ is a constant that depends only on $\sC, \bd, \A, \ub_0, \L$ and given by:
\begin{equation}
    c'=\frac{\kappa^2(B)\norinf{A\ub_0-\bd}}{(1-\eta)\norinf{B}}\LRp{\frac{\norinf{A}^2\norinf{C}}{\norinf{B}}+\frac{1+\norinf{A}}{\norinf{A\ub_0-\bd}}}.
    \label{const_c}
\end{equation}
Now from Proposition \ref{eOmega}, we also have:
\begin{equation}
   \norinf{\Omega\Omega^T} \le \norinf{\sC}+ \norinf{\Omega\Omega^T - \sC} \le  \varepsilon_2+ \norinf{\sC},  
    \label{c_11}
\end{equation}
holds  with probability at least $1 -2\E_2^{c_2}\LRp{\varepsilon_2,\sC}$. Now intersecting the events \eqref{lam_11} and \eqref{c_11} we have:
\begin{equation}
     \norinf{\Omega\Omega^T} \norinf{\overline{\Blambda}^* - \Blambda^*}\le \LRp{\varepsilon_2+ \norinf{\sC}}c' \epsilon_3,
     \label{use}
\end{equation}
      with probability at least 
$1 -4\E_2^{c_2}\LRp{\varepsilon_3,C}-2\E_2^{c_2}\LRp{\varepsilon_2,C}- \E_1^{c_1}\LRp{\varepsilon_3,\sC} - \E_1^{c_1}\LRp{\varepsilon_3,\L}$. Now choose $\epsilon_2=\epsilon_3$ in \eqref{use}. Then, there holds:
\begin{equation}
    \norinf{\Omega\Omega^T} \norinf{\overline{\Blambda}^* - \Blambda^*}\le \LRp{\varepsilon_3+ \norinf{\sC}}c' \epsilon_3= c'\epsilon_4,
    \label{up_n}
\end{equation}
with probability at least 
$1 -6\E_2^{c_2}\LRp{\varepsilon_3,C}- \E_1^{c_1}\LRp{\varepsilon_3,\sC} - \E_1^{c_1}\LRp{\varepsilon_3,\L}$, and we defined $\epsilon_4$ in \eqref{up_n} that satisfies:
\begin{equation}
    \epsilon_3=\frac{-\norinf{C}+\sqrt{\norinf{C}^2+4\epsilon_4}}{2}.
    \label{ep_3}
\end{equation}
Further, the bound on $\epsilon_3$ from \eqref{lam_11} translates to the bound on $\epsilon_4$ as:
\[ 0<\epsilon_4<  \LRp{\frac{\eta\norinf{\L+\A\sC\A^T}}{\kappa\LRp{\B}\norinf{\A}^2\norinf{\sC}}}\LRp{\frac{\eta\norinf{\L+\A\sC\A^T}}{\kappa\LRp{\B}\norinf{\A}^2\norinf{\sC}}+\norinf{C}}.\]
Now substituting $\epsilon_3$ from \eqref{ep_3} in $\E_2^{c_2}\LRp{\varepsilon_3,C},\ \E_1^{c_1}\LRp{\varepsilon_3,\sC},\ \E_1^{c_1}\LRp{\varepsilon_3,\L}$  and defining new functions $\E_3^{c_2}\LRp{\varepsilon_4,\sC}$, $\E_4^{c_1}\LRp{\varepsilon_4,\sC}$, and, $\E_5^{c_1}\LRp{\varepsilon_4,\sC,\Sigma}$
as given in \eqref{new_constants}, we get the desired result. This concludes the proof.

\end{proof}

\begin{theorem}[Non-asymptotic error estimator for EnKF]
\label{non_asymp}
Let $\del^{(i)} \sim \GM{0}{\sC}$, and $\sigb^{(i)} \sim \GM{0}{\L}$, $i = 1,\dots,N$. For any $0< \varepsilon <  \LRp{\frac{\eta\norinf{\L+\A\sC\A^T}}{\kappa\LRp{\B}\norinf{\A}^2\norinf{\sC}}}\LRp{\frac{\eta\norinf{\L+\A\sC\A^T}}{\kappa\LRp{\B}\norinf{\A}^2\norinf{\sC}}+\norinf{C}} $, with $\eta<1$, there exists a constant $c$, independent of $\varepsilon$, such that
\begin{equation}
    \eb_{\ub} \le c\varepsilon
    \label{error_bound}
\end{equation}
holds with probability at least $1 -6\E_3^{c_2}\LRp{\varepsilon,C}- \E_4^{c_1}\LRp{\varepsilon,\sC} - \E_5^{c_1}\LRp{\varepsilon,\sC,\Sigma}- \E_1^{c_1}\LRp{\varepsilon,\sC}- 2\E_2^{c_2}\LRp{\varepsilon,\sC}$,
where $\eb_{\ub}$ is defined in \eqref{decomp}. Further, an explicit expression for the constant $c$ in \eqref{error_bound} is given by:
\begin{equation}
    c=1+\norinf{A^T}\LRp{\nu_1+\nu_2},
    \label{constant_c_non_asy}
\end{equation}
where
\[ \nu_1=\frac{\kappa^2(B)\norinf{A\ub_0-\bd}}{(1-\eta)\norinf{B}}\LRp{\frac{\norinf{A}^2\norinf{C}}{\norinf{B}}+\frac{1+\norinf{A}}{\norinf{A\ub_0-\bd}}},\]
\[ \nu_2= \frac{\kappa(B)\norinf{A\ub_0-\bd}}{\norinf{B}}.\]
  
\end{theorem}

\begin{proof}
    The proof follows from invoking union bound on (\ref{decomp}) and using the results of proposition \ref{eDelta}, proposition \ref{eOmega} and proposition \ref{elam}. Note that the constants $\nu_1$ and $\nu_2$ in \eqref{constant_c_non_asy} comes from second and third term in \eqref{decomp}.  For deriving the constant $\nu_2$ in \eqref{constant_c_non_asy} we used the bound in \eqref{lambvda_star_bound}.
\end{proof}

Note that Theorem \ref{non_asymp} requires computing $\kappa(B)$, and $ACA^T$ which can be computationally expensive in high-dimensional settings. In Corollary \ref{cor_easy}, we therefore provide a simpler non-asymptotic result for the special case $\Sigma=\mu I$, where $\mu>0$ and $I$ denotes the identity. This result involves only the computation of the quantities  $\norinf{C},\ \norinf{A},\norinf{A^T}$ and $\norinf{A\ub_0-\bd}$.
\begin{corollary}
\label{cor_easy}
    Consider Theorem \ref{non_asymp} with $\Sigma=\mu I$ (where $\mu>0$ and $I$ denotes identity). Then, for any $0< \varepsilon <  \LRp{\frac{\eta\ \mu}{\sqrt{m}\norinf{\A}^2\norinf{\sC}}}\LRp{\frac{\eta\ \mu}{\sqrt{m}\norinf{\A}^2\norinf{\sC}}+\norinf{C}} $, with $\eta<1$, the non-asymptotic result in Theorem \ref{non_asymp}  also holds with the new constant $c$ in \eqref{constant_c_non_asy} given by:
    \begin{equation}
    c=1+\norinf{A^T}\LRp{\nu_1+\nu_2},
    \label{constant_c_non_asy_new}
\end{equation}
where
\[ \nu_1=\frac{m\norinf{A\ub_0-\bd}\norinf{A}^2\norinf{C}}{\mu^2(1-\eta)}+   \frac{m\LRp{\mu+\norinf{A}\norinf{C}\norinf{A^T}}(1+\norinf{A})}{\mu^2(1-\eta)} ,\]
\[ \nu_2= \frac{\sqrt{m}\norinf{A\ub_0-\bd}}{\mu}.\]
\end{corollary}
\begin{proof}
Note that Theorem \ref{non_asymp} holds for $\epsilon$ in the following range:
\[ 0<\epsilon<\LRp{\frac{\eta\norinf{B}}{\kappa\LRp{\B}\norinf{\A}^2\norinf{\sC}}}\LRp{\frac{\eta\norinf{B}}{\kappa\LRp{\B}\norinf{\A}^2\norinf{\sC}}+\norinf{C}}=\epsilon_u.\]
Simplifying the above expression using $\kappa(B)=\norinf{B}\norinf{B^{-1}}$, we have:
\begin{equation}
    \begin{aligned}
\epsilon_u=& \LRp{\frac{\eta}{\norinf{B^{-1}}\norinf{\A}^2\norinf{\sC}}}\LRp{\frac{\eta}{\norinf{B^{-1}}\norinf{\A}^2\norinf{\sC}}+\norinf{C}}\\
\geq & \LRp{\frac{\eta\ \mu}{\sqrt{m}\norinf{\A}^2\norinf{\sC}}}\LRp{\frac{\eta\ \mu}{\sqrt{m}\norinf{\A}^2\norinf{\sC}}+\norinf{C}}, 
    \end{aligned}
\end{equation}
where we used the fact that $\norinf{B^{-1}}\leq \sqrt{m}\nor{B^{-1}}_2\leq \frac{\sqrt{m}}{\mu}$, leading to  $\frac{1}{\norinf{B^{-1}}}\geq \frac{\mu}{\sqrt{m}}$. Therefore, Theorem \ref{non_asymp} is also valid for the new range of $\epsilon$ given by:
\[ 0<\epsilon< \LRp{\frac{\eta\ \mu}{\sqrt{m}\norinf{\A}^2\norinf{\sC}}}\LRp{\frac{\eta\ \mu}{\sqrt{m}\norinf{\A}^2\norinf{\sC}}+\norinf{C}}.\]
Moreover, using the fact that $\kappa(B)=\norinf{B}\norinf{B^{-1}}$, $\nu_1$ in \eqref{constant_c_non_asy} can be bounded as follows:

\begin{equation}
    \begin{aligned}
         \nu_1=&\frac{\norinf{B^{-1}}^2\norinf{A\ub_0-\bd}\norinf{A}^2\norinf{C}}{(1-\eta)}+ \frac{\norinf{B}\norinf{B^{-1}}^2(1+\norinf{A})}{(1-\eta)}\\
         \leq & \frac{m\norinf{A\ub_0-\bd}\norinf{A}^2\norinf{C}}{\mu^2(1-\eta)}+   \frac{m\LRp{\mu+\norinf{A}\norinf{C}\norinf{A^T}}(1+\norinf{A})}{\mu^2(1-\eta)} 
    \end{aligned}
    \label{nu_1_b}
\end{equation}
where we again used  $\norinf{B^{-1}}\leq \frac{\sqrt{m}}{\mu}$, and $\norinf{B}\leq \mu+\norinf{A}\norinf{C}\norinf{A^T} $. Similarly, an upper bound for $\nu_2$ can be derived as:
\begin{equation}
\nu_2= \norinf{B^{-1}}\norinf{A\ub_0-\bd}\leq \frac{\sqrt{m}\norinf{A\ub_0-\bd}}{\mu}.
\label{n_2_b}
\end{equation}
Constructing a new constant $c$ in \eqref{constant_c_non_asy} using the bounds in \eqref{nu_1_b}, and \eqref{n_2_b} concludes the proof.
\end{proof}

\begin{remark}
\label{import_theo}
The importance of Theorem \ref{non_asymp}/ Corollary \ref{cor_easy} is that, given the constants $c_1,\ c_2$,  one can determine the number of samples $N$ required so that, for a given choice of $\epsilon$, the error bound $\eb_{\ub}\le c\varepsilon$ holds with a  prescribed high probability $p$. In particular, one need to choose the sample size $N$ that satisfies:
\[ p\leq 1 -6\E_3^{c_2}\LRp{\varepsilon,C}- \E_4^{c_1}\LRp{\varepsilon,\sC} - \E_5^{c_1}\LRp{\varepsilon,\sC,\Sigma}- \E_1^{c_1}\LRp{\varepsilon,\sC}- 2\E_2^{c_2}\LRp{\varepsilon,\sC}\]
For above to hold, it is sufficient to choose $N$ that satisfies:
\[6\E_3^{c_2}\LRp{\varepsilon,C} \leq  6\LRp{\frac{1-p}{11}},\  \E_4^{c_1}\LRp{\varepsilon,\sC}\leq  \LRp{\frac{1-p}{11}},\ \E_5^{c_1}\LRp{\varepsilon,\sC,\Sigma}\leq  \LRp{\frac{1-p}{11}}, \]

\[ \E_1^{c_1}\LRp{\varepsilon,\sC}\leq \LRp{\frac{1-p}{11}},\ \ 2\E_2^{c_2}\LRp{\varepsilon,\sC}\leq 2\LRp{\frac{1-p}{11}}.\]
Therefore, it is sufficient to choose $N$ that satisfies:
\begin{equation}
\begin{aligned}
    N\geq c_m \log {\LRp{\frac{11}{1-p}}}\max\ &\Biggl\{\frac{4\norinf{C}^2}{\LRp{-\norinf{C}+\sqrt{\norinf{C}^2+4\epsilon}}^2},\  \frac{16n\norinf{C^\half}^2}{\LRp{-\norinf{C}+\sqrt{\norinf{C}^2+4\epsilon}}^2},\\
    & \frac{16 m\norinf{\Sigma^\half}^2}{\LRp{-\norinf{C}+\sqrt{\norinf{C}^2+4\epsilon}}^2},\ \frac{4 n\norinf{C^\half}^2}{\epsilon^2},\ \frac{\norinf{C}^2}{\epsilon^2}\Biggl \}
\end{aligned}
    \label{sample_size_needed}
\end{equation}
where $c_{m}=\max\{\frac{1}{c_2},\ c_1\}$. Even when the constants $c_1,\ c_2$ are unknown, \eqref{sample_size_needed} can be used to understand how the choice of two prior covariances $\norinf{C_1}$ and $\norinf{C_2}$ affects the ratio of required sample sizes $\frac{N_1}{N_2}$, for achieving a given error bound $\eb_{\ub}\le \epsilon_1$. Here, $\epsilon_1$ is a fixed accuracy to be achieved and,  $N_1$, $N_2$ are the minimum sample size corresponding to $\norinf{C_1}$ and $\norinf{C_2}$ respectively, as computed from \eqref{sample_size_needed} (ignoring the constant $c_m$).
\end{remark}

\begin{remark}[Behavior of EnKF when $\mu$ is small]
    Of particular interest is the case when $\mu$ is very small in Corollary \ref{cor_easy}. The papers  \cite{agapiou2017importance,sanz2020bayesian,al2024non} highlight the need to increase
the sample size $N$ along small noise limits, i.e when $\mu$ is very small. This observation can be explained using Corollary \ref{cor_easy} as follows:

Consider the small noise regime where $\mu_1,\ \mu_2$ are both small  with $\mu_1<<\mu_2$ such that required number of samples $N$  is not dominated by the third term in \eqref{sample_size_needed}. Now setting $\epsilon_{req}= \LRp{\frac{\eta\ \mu_1}{\sqrt{m}\norinf{\A}^2\norinf{\sC}}}\LRp{\frac{\eta\ \mu_1}{\sqrt{m}\norinf{\A}^2\norinf{\sC}}+\norinf{C}} $,
if one requires $\eb_{\ub}\le \epsilon_{req}$, then, according to \eqref{error_bound}, one must choose:
\[ \epsilon=\frac{\epsilon_{req}}{c_1}\ \mathrm{for \ \mu=\mu_1}, \ \ \ \ \ \ \  \epsilon=\frac{\epsilon_{req}}{c_2}\ \mathrm{for \ \mu=\mu_2},\]
 where the constants $c_1,\ c_2$ are computed using \eqref{constant_c_non_asy_new} for $\mu_1$ and $\mu_2$ respectively. Since $c_1 >c_2$ (as $\mu_1<\mu_2)$, it follows that $\epsilon$ must be choosen significantly smaller for $\mu_1$ than for $\mu_2$. As a consequence, the sample size bound in \eqref{sample_size_needed} requires $N$  to be substantially larger for $\mu_1$ compared to $\mu_2$. Our work here provides a rigorous justification for the observation made in \cite{agapiou2017importance,sanz2020bayesian,al2024non}.
\end{remark}
\begin{remark}[Related work on non-asymptotic analysis]
Our $\infty$-norm based analysis provides a sufficient condition on the number of samples  required when comparing two prior covariance matrices $\sC_1$ and $\sC_2$ (see Remark \ref{import_theo}). However, as shown in \cite{al2024non}, the sample size requirements for EnKF updates are more directly governed by the effective dimension, which is related to the spectral properties of the underlying covariance operators. Their non-asymptotic analysis demonstrates that sample complexity scales with the effective dimension $d_{\mathrm{eff}}$ rather than the ambient dimension $n$.

Our work  complements their analysis by providing an alternative perspective via $\infty-$norm concentration inequalities. 
We show that novel non-asymptotic bounds such as   Theorem \ref{non_asymp} can be derived for EnKF via the duality viewpoint developed in Sections \ref{dual_1} and \ref{randomized_sec}. Future work will explore leveraging this duality framework to derive additional non-asymptotic bounds that depend on the effective dimension $d_{\mathrm{eff}}$ rather than $\infty$-norms.
\end{remark}

\section{From Ensemble Kalman Filter (EnKF) to Ensemble Kalman Inversion (EnKI)}
\label{from_enkf_enki}
 The Ensemble Kalman Inversion (EnKI)  is an iterative version of the update equation \eqref{each_ensemble_kalman_inversion}  where the prior mean $\bu_0$ and the sample covariance matrix ($\Omega\Omega^T$ in \eqref{each_ensemble_kalman_inversion}) are updated in each iteration using the current estimates $\{ {{\ub}^{(i)}}^*\}_{i=1}^N$  \cite{iglesias2013ensemble, schillings2017analysis}. In particular, given an initial set of ensembles  $\{\ub_0^{(i)}\}_{j=1}^N$ (prior measurement samples of the parameter field or generated from some assumed prior distribution $ \GM{\ub_0}{\sC}$) let us first define the sample covariance matrix $\sC(\ub_0)$ and the sample mean $ \bar{\ub}_0$ as follows:
\begin{equation}
\begin{aligned}
    & \sC(\ub_0)=\frac{1}{N}\sum_{i=1}^N\LRp{\ub_0^{(i)}-\bar{\ub}_0}\LRp{\ub_0^{(i)}-\bar{\ub}_0}^T,\\
    &  \bar{\ub}_0=\frac{1}{N}\sum_{i=1}^N\ub_0^{(i)}.
\end{aligned}    
\label{sample_co}
\end{equation}
Now by replacing $\ub_0+ \del^{(i)}$ with $\ub_0^{(i)}$, replacing the measurement $\db+\sigb^{(i)}$ with $\db^{(i)}$  and using $\sC(\ub_0)$ instead of $\Omega\Omega^T$  in \eqref{each_ensemble_kalman_inversion},  the update equation   for each sample is written as:
\begin{equation}
\label{EnKF_new}
 \ub_1^{(i)}= \ub_0^{(i)} +\sC(\ub_0)\A^T  \Blambda^{(i)}(\ub_0),
\end{equation}
where,  $\Blambda^{(i)}(\ub_0)$  is given by:
\begin{equation}
    \Blambda^{(i)}(\ub_0) :=  \LRp{\L+\A\sC(\ub_0)\A^T}^{-1}\LRp{\db^{(i)} -\A\ub_0^{(i)}}.
     \label{considered_case}
\end{equation}
Note that based on the new estimates  $\{\ub_1^{(i)}\}_{j=1}^N$ from \eqref{EnKF_new}, one may compute a new sample covariance $\sC_1$ and again update the current solution $\{\ub_1^{(i)}\}_{i=1}^N$  to obtain the new estimates $\{\ub_2^{(i)}\}_{i=1}^N$ via \eqref{EnKF_new}.   More generally, if the scheme \eqref{EnKF_new} is continued in an iterative fashion, we arrive at the Ensemble Kalman Inversion (EnKI) as follows:
\begin{equation}
\begin{aligned}
 \ub_{k+1}^{(i)} &= \ub_k^{(i)} +\sC(\ub_k)\A^T  \Blambda^{(i)}(\ub_k),\\
 &=\ub_k^{(i)} +\sC(\ub_k)\A^T  \LRp{\L+\A\sC(\ub_k)\A^T}^{-1}\LRp{{\db} -\A\ub_k^{(i)}},
\end{aligned}
 \label{iterative_version}
\end{equation}
where we assumed ${\db}=\A \ub^{\dagger}+\sigb$ ($\ub^{\dagger}$ is the true solution and same measurement noise  is used for each ensemble), and the matrix $\sC(\ub_k)$  is computed based on the current estimates $\{\ub_k^{(i)}\}_{i=1}^N$. Now if one introduces an artificial step size $h$ and defines a new noise covariance matrix  $\L_h=\frac{\L} {h}$, the scheme \eqref{iterative_version} using the new covariance $\L_h$ can be modified as:
\begin{equation}
\begin{aligned}
    \ub_{k+1}^{(i)} &= \ub_k^{(i)} +\sC(\ub_k)\A^T  \LRp{\L_h+\A\sC(\ub_k)\A^T}^{-1}\LRp{\db -\A\ub_k^{(i)}}\\
    &= \ub_k^{(i)} +h\sC(\ub_k)\A^T  \LRp{\L+h\A\sC(\ub_k)\A^T}^{-1}\LRp{\db -\A\ub_k^{(i)}}
\end{aligned}  
\label{new_ss}
\end{equation}
It is clear that if we take $h\rightarrow 0$, \eqref{new_ss} is the Euler  type discretization of the following set of coupled ODEs \cite{schillings2017analysis}:
\begin{equation}
    \frac{d \ub^{(i)}}{dt}=-\sC(\ub)\nabla_{\ub} \Phi\LRp{\ub^{(i)},\ \db},\quad i=1,\dots,  N,
    \label{ode_sc}
\end{equation}
where $\ub=[\ub^{(1)},\ \ub^{(2)}, \dots \ub^{(N)}]$ and
\begin{equation}
    \Phi\LRp{\ub^{(i)},\ \db}=\frac{1}{2}\nor{\L^{-1/2}\LRp{\db-\A \ub^{(i)}}}_2^2,\quad \sC(\ub)=\frac{1}{N}\sum_{j=1}^N\LRp{\ub^{(i)}-\bar{\ub}}\LRp{\ub^{(i)}-\bar{\ub}}^T.
    \label{Enk_loss}
\end{equation}
Thus each particle performs a preconditioned gradient descent for $\Phi\LRp{.,\ \db}$ and all the
individual gradient descents are coupled through the preconditioning of the flow by
the empirical covariance $\sC(\ub)$ \cite{schillings2017analysis}. 

Schillings et al. \cite{schillings2017analysis} analyzed the  long-time behavior  of \eqref{ode_sc} and proved the key properties  for the EnKI algorithm in continuous-time limit(considering linear setting and noise free scenario) such as: i) \eqref{ode_sc} has a unique  solution $\ub^{(i)}(.)\in C([0,\infty); \sD)$ where $\sD$ represents the finite dimensional subspace spanned by the initial ensembles, $\{\ub^{(1)}(0),\ \ub^{(2)}(0),\ \ub^{(3)}(0)\dots \ub^{(N)}(0)\}$;  ii) collapse of all ensemble members towards their mean value at an algebraic rate. Further, the solution $\ub^i(t)$ to \eqref{ode_sc} at any time always lies in $\sD$. This is referred to as the ``subspace property" of the EnKI algorithm. Since EnKI performs a preconditioned gradient descent for the loss function $\Phi\LRp{.,\ \db}$ in \eqref{Enk_loss} and the solution always lies in the subspace $\sD$, previous research \cite{iglesias2013ensemble} numerically showed that EnKI approximately solves the following least-squares minimization problem: 
\begin{equation}
    \underset{\ub \in \sD}{\min}\ \J_{EnKI}(\ub)=\min_{\ub\in \sD}\LRp{\frac{1}{2}\nor{\L^{-1/2}\LRp{\db-\A \ub}}_2^2}.
    \label{enki_loss_approx}
\end{equation}
We immediately see that the ``subspace property" of the EnKI algorithm plays the role of an implicit regularization in minimizing the above loss function, ensuring that the solution always lies in the subspace $\sD$.

The literature on the theoretical analysis of the continuous-time limit of the EnKI algorithm is extensive \cite{schillings2017analysis,blomker2022continuous,ding2021ensemble}, and numerous variants of the EnKI algorithm exist. We discuss here some of the recent developments in the field. Ding et al. \cite{ding2021ensemble}  analyzed the continuous time-limit version of
EnKI and  proved the mean field limit (as $N\rightarrow \infty$) of the resulting system of stochastic differential equation. To prevent ensemble collapse and break the subspace property, popular techniques commonly adopted in the Kalman filter community such as additive inflation, localization has been applied to EnKI \cite{chada2019convergence,tong2023localized}. However, one of the issues with breaking the ``subspace property" is that this destroys the implicit regularization imposed while minimizing \eqref{enki_loss_approx}. To tackle this issue, Chada et al. \cite{chada2019convergence,chada2020tikhonov} considered a Tikhonov regularization within the EnKI framework.

\begin{remark}
While the variants of EnKI discussed above offer various advantages, the vanilla EnKI \eqref{new_ss} remains widely used, particularly in cases where limited prior knowledge is available, such as when only a few prior measurements of the parameter field $\ub$ are accessible. In such situations, one can run the iterative scheme \eqref{new_ss} using these limited measurement samples as the initial ensembles. Recently, Harris et al. \cite{harris2025accuracy} proposed an accuracy improvement strategy for the vanilla EnKI \eqref{new_ss} through data-informed ensemble selection.

It is interesting to note that most theoretical work on EnKI relies on the continuous time-limit analysis of \eqref{new_ss} \cite{schillings2017analysis,ding2021ensemble}. In this study, we approach the EnKI algorithm from the perspective of fixed-point iteration.  
In particular, we propose a perturbed fixed-point iteration to accelerate the convergence of the EnKI algorithm \eqref{new_ss}. We will invoke  the duality perspective introduced in sections \ref{dual_1} and \ref{randomized_sec} to derive the perturbed fixed-point iteration.
\end{remark}
For deriving a convergence improvement strategy for EnKI, let us first rewrite \eqref{new_ss} as the following iterative scheme:
\begin{equation}
  \ub_{k+1}=g(\ub_k),\quad k=0,\dots \infty,
  \label{first_fix}
\end{equation}
where the function $g(\ub_k)$ is given as:
\begin{equation}
g(\ub_k) = \ub_k + 
\begin{bmatrix}
\sC(\ub_k)\A^T \\
\sC(\ub_k)\A^T \\
\vdots \\
\sC(\ub_k)\A^T 
\end{bmatrix}
\odot
\begin{bmatrix}
\Blambda^{(1)}(\bu_k) \\
\Blambda^{(2)}(\bu_k) \\
\vdots \\
\Blambda^{(N)}(\bu_k)
\end{bmatrix},
\label{fixed_point_equation}
\end{equation}
where $\odot$ represents the element-wise multiplication of the two tuples, %and 
 $\ub_k=[\ub^{(1)}_k,\ \ub^{(2)}_k, \dots \ub^{(N)}_k]$ and
    \begin{equation}
        \Blambda^{(i)}(\ub_k)=\LRp{\L_h+\A\sC(\ub_k)\A^T}^{-1}\LRp{\db -\A\ub^{(i)}_k},
        \label{iterative_lambda}
    \end{equation}
    \[\sC(\bu_k)=\frac{1}{N}\sum_{i=1}^N\LRp{\ub_k^{(i)}-\bar{\ub}_k}\LRp{\ub_k^{(i)}-\bar{\ub}_k}^T,\]
In the below discussions, we will assume that $\L_h=\mu I$, where $I$ is the identity matrix and $\mu>0$.

\subsection{EnKI-MC (I): A multiplicative covariance correction strategy for convergence acceleration of EnKI} 
\label{conv_imp_1}
In order to speed up the convergence of EnKI algorithm and improve the quality of inverse solution, we consider a multiplicative correction to the sample covariance matrix $\sC(\bu_k)$ as follows:
\begin{equation}
     \sC_m(\bu_k,\ k)=\alpha(\bu_k,\ k)\sC(\bu_k),
     \label{mul_cov}
\end{equation}
where  $\alpha(\bu_k,\ k)\in \R$ and $k\in \mathbb{N}$. Previously, a multiplicative covariance inflation in the context of the Ensemble Kalman Filter for dynamical systems was considered by Anderson et al. \cite{anderson2007adaptive}. In our case, this leads to the new iterative scheme: 
\begin{equation}
  \ub_{k+1}=g_m^{(I)}(\bu_k,\ \alpha(\bu_k,\ k)) = \begin{bmatrix}
\bu_k^{(1)} \\
\bu_k^{(2)} \\
\vdots \\
\bu_k^{(N)}
\end{bmatrix} + 
\begin{bmatrix}
\alpha(\bu_k,\ k)\sC(\ub_k)\A^T \\
\alpha(\bu_k,\ k)\sC(\ub_k)\A^T \\
\vdots \\
\alpha(\bu_k,\ k)\sC(\ub_k)\A^T 
\end{bmatrix}
\odot
\begin{bmatrix}
\Blambda^{(1)}(\bu_k,\ \alpha(\bu_k,\ k)) \\
\Blambda^{(2)}(\bu_k,\ \alpha(\bu_k,\ k)) \\
\vdots \\
\Blambda^{(N)}(\bu_k,\ \alpha(\bu_k,\ k))
\end{bmatrix}
\label{iter_new},\quad k=0,\dots \infty.
\end{equation}
where the function $\Blambda^{(i)}(\bu_k,\ \alpha(\bu_k,\ k))$ is obtained by replacing $\sC(\bu_k)$ in \eqref{iterative_lambda} with $\alpha(\bu_k,\ k)\sC(\bu_k)$. We call \eqref{iter_new} as EnKI-MC (I) (EnKI with multiplicative covariance correction). Looking carefully at \eqref{new_ss}, it is clear that using $\alpha(\bu_k,\ k)>1$ is equivalent to using a larger step-size of $h\alpha(\bu_k,\ k)$ in  discretizing \eqref{ode_sc}. This is expected to improve the speed of convergence provided that the iterative scheme \eqref{iter_new} is numerically stable.  

\begin{remark}
Such adaptive correction/step-size has been considered earlier by Chada et al. \cite{chada2019convergence} and  Iglesias et al. \cite{iglesias2021adaptive}. While Chada et al. \cite{chada2019convergence} considered a monotonically increasing stepsize  for convergence acceleration of EnKI, Iglesias et al. \cite{iglesias2021adaptive} derived adaptive stepsize for a different purpose-using EnKI within the Bayesian setting for parameter identification for approximating the target posterior distribution. The constraint that the step-size needs to be less than $1$ in Iglesias et al. \cite{iglesias2021adaptive}  leads to slow convergence of the algorithm.
In our numerical results, we provide comparison of our approach with the one by Chada et al. \cite{chada2019convergence} since our main objective is convergence acceleration of EnKI.
\end{remark}

In this work we will derive a new expression for the correction factor $\alpha(\bu_k,\ k)$ based on the duality perspective introduced in \eqref{randomized_sec}. Before we derive an expression for $\alpha(\bu_k,\ k)$, we will examine the convergence behavior of \eqref{iter_new} in Theorem \ref{mult_conv}. In particular, we look at the conditions on $\alpha(\bu_k,\ k)$  such that 
the perturbed iterative scheme \eqref{iter_new} converges.

\begin{theorem}[Convergence of EnKI-MC (I)]

  \label{mult_conv}
  
  Let $\sD \subset \R^n$ represents the finite dimensional subspace spanned by the initial ensembles, $\{\ub_0^{(1)},\ \ub_0^{(2)},\ \ub_0^{(3)}\dots \ub_0^{(N)}\}$. Assume the following:
  \begin{enumerate}
  \item  There exists a stepsize $h$ in \eqref{new_ss} such that the iterative scheme \eqref{first_fix} converges for any $\ub_0\in \sD$ and all the ensembles collapses. \footnote{Convergence of \eqref{new_ss} as $h\rightarrow 0$ under specific assumptions has been studied in Schillings et al. \cite{schillings2017analysis}}.
      \label{0}
   \item  $\alpha(\bu_k,\ k)$ in \eqref{iter_new} satisfies $\lim_{k\rightarrow \infty}\alpha(\bu_k,\ k)=1$ and $\alpha(\ub_k,\ k)\geq 1$ and bounded $\forall k$.
    \label{1}
   \item Assume that the sequence $\{ g_m^{(I)}(\bu_k,\ \alpha(\bu_k,\ k)), \}_{k=0}^{\infty}$ is bounded.
    \label{2}
      \item  $\exists k_0$, such that $\forall k\geq k_0$ 
\[ \nor{ g_m^{(I)}(\bu_k, 1) - g_m^{(I)}(\bv_k, 1) }_2  \leq L \nor{\bu_k-\bv_k}_2 ,\quad \mathrm{where},\quad L<1.\]
where $\bv_{k+1} = g_m^{(I)}(\bv_k, 1), \quad \text{for } k \geq k_0, \quad \bv_{k_0}=\bu_{k_0}$, and $\bu_{k_0}$ is the solution of \eqref{iter_new} at iteration $k_0$.
\label{4}
  \end{enumerate}
 Then for the iterative scheme \eqref{iter_new} the following holds true:
\begin{enumerate}
 \item  The solution $ \ub_{k}^{(i)}\in \sD$ for any $k$.
\item  The iterative scheme (\ref{iter_new}) converges to a fixed point of \eqref{first_fix}. That is, the converged solution $\ub^*$ satisfies $\ub^*=g(\ub^*)$.
      \item  For the converged solution $\ub^*$, all the ensembles   are the same (collapses), i.e $(\ub^*)^{(1)}=(\ub^*)^{(1)}=\dots (\ub^*)^{(N)}$.
\end{enumerate}

  \end{theorem}
  \begin{proof}
First we will show that for EnKI-MC, the solution $ \ub_{k}^{(i)}\in \sD$ for any $k$. 
Consider the case $k=0$, then \eqref{iter_new} can be rewritten as:
\begin{equation}
\begin{aligned}
\begin{bmatrix}
\ub^{(1)}_{1} \\
\ub^{(2)}_{1} \\
\vdots \\
\ub^{(N)}_{1}
\end{bmatrix} =\ 
&\begin{bmatrix}
\ub^{(1)}_0 \\
\ub^{(2)}_0 \\
\vdots \\
\ub^{(N)}_0
\end{bmatrix}
+
\begin{bmatrix}
\frac{1}{N} \sum\limits_{i=1}^N \alpha(\bu_0,0) \langle  \ub^{(i)}_0-\bar{\ub}_0,\, \A^T \Blambda^{(1)}_m(\ub_0)  \rangle  \LRp{\ub^{(i)}_0-\bar{\ub}_0} \\
\frac{1}{N} \sum\limits_{i=1}^N \alpha(\bu_0,0) \langle  \ub^{(i)}_0-\bar{\ub}_0,\, \A^T \Blambda^{(2)}_m(\ub_0)  \rangle  \LRp{\ub^{(i)}_0-\bar{\ub}_0} \\
\vdots \\
\frac{1}{N} \sum\limits_{i=1}^N \alpha(\bu_0,0) \langle  \ub^{(i)}_0-\bar{\ub}_0,\, \A^T \Blambda^{(N)}_m(\ub_0)  \rangle  \LRp{\ub^{(i)}_0-\bar{\ub}_0}
\end{bmatrix},
\end{aligned}
\label{fixed_point_equation_new}
\end{equation}
where $\langle .,\ ,\rangle$ denotes the standard inner product in $\R^n$, $\Blambda^{(i)}_m(\ub_k)=\LRp{\L+\A\sC_m(\ub_k)\A^T}^{-1}\LRp{\db -\A\ub^{(i)}_k}$, and $\sC_m(\ub_k)$ is given in \eqref{mul_cov}. Now from  (\ref{fixed_point_equation_new}) it is clear that the components $\ub^{(1)}_1,\  \ub^{(1)}_2\dots  \ub^{(1)}_N \in \sD$ since each of can be written as  a linear combination of $\ub^{(1)}_0,\ \ub^{(2)}_0,\ \dots \ub^{(N)}_0$. Following similar reasoning for subsequent iteration it is clear that the solution $ \ub_{k}^{(i)}\in \sD$ for any $k$.

Now let us prove that the iterative scheme (\ref{iter_new}) converges. Let $K=\overline{\{\ub_k\}_{k=0}^\infty}$ denote the compact set, where $\overline{\LRp{.}}$ denotes the closure (set is also bounded by assumption \ref{2}). Since $g_m^{(I)}(.,.)$ is continuous with respect to both arguments we have:
\[\lim_{k\rightarrow \infty}g_m^{(I)}(\bu, \alpha(\bu_k,\ k))=g_m^{(I)}(\bu,\ 1),\quad \forall \bu \in K.  \]
where we used the assumption $\lim_{k\rightarrow \infty}\alpha(\bu_k,\ k)=1$. From the $\epsilon-\delta$ definition for the convergence of sequence, we have
for a given $\ub\in K$ and $\epsilon_1>0$, $\exists k_0(\ub)\in \mathbb{N}$ such that $\forall k \geq k_0(\bu)$,
\begin{equation}
 \nor{g_m^{(I)}(\bu, \alpha(\bu_k,\ k)) - g_m^{(I)}(\bu, 1)}_2<\epsilon_1.
\label{epsilon_delta}
\end{equation}
First we will show that the above convergence is uniform, i.e $\exists k_0(\ub)=k_0$ (independent of $\ub$) such that \eqref{epsilon_delta} holds $\forall \ub\in K$. For this it suffice to show the following:
\[ \lim_{k\rightarrow}\sup_{\bu \in K}\nor{g_m^{(I)}(\bu, \alpha(\bu_k,\ k)) - g_m^{(I)}(\bu, 1)}_2=0.\]
From \eqref{first_fix} and \eqref{iter_new} we have:
\begin{equation}
\begin{aligned}
\nor{g_m^{(I)}(\bu, \alpha(\bu_k,\ k)) - g_m^{(I)}(\bu, 1)}_2^2&=\nor{\begin{bmatrix}
\sC(\ub)\A^T \\
\sC(\ub)\A^T \\
\vdots \\
\sC(\ub)\A^T 
\end{bmatrix} \odot \begin{bmatrix}
\alpha(\bu_k,\ k)\Blambda^{(1)}(\bu,\ \alpha(\bu_k,\ k))-\Blambda^{(1)}(\bu,\ 1) \\
\alpha(\bu_k,\ k)\Blambda^{(2)}(\bu,\ \alpha(\bu_k,\ k))-\Blambda^{(2)}(\bu,\ 1) \\
\vdots \\
\alpha(\bu_k,\ k)\Blambda^{(N)}(\bu,\ \alpha(\bu_k,\ k))-\Blambda^{(N)}(\bu,\ 1)
\end{bmatrix}}_2^2    \\
&=\sum_{i=1}^N \nor{\sC(\ub)\A^T\LRp{\alpha(\bu_k,\ k)\Blambda^{(i)}(\bu,\ \alpha(\bu_k,\ k))-\Blambda^{(i)}(\bu,\ 1)}}_2^2\\
& \leq  \sum_{i=1}^N \norm{\sC(\ub)\A^T}_2^2\norm{\alpha(\bu_k,\ k)\Blambda^{(i)}(\bu,\ \alpha(\bu_k,\ k))-\Blambda^{(i)}(\bu,\ 1)}_2^2.
\end{aligned}    
\label{simplify}
\end{equation}
Using definition of $\Blambda^{(i)}$ in \eqref{iterative_lambda}, we have:
\begin{align}
\label{eq:step1}
& \norm{\alpha(\bu_k,\ k)\Blambda^{(i)}(\bu,\ \alpha(\bu_k,\ k))-\Blambda^{(i)}(\bu,\ 1)}_2^2\leq\\
\label{eq:step2}
& \norm{\LRp{\frac{\L_h}{\alpha(\bu_k,\ k)}+\A\sC(\ub)\A^T}^{-1}-\LRp{\L_h+\A\sC(\ub)\A^T}^{-1}}_2^2\norm{\underbrace{\LRp{{\db} -\A\ub^{(i)}}}_{{\bf{r}}^{(i)}(\ub)}}_2^2,\\
\label{eq:step3}
&\leq \LRp{\frac{\alpha(\bu_k,\ k)-1}{\alpha(\bu_k,\ k)}}^2\norm{\LRp{\frac{\L_h}{\alpha(\bu_k,\ k)}+\A\sC(\ub)\A^T}^{-1}}_2^2\norm{\L_h}_2^2\norm{\LRp{\L_h+\A\sC(\ub)\A^T}^{-1}}_2^2\norm{{\bf{r}}^{(i)}(\ub)}_2^2\\
& \leq \LRp{\frac{\alpha(\bu_k,\ k)-1}{\alpha(\bu_k,\ k)}}^2\frac{1}{\mu^2}\norm{{\bf{r}}^{(i)}(\ub)}_2^2
\end{align}    
where we used the matrix identity $X^{-1}-Y^{-1}=X^{-1}\LRp{Y-X}Y^{-1}$ in going from \eqref{eq:step2} to \eqref{eq:step3}. We also used $\L_h=\mu I$ and $\norm{\LRp{\L_h+\A\sC(\ub)\A^T}^{-1}}_2^2,\ \norm{\LRp{\frac{\L_h}{\alpha(\bu_k,\ k)}+\A\sC(\ub)\A^T}^{-1}}_2^2\leq \frac{1}{\mu^2}$. Using the above result in \eqref{simplify} we have:
\[\sup_{\ub \in K}\nor{g_m^{(I)}(\bu, \alpha(\bu_k,\ k)) - g_m^{(I)}(\bu, 1)}_2^2\leq   \sum_{i=1}^N \frac{1}{\mu^2}\LRp{\frac{\alpha(\bu_k,\ k)-1}{\alpha(\bu_k,\ k)}}^2\sup_{\ub\in K}\LRp{\norm{\sC(\ub)\A^T}_2^2\norm{{\bf{r}}^{(i)}(\ub)}_2^2},\]
where the supremum on the right hand side of the inequality is attained due to Weierstrass extreme value theorem. Now considering the limit $k\rightarrow \infty$
 for the right hand side of the inequality and using assumption \ref{1}, it is clear that:
 \[ \lim_{k\rightarrow}\sup_{\bu \in K}\nor{g_m^{(I)}(\bu, \alpha(\bu_k,\ k)) - g_m^{(I)}(\bu, 1)}_2=0.\]
Therefore, the convergence in \eqref{epsilon_delta} is uniform, i.e $\exists k_0(\ub)=k_0$ (independent of $\ub$) such that \eqref{epsilon_delta} holds $\forall \ub\in K$. 

Now consider the following iterative scheme with the starting point $\bv_{k_0}=\bu_{k_0}$:
\begin{equation}
\bv_{k+1} = g_m^{(I)}(\bv_k, 1), \quad \text{for } k \geq k_0, \quad \bv_{k_0}=\bu_{k_0},
\label{unperturbed_scheme}
\end{equation}
Since by assumption \ref{0}, \eqref{unperturbed_scheme} converges for the starting point $\bv_{k_0}=\bu_{k_0}$, $\forall \epsilon_2>0$,
 $\exists k_1 \geq k_0$ such that
 \begin{equation}
\nor{\bv_k - \bv_\infty }_2 < \epsilon_2 \quad \forall k \geq k_1,
\label{v_conv}
 \end{equation}
 where $\bv_\infty$ denotes the limit of the sequence. 
Now let us compute the error $e_k := \| \mathbf{u}_k - \mathbf{v}_k \|$ for $k \geq k_0$ as follows:
\begin{equation}
\begin{aligned}
e_{k+1} &=\nor{ \bu_{k+1} - \bv_{k+1} }_2 \\
&= \nor{ g_m^{(I)}(\bu_k, \alpha(\bu_k,\ k)) - g_m^{(I)}(\bv_k, 1) }_2 \\
&= \nor{ g_m^{(I)}(\bu_k, \alpha(\bu_k,\ k))- g_m^{(I)}(\bu_k, 1)+ g_m^{(I)}(\bu_k, 1) - g_m^{(I)}(\bv_k, 1) }_2 \\
&\leq \nor{ g_m^{(I)}(\bu_k, \alpha(\bu_k,\ k)) - g_m^{(I)}(\bu_k, 1) }_2+ \nor{ g_m^{(I)}(\bu_k, 1) - g_m^{(I)}(\bv_k, 1) }_2\\
&\leq  \epsilon_1 + L \cdot e_k,
\end{aligned}
\label{recurrence}
\end{equation}
where we used \eqref{epsilon_delta}, and used the assumption \ref{4}, i.e. $\exists k_0$, such that $\forall k\geq k_0$  we have:
\[\nor{ g_m^{(I)}(\bu_k, 1) - g_m^{(I)}(\bv_k, 1) }_2  \leq L \nor{\bu_k-\bv_k}_2 ,\quad \mathrm{where},\quad L<1.\]
Solving the above recurrence \eqref{recurrence} with $e_{k_0} = 0$, we get:
\begin{equation}
e_k \leq \epsilon_1 \sum_{j=0}^{k - k_0 - 1} L^j,\quad \forall k\geq k_0+1.
\label{recurrence_sol}
\end{equation}
 Therefore, from \eqref{recurrence_sol} and \eqref{v_conv} we have $\forall k\geq k_1+1$:
 \begin{equation}
     \begin{aligned}
\nor{ \bu_k - \bv_\infty }_2 &\leq \nor{ \bu_k - \bv_k }_2 + \nor{ \bv_k - \bv_\infty }_2 \leq \epsilon_1 \sum_{j=0}^{k - k_0 - 1} L^j+\epsilon_2\\
& \leq \epsilon_1\frac{1-L^{k-k_0}}{1-L}+\epsilon_2\leq  \frac{\epsilon_1}{1-L}+\epsilon_2 \leq \epsilon,
     \end{aligned}
     \label{conv_2}
 \end{equation}
where we defined $\epsilon=\frac{\epsilon_1}{1-L}+\epsilon_2 $. From \eqref{conv_2} it is clear that the sequence $\{\ub_k\}_{k=0}^\infty$ converges to $\bv_\infty$. In addition since $\bv_\infty^1=\bv_\infty^2=\dots \bv_\infty^N$ (ensembles collapses by 
assumption \ref{0}), we also have $\bu_\infty^1=\bu_\infty^2=\dots \bu_\infty^N$. That is, all the ensembles collapses for the perturbed scheme as well on convergence. 
Finally, note that \eqref{iter_new} can be rewritten as:
\begin{equation}
    \bu_{k+1}=g(\bu_k)+\LRp{ g_m^{(I)}(\bu_k,\ \alpha(\bu_k,\ k))-g(\bu_k)},\quad k=0,\dots \infty,
    \label{new_s}
\end{equation}
where $g(\bu_k)$ corresponds to the unperturbed scheme \eqref{first_fix}. Taking the limit $k\rightarrow \infty$ we have:
\[ \bu^*=g(\bu^*)+\LRp{ g_m^{(I)}(\bu^*,\ 1)-g(\bu^*)},\]
where we have used the fact that $\alpha(\bu_k,\ k)$, $g_m^{(I)}(\bu_k,\ \alpha(\bu_k,\ k))$, and $g(\bu_k)$ are  continuous functions and also used  $\lim_{k\rightarrow \infty}\alpha(\bu_k,\ k)=1$. Now noting that $g_m^{(I)}(\bu^*,\ 1)=g(\bu^*)$ we have:
\[ \bu^*=g(\bu^*).\]
That is,  $\bu^*$ is a fixed point of $g$ in \eqref{first_fix}. Therefore, the iterative scheme (\ref{iter_new}) converges to a fixed point of \eqref{first_fix}, and this concludes the proof.

  \end{proof}

\subsubsection{Designing the function $\alpha(\bu_k,\ k)$ in \eqref{mul_cov}}
\label{design}
In this section, we will formulate an optimization problem to construct an optimal multiplicative correction $\alpha(\ub_k,\ k)$ in \eqref{mul_cov} for EnKI-MC (I). To that end, recall the optimal sample averaged dual function in \eqref{optimal_dual}. In the context of EnKI algorithm \eqref{new_ss}, let us define the optimal sample averaged dual function at the $k^{th}$ iteration of \eqref{new_ss}  formed using $N$ estimates $\{\ub_{k}^{(i)}\}_{i=1}^N$ as follows:
\begin{equation}
\begin{aligned}
       \DexpectSA^k :=& -\frac{1}{2}(\overline{\Blambda}_k)^T \LRp{\L_h+\A  \sC_k\A^T}\overline{\Blambda}_k-(\overline{\Blambda}_k)^T\bar{{\bf{r}}}_k,
       \label{samp_d_e}
\end{aligned}
\end{equation}
where we introduced the following notations:
\begin{equation}
\begin{aligned}
   \sC_k:= \sC(\ub_k)=&\frac{1}{N}\sum_{i=1}^N\LRp{\ub_k^{(i)}-\bar{\ub}_k}\LRp{\ub_k^{(i)}-\bar{\ub}_k}^T,\quad  \overline{\ub}_k=\frac{1}{N}\sum_{i=1}^N\ub_k^{(i)},\\
    \bar{{\bf{r}}}_k=&\db -\A \overline{\ub}_k,\ \ \overline{\Blambda}_k:=\overline{\Blambda}(\ub_k)= \LRp{\L_h+\A \sC_k\A^T}^{-1}\LRp{\A \overline{\ub}_k-\db},
\end{aligned}    
\label{r_n_C}
\end{equation}
where we assumed $\db$ is free of noise (i.e we set $\overline{\sigb}=0$ in \eqref{optimal_dual}), replaced $\bu_0$ with the updated prior mean $\overline{\ub}_k$ and also replaced the samples $\del^{(i)}$ with $\bu_k^{(i)}-\overline{\ub}_k$ in \eqref{sample_average_dual} to obtain  \eqref{samp_d_e}. That is, \eqref{samp_d_e} is nothing but the sample averaged dual function when one uses  the updated prior mean $\overline{\ub}_k$, noise free data ${\bd}$ and covariance $\sC_k$. Here, the assumption is that  $\del^{(i)}$ are samples from some true underlying Gaussian  distribution at $k^{th}$ iteration.

In order to design the function $\alpha(\bu_k,\ k)$ in \eqref{iter_new}, let us consider the perturbed dual function by introducing  the multiplicative factor $\alpha_k:=\alpha(\bu_k,\ k)$ in \eqref{samp_d_e} as follows:
\begin{equation}
\DexpectSA^k(\alpha_k)=  -\frac{1}{2}(\overline{\Blambda}_k(\alpha_k))^T \LRp{\L_h+\alpha_k\A \sC_k\A^T}\overline{\Blambda}_k(\alpha_k)-(\overline{\Blambda}_k(\alpha_k))^T\bar{{\bf{r}}}_k,
\label{pert_d}
\end{equation}
where
\[ \overline{\Blambda}_k(\alpha_k)= \LRp{\L_h+\alpha_k\A \sC_k\A^T}^{-1}\LRp{\A \overline{\ub}_k-\db}.\]
Considering the assumptions and results in Theorem \ref{mult_conv}, let us investigate what happens to the perturbed  dual function  as $k\rightarrow \infty$:
\begin{equation}
         \lim_{k\rightarrow \infty}\DexpectSA^k(\alpha_k)=
        \frac{1}{2}\LRp{\A \overline{\ub}_{\infty}-\db}^T\L_h^{-1}\LRp{\A \overline{\ub}_{\infty}-\db},
      \label{inf_dual}
\end{equation}
where we used the assumption $\lim_{n\rightarrow \infty}\alpha(\bu_k,\ k)=1$, and the results in  Theorem \ref{mult_conv} that the EnKI iteration \eqref{iter_new} converges and all the samples collapses, i.e. one has $\lim_{k\rightarrow \infty}\sC_{k}={\bf{0}}$.

Now it is natural to consider  the following optimization problem for faster convergence:
\begin{equation}
    \min_{\alpha_k}\sS(\alpha_k) := \min_{\alpha_k}\LRs{\underbrace{\LRp{\DexpectSA^{k}(\alpha_k)-\lim_{k\rightarrow \infty}\DexpectSA^k(\alpha_k)}^2}_{(I)}+\underbrace{\delta(k) \times \LRp{\alpha_k-1}^2}_{(II)}}.
    \label{alpha_opt}
\end{equation}
where the regularization parameter $\delta(k)$ is chosen to satisfy $\lim_{k\rightarrow \infty} \delta(k)=\infty $. Note that the regularization term $\delta(k)\times  \LRp{\alpha_k-1}^2$ keeps $\alpha_k$ bounded and also promotes $\lim_{k\rightarrow \infty}\alpha_k=1$ (assuming that the first term in \eqref{alpha_opt} remains bounded for any $k$) which is a key assumption in Theorem \ref{mult_conv} for the convergence of \eqref{new_ss}.

Now let us consider the ideal scenario where the data $\db$ is free of noise and the true solution $\ub^\dagger$ (i.e  $\ub^\dagger$ satisfies $\A\ub^\dagger=\db$) lies in $\sD$ where $\sD$ is the finite dimensional subspace spanned by the initial ensembles. 

\textit{Our objective is to find the sequence $\{\alpha_k\}_{k=1}^\infty$ that promotes the convergence of scheme \eqref{iter_new} to the true solution $\ub^\dagger$, i.e.  $\overline{\ub}_{\infty}=\ub^\dagger$ and minimizes \eqref{alpha_opt} (for faster convergence).} 

Note that in this case from \eqref{inf_dual} $  \lim_{k\rightarrow \infty}\DexpectSA^k(\alpha_k)=0$. Even when the data $\db$ is contaminated with small measurement noise, and $\ub^\dagger \notin \sD$, we assume that the subspace $\sD$ is rich enough to contain a solution $\overline{\ub}_{\infty} \in \sD$ close enough to $\ub^{\dagger}$ such that  $  \lim_{k\rightarrow \infty}\DexpectSA^k(\alpha_k)$ in \eqref{inf_dual} is small and hence neglected.

In Theorem \ref{opt_inf_the} we will derive the optimal solution to 
\eqref{alpha_opt} when $  \lim_{k\rightarrow \infty}\DexpectSA^k(\alpha_k)=0$.

\begin{theorem}[Optimal covariance correction factor $\alpha(\bu)$ in EnKI-MC]
\label{opt_inf_the}
    Assume that $  \lim_{k\rightarrow \infty}\DexpectSA^k(\alpha_k)=0$ in \eqref{alpha_opt}.
    Then, the minimizer $\alpha_k$ of \eqref{alpha_opt} satisfies the following fixed point equation:
    \begin{equation}
    \alpha_k=1+\frac{\LRp{\bar{{\bf{r}}}_k^T(\M(\alpha_k))^{-1}\bar{{\bf{r}}}_k}\LRp{r_n^T(\M(\alpha_k))^{-1}\A \sC_k\A^T(\M(\alpha_k))^{-1}\bar{{\bf{r}}}_k}}{4\delta(k)}=\zeta_{\delta}(\alpha_k),
    \label{alpha_method_II}
\end{equation}
where $\M(\alpha)=\L_h+\alpha\A \sC_k\A^T$ with $\L_h=\mu I$, $\bar{{\bf{r}}}_k,\ \sC_k$ are given in \eqref{r_n_C}.
\end{theorem}

\begin{proof}
   The cost function \eqref{alpha_opt}  simplifies as:
    \begin{equation}
     \sS(\alpha_k) =\LRp{\frac{1}{2}(\overline{\Blambda}_k(\alpha_k))^T \LRp{\L_h+\alpha_k\A \sC_k\A^T}\overline{\Blambda}_k(\alpha_k)+\overline{\Blambda}_k^T\bar{{\bf{r}}}_k}^2+\delta(k)  \LRp{\alpha_k-1}^2.
     \label{loss_mo_o}
\end{equation}
Using the fact that $\LRp{\L_h+\alpha_k\A \sC_k\A^T} \overline{\Blambda}_k(\alpha_k)=-\bar{{\bf{r}}}_k$, \eqref{loss_mo_o} can be simplified as::
  \begin{equation}
     \sS(\alpha_k)  = \LRp{\frac{1}{2}(\overline{\Blambda}_k(\alpha_k))^T\bar{{\bf{r}}}_k}^2+\delta(k)  \LRp{\alpha_k-1}^2.
      \label{loss_mo_oo}
\end{equation}
From the first order optimality condition for minimizing $\sS(\alpha_k)$ we get:
\begin{equation}
    \frac{1}{2} \LRp{(\overline{\Blambda}_k(\alpha_k))^T\bar{{\bf{r}}}_k}\LRp{\frac{d (\overline{\Blambda}_k(\alpha_k))^T}{d\alpha_k}\bar{{\bf{r}}}_k}+2\delta(k) (\alpha_k-1)=0.
    \label{opt_cond}
\end{equation}
Note that since $\LRp{\L_h+\alpha_k\A \sC_k\A^T} \overline{\Blambda}_k(\alpha_k)=-\bar{{\bf{r}}}_k$, we have:
\begin{equation}
   \frac{d (\overline{\Blambda}_k(\alpha_k))^T}{d\alpha_k}=\bar{{\bf{r}}}_k^T\LRp{\L_h+\alpha_k\A \sC_k\A^T}^{-1}\A \sC_k\A^T\LRp{\L_h+\alpha_k\A \sC_k\A^T}^{-1}.  
\end{equation}
Therefore \eqref{opt_cond} becomes
\[\LRp{(\overline{\Blambda}_k(\alpha_k))^T\bar{{\bf{r}}}_k} \LRp{\bar{{\bf{r}}}_k^T\LRp{\L_h+\alpha_k\A \sC_k\A^T}^{-1}\A \sC_k\A^T\LRp{\L_h+\alpha_k\A \sC_k\A^T}^{-1}\bar{{\bf{r}}}_k}+4\delta(k) (\alpha_k-1)=0. \]
Simplifying the above expression, the optimal correction factor $\alpha_k$  satisfies the following fixed-point equation:
\begin{equation}
    \alpha_k=1+\frac{\LRp{\bar{{\bf{r}}}_k^T(\M(\alpha_k))^{-1}\bar{{\bf{r}}}_k}\LRp{\bar{{\bf{r}}}_k^T(\M(\alpha_k))^{-1}\A \sC_k\A^T(\M(\alpha_k))^{-1}\bar{{\bf{r}}}_k}}{4\delta(k)}=\zeta_\delta(\alpha_k),
\end{equation}
where $\M(\alpha_k)=\L_h+\alpha_k\A \sC_k\A^T$.
\end{proof}

\begin{remark}
   To find the optimal $\alpha_k$ in \eqref{alpha_method_II} one may consider the following fixed-point iteration:
\begin{equation}
   \alpha_k^{p+1}=\zeta_\delta\LRp{\alpha_k^{p}},\quad p=0,\dots\ \ \mathrm{with}\ \ \alpha_k^0=\alpha_{k-1}^{\infty},
    \label{iterative_conver}
\end{equation}   
    where $\alpha_k^{p+1}$ denotes the value of $\alpha_k$ at the $(p+1)^{th}$ iteration of \eqref{iterative_conver}, $\alpha_{k-1}^{\infty}$ denotes the converged value of $\alpha_{k-1}$ for the previous EnKI iteration. One of the key questions that remains is if \eqref{alpha_method_II} has a solution and if the scheme \eqref{iterative_conver} converges?
    Proposition \ref{prop_conv_ban} gives a condition on the regularization factor $\delta(k)$ such that the fixed point equation \eqref{alpha_method_II} has a unique solution.
\end{remark}
\begin{proposition}[Existence of unique solution for \eqref{alpha_method_II}]
\label{prop_conv_ban}
    Consider the fixed point equation \eqref{alpha_method_II} where the function $\zeta_\delta(.):[1,\ \infty) \mapsto [1,\ \infty)$. The equation  \eqref{alpha_method_II}  has a unique solution if the regularization $\delta(k)$ has the form:
    \begin{equation}
        \delta(k)=\frac{3}{4q}\LRp{\lambda_{max}^2\LRp{\A \sC_k\A^T}\frac{\nor{\bar{{\bf{r}}}_k}_2^4}{(\mu+ \lambda_{min}\LRp{\A\sC_k\A^T})^4}}+\epsilon_{\delta}\times k.
        \label{delta_val}
    \end{equation}
    where $\bar{{\bf{r}}}_k,\ \sC_k$ are given in \eqref{r_n_C}, $\mu$ is defined in \eqref{opt_inf_the}, $q$ is a positive number less than $1$, and $\epsilon_\delta\in \R^+$. Moreover, the iterative scheme \eqref{iterative_conver} converges to the unique solution.
\end{proposition}
\begin{proof}
We first note the range of $\zeta_\delta(\alpha)$ in \eqref{alpha_method_II} is $[1,\infty)$ which is due to the fact that $\M(\alpha)$, and $(\M(\alpha))^{-1}\A \sC_k\A^T(\M(\alpha))^{-1}$ are positive semi-definite matrices. 
We will now determine the conditions on $\delta(k)$  such that the mapping $\zeta_\delta(\alpha):[1,\ \infty) \mapsto [1,\ \infty)$ is a contraction \cite{conrad2014contraction}. This ensures that \eqref{alpha_method_II} has a unique solution and  the iterative scheme \eqref{iterative_conver} converges to the unique solution \cite{conrad2014contraction}. Since $\zeta_\delta(\alpha)$ is differentiable on $[1,\ \infty)$,  in order to show that $\zeta_\delta(\alpha)$ is a contraction,  it is sufficient to show the following:
\begin{equation}
     \snor{\zeta_\delta'(\alpha)}\leq q<1,\quad \forall \alpha \in [1,\ \infty).
     \label{contract}
\end{equation}
    We have:
    \begin{equation}
            \begin{aligned}
        \zeta_\delta'(\alpha)=&\frac{1}{4\delta(k)}\LRp{\frac{d\LRp{\bar{{\bf{r}}}_k^T(\M(\alpha))^{-1}\bar{{\bf{r}}}_k}}{d\alpha}\LRp{\bar{{\bf{r}}}_k^T(\M(\alpha))^{-1}\A \sC_k\A^T(\M(\alpha))^{-1}\bar{{\bf{r}}}_k}}\\
        &+\frac{1}{4\delta(k)}\LRp{{\LRp{\bar{{\bf{r}}}_k^T(\M(\alpha))^{-1}\bar{{\bf{r}}}_k}}\frac{d\LRp{ \bar{{\bf{r}}}_k^T(\M(\alpha))^{-1}\A \sC_k\A^T(\M(\alpha))^{-1}\bar{{\bf{r}}}_k}}{d\alpha}}.
    \end{aligned}
    \end{equation}
    Define $f_1(\alpha)=\bar{{\bf{r}}}_k^T(\M(\alpha))^{-1}\bar{{\bf{r}}}_k$, $f_2(\alpha)= \bar{{\bf{r}}}_k^T(\M(\alpha))^{-1}\A \sC_k\A^T(\M(\alpha))^{-1}\bar{{\bf{r}}}_k$. Then,
\[\frac{d f_1(\alpha)}{d\alpha}= -\LRp{ \bar{{\bf{r}}}_k^T(\M(\alpha))^{-1}\A \sC_k\A^T(\M(\alpha))^{-1}\bar{{\bf{r}}}_k}=-f_2(\alpha), \]
\[\frac{d f_2(\alpha)}{d\alpha}= -2\LRp{\bar{{\bf{r}}}_k^T(\M(\alpha))^{-1}\A \sC_k\A^T(\M(\alpha))^{-1}\A \sC_k\A^T(\M(\alpha))^{-1}\bar{{\bf{r}}}_k}=-2f_3(\alpha). \]
Therefore, we have:
\begin{equation}
    \snor{\zeta_\delta'(\alpha)}= \frac{1}{4\delta(k)}\snor{\LRp{(f_2(\alpha))^2+2f_1(\alpha)f_3(\alpha)}}=\frac{1}{4\delta(k)}\LRp{(f_2(\alpha))^2+2f_1(\alpha)f_3(\alpha)},
    \label{deri_bound}
\end{equation}
since $f_1(\alpha),\ f_2(\alpha),\ f_3(\alpha)\geq 0$. 
Now we note that $f_1(\alpha)$ satisfies the following bound:
\begin{equation}
\begin{aligned}
    f_1(\alpha)=\snor{f_1(\alpha)}=\snor{\bar{{\bf{r}}}_k^T(\M(\alpha))^{-1}\bar{{\bf{r}}}_k}\leq \nor{\bar{{\bf{r}}}_k}_2^2\nor{(\M(\alpha))^{-1}}_2\leq \frac{\nor{\bar{{\bf{r}}}_k}_2^2}{\lambda_{min}\LRp{(\M(\alpha))}} ,
\end{aligned}
 \label{bound_1}
\end{equation}
where we used the Cauchy-Swartz inequality and used the fact that $l_2$ norm is consistent. Similarly, one can show that $f_2(\alpha),\ f_3(\alpha)$ satisfies the following bounds:
\begin{equation*}
    \begin{aligned}
         f_2(\alpha)&\leq \lambda_{max}\LRp{\A \sC_k\A^T}\frac{\nor{\bar{{\bf{r}}}_k}_2^2}{\lambda_{min}^2\LRp{(\M(\alpha))}},\\
         f_3(\alpha)&\leq  \lambda_{max}^2\LRp{\A \sC_k\A^T}\frac{\nor{\bar{{\bf{r}}}_k}_2^2}{\lambda_{min}^3\LRp{(\M(\alpha))}},
    \end{aligned}
    \label{bound_2}
\end{equation*}
where we used the submultiplicativity of spectral norm to derive the bounds.
Substituting \eqref{bound_1} and \eqref{bound_2}  in \eqref{deri_bound} we have:
\begin{equation*}
\begin{aligned}
\snor{\zeta_\delta'(\alpha)}&\leq \frac{1}{4\delta(k)} \LRp{\lambda_{max}^2\LRp{\A \sC_k\A^T}\frac{\nor{\bar{{\bf{r}}}_k}_2^4}{\lambda_{min}^4\LRp{(\M(\alpha))}}+2 \frac{\nor{\bar{{\bf{r}}}_k}_2^4}{\lambda_{min}^4\LRp{(\M(\alpha))}} \ \lambda_{max}^2\LRp{\A \sC_k\A^T} } \\
&\leq \frac{3}{4\delta(k)}\LRp{\lambda_{max}^2\LRp{\A \sC_k\A^T}\frac{\nor{\bar{{\bf{r}}}_k}_2^4}{\lambda_{min}^4\LRp{(\M(\alpha))}}}\\
& \leq \frac{3}{4\delta(k)}\LRp{\lambda_{max}^2\LRp{\A \sC_k\A^T}\frac{\nor{\bar{{\bf{r}}}_k}_2^4}{(\mu+ \lambda_{min}\LRp{\A\sC_k\A^T})^4}},\quad \forall \alpha \in [1,\infty).
\end{aligned}    
\end{equation*}
where we used:
\[\lambda_{min}\LRp{(\M(\alpha))}\geq \mu+ \lambda_{min}\LRp{\A\sC_k\A^T},\quad \forall \alpha \in [1,\infty). \]
Now for \eqref{contract} to hold we need:
\[
\frac{3}{4\delta(k)}\LRp{\lambda_{max}^2\LRp{\A \sC_k\A^T}\frac{\nor{\bar{{\bf{r}}}_k}_2^4}{(\mu+ \lambda_{min}\LRp{\A\sC_k\A^T})^4}}\leq q<1.
\]
\[ \implies \delta \geq\frac{3}{4q}\LRp{\lambda_{max}^2\LRp{\A \sC_k\A^T}\frac{\nor{\bar{{\bf{r}}}_k}_2^4}{(\mu+ \lambda_{min}\LRp{\A\sC_k\A^T})^4}}.\]
Finally choose $\delta(k)$ as:
\begin{equation}
    \delta(k)=\underbrace{\frac{3}{4q}\LRp{\lambda_{max}^2\LRp{\A \sC_k\A^T}\frac{\nor{\bar{{\bf{r}}}_k}_2^4}{(\mu+ \lambda_{min}\LRp{\A\sC_k\A^T})^4}}}_{\delta_1(k)}+\underbrace{\epsilon_{\delta}\times k}_{\delta_2(k)},
    \label{two_parts}
\end{equation}
where $\epsilon_{\delta}\in \R^+$. This concludes the proof.
\end{proof}

\begin{remark}
\label{stability_remark}
Note that the regularization $\delta(k)$ in \eqref{two_parts} has two components: $\delta_1(k)$ and $\delta_2(k)$. When $\delta_1(k)$ or $\delta_2(k)$ is very large, this forces the solution $\alpha_k$ to be very close to $1$ (see the optimization problem \eqref{alpha_opt}). 
A very large value of $\delta_1(k)$ is related to a large condition number of  $\LRp{\L_h+\A\sC_k\A^T}$ in \eqref{new_ss} or a high value of $\nor{\bar{{\bf{r}}}_k}_2$ as can be seen below: 
\begin{equation}
    \delta_1(k)= \frac{3}{4q}\LRp{\frac{ \lambda_{max}\LRp{\A\sC_k\A^T})}{\mu+ \lambda_{min}\LRp{\A\sC_k\A^T})}}^2\LRp{\frac{\nor{\bar{{\bf{r}}}_k}_2^2}{(\mu+ \lambda_{min}\LRp{\A\sC_k\A^T})^2}}\leq \frac{3}{4q} \kappa^2\LRp{\L_h+\A\sC_k\A^T}\LRp{\frac{\nor{\bar{{\bf{r}}}_k}_2}{\mu}}^2,
    \label{delta_1}
\end{equation}
where $\kappa\LRp{\L_h+\A\sC_k\A^T}$ denotes the  condition number of  $\LRp{\L_h+\A\sC_k\A^T}$ and we assumed $\lambda_{min}\LRp{\A\sC_k\A^T})\approx 0$ due to the rank deficiency of the sample covariance matrix $\sC_k$. Now note that the Kalman gain of the perturbed scheme \eqref{iter_new} can be written as:    
\[ \alpha_k\sC_k\A^T\LRp{\L_h+\alpha_k\A\sC_k\A^T}^{-1}=\sC_k\A^T\LRp{\frac{\L_h}{\alpha_k}+\A\sC_k\A^T}^{-1}.\]
It is clear that in EnKI, the term $\frac{\L_h}{\alpha_k}$ acts as the regularization term to stabilize the inversion of ill-conditioned $\A\sC_k\A^T$ \cite{iglesias2021adaptive}. When $\alpha_k>>1$,
$\LRp{\frac{\L_h}{\alpha_k}+\A\sC_k\A^T}$ is  even more ill-conditioned than $\LRp{{\L_h}+\A\sC_k\A^T}$.
Thus, a large value of $\delta_1(k)$ forces $\alpha_k$ to be close to $1$ which arises from a numerical stability of inversion point of view. 

The term $\delta_2(k)$ forces convergence of $\alpha_k$ to 1 as $k\rightarrow \infty$.
which is a key assumption in Theorem \ref{mult_conv}. 
\end{remark}
\begin{remark}[Nature of the resulting curve of $\alpha_k$ vs. $k$]
\label{curve_nature}
   A consequence of remark \ref{stability_remark} is the increasing-decreasing nature of $\alpha_k$ curve as shown in Figure \ref{opt_1}. To understand this behavior we split the regularization $\delta(k)$ as in \eqref{two_parts}:
   \[\delta(k)=\delta_1(k)+\delta_2(k), \]
   \begin{figure}[h!]
   \centering
 \includegraphics[scale=0.38]{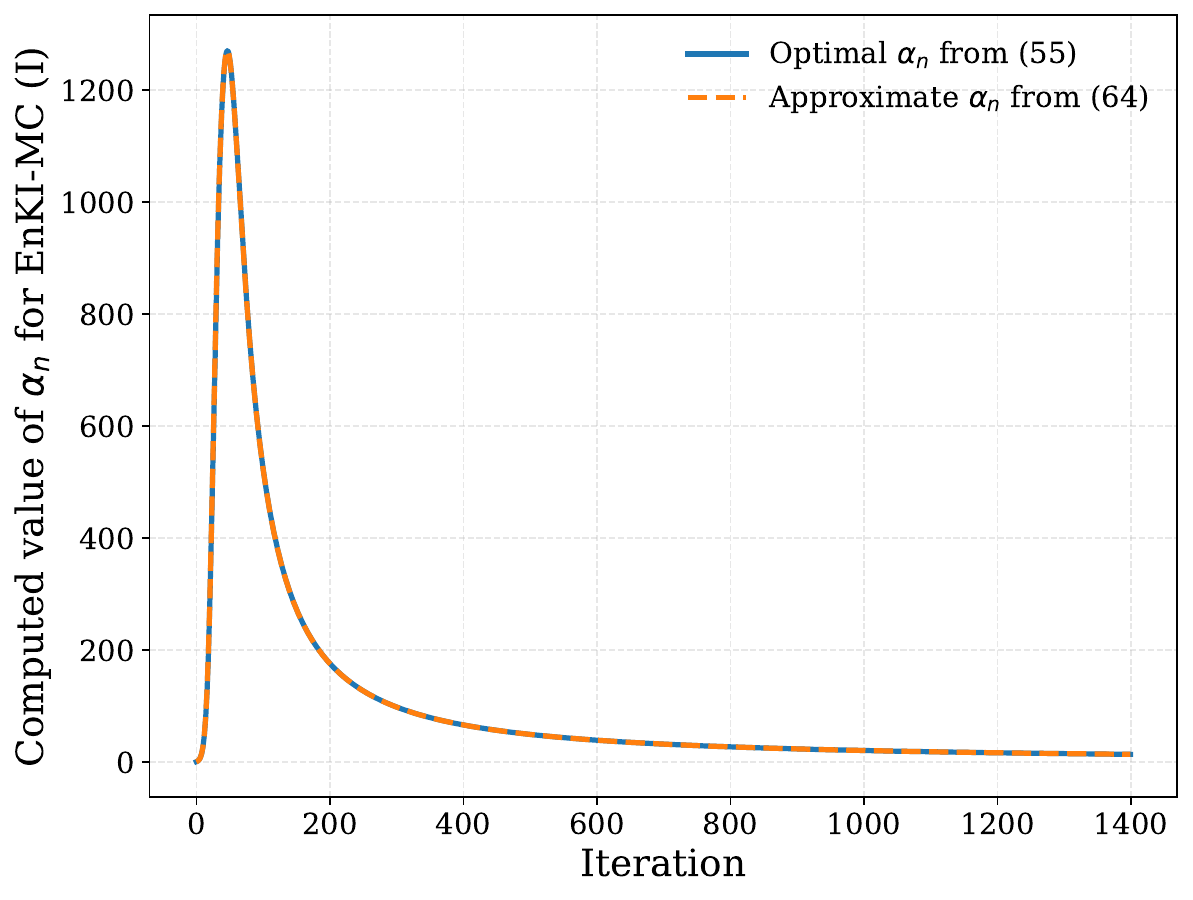}
\caption{Optimal correction factor computed at each iteration of EnKI-MC (I) using \eqref{iterative_conver} and the corresponding approximation \eqref{approx_alpha} for 1D Deconvolution problem.}
\label{opt_1}
\end{figure}

   \vspace{0.1 cm}
  \hspace{-0.65 cm} {\bf{Regime (I): Increasing $\alpha_k$ for convergence acceleration}}
  \vspace{0.1 cm}
  
   In this regime $\delta_2(k)$ is very small due to the choice of a very small $\epsilon_\delta$ in \eqref{two_parts}, and the behavior is dominated by the value of $\delta_1(k)$ and the value of loss (part (I)) in \eqref{alpha_opt}.
   In the beginning of the iteration, numerically we observe that a huge condition number of condition number of  $\LRp{\L_h+\A\sC_k\A^T}$ (and hence a high value of $\delta_1(k)$ in \eqref{delta_1}) forces $\alpha_k$ to be close to $1$. Here, the huge condition number of condition number of  $\LRp{\L_h+\A\sC_k\A^T}$ is caused due to $\sigma_{max}\LRp{\A\sC_k\A^T}>>\sigma_{min}\LRp{\A\sC_k\A^T}$.
    As the iteration proceeds, the ensembles start collapsing, and $\nor{\sC_k}_2$ starts decreasing  and we see that the condition number of  $\LRp{\L_h+\A\sC_k\A^T}$ decreases as well. Consequently the regularization $\delta_1(k)$ decreases rapidly as well and $\alpha_k$ starts increasing to minimize part (I) in \eqref{alpha_opt} as seen in the initial part of the curve.   
    
   \vspace{0.1 cm}
  \hspace{-0.65 cm} {\bf{Regime (II): Decreasing $\alpha_k$ for reimposing implicit regularization in EnKI framework}}
\vspace{0.1 cm}

  Once $\delta_1(k)$ is very small, the behavior of $\alpha_k$ is influenced by the part of regularization $\delta_2(k)$ in \eqref{two_parts} and  the value of loss (part (I)) in \eqref{alpha_opt}. From the optimal solution \eqref{alpha_method_II}, it is clear that the monotonically increasing nature of regularization $\delta_2(k)$ coupled with the collapse of ensembles (decreasing  $\nor{\sC_k}_2$) causes $\alpha_k$ to decay rapidly. Since the term $\frac{\L_h}{\alpha_k}$ acts as the regularization term inside the EnKI framework (see remark \ref{stability_remark}), the  decay of $\alpha_k$ plays a crucial role
    in reimposing the regularization inside the EnKI iteration. In our numerical results, we will show that this prevents overfitting on the residuals and leads to better quality solution on termination of our algorithm.
\end{remark}

\begin{remark}
\label{approximation_remark}
Note that for EnKI-MC (I), a fixed-point iteration \eqref{iterative_conver} has to be performed within the iterative scheme \eqref{iter_new} (for each $k$) which would be computationally expensive. An approximate solution to  \eqref{alpha_method_II} may be computed  by considering the first order Taylor series approximation of $\zeta_{\delta}(\alpha)$ about $\alpha_{k-1}^*$ as follows:
\[ \alpha_k^*\approx\zeta_{\delta}(\alpha_{k-1}^*)+\zeta_{\delta}'(\alpha_{k-1}^*)\LRp{\alpha_k^*-\alpha_{k-1}^*},\]
where we  neglected the higher-order terms. Thus, we arrive at an expression for $\alpha_k^*$ as follows:
\begin{equation}
    \alpha_k^*\approx  \alpha_{k-1}^*+\frac{\zeta_{\delta}(\alpha_{k-1}^*)-\alpha_{k-1}^*}{1-\zeta_{\delta}'(\alpha_{k-1}^*)}, \quad \alpha_0^*=1,
    \label{approx_alpha}
\end{equation}
and we use $\delta(k)$ given by \eqref{delta_val} in computing $\zeta_{\delta}(\alpha_{k-1}^*)$ given in \eqref{alpha_method_II}. Figure \ref{opt_1} shows the comparison between $\alpha_k^*$ computed using \eqref{alpha_method_II} (fixed-point iteration) and the approximate formulae \eqref{approx_alpha} for the 1D Deconvolution problem. We see from Figure \ref{opt_1} that the computed $\alpha_k^*$ from both approaches matches closely justifying the use of  approximation \eqref{approx_alpha} in our case.
\end{remark}

\subsection{EnKI-MC (II): Ensemble specific multiplicative covariance correction strategy }
\label{conv_imp_2}
Note that in the EnKI update equation \eqref{iter_new}, the  multiplicative factor $\alpha(\ub_k,\ k)$ in \eqref{mul_cov} is the same for each ensemble  member. However, note that one may assign a distinct correction factor $\alpha^{(i)}(\ub_k,\ k)$ to each ensemble member, allowing larger updates for ensembles that are farther from convergence compared to those that are already close. In this section, we consider a strategy where the update equation for each ensemble, uses a different multiplicative correction to the covariance. In particular, we consider the following modification to the update rule \eqref{iter_new}:
\begin{equation}
  \ub_{k+1}=g_m^{(II)}(\bu_k,\ {\boldsymbol{\alpha}}(\bu_k,\ k)) = \begin{bmatrix}
\bu_k^{(1)} \\
\bu_k^{(2)} \\
\vdots \\
\bu_k^{(N)}
\end{bmatrix} + 
\begin{bmatrix}
\alpha^{(1)}(\bu_k,\ k)\sC(\ub_k)\A^T \\
\alpha^{(2)}(\bu_k,\ k)\sC(\ub_k)\A^T \\
\vdots \\
\alpha^{(N)}(\bu_k,\ k)\sC(\ub_k)\A^T 
\end{bmatrix}
\odot
\begin{bmatrix}
\Blambda^{(1)}(\bu_k,\ \alpha^{(1)}(\bu_k,\ k)) \\
\Blambda^{(2)}(\bu_k,\ \alpha^{(2)}(\bu_k,\ k)) \\
\vdots \\
\Blambda^{(N)}(\bu_k,\ \alpha^{(N)}(\bu_k,\ k))
\end{bmatrix},
\label{fixed_point_equation_dif}
\end{equation}
where ${\boldsymbol{\alpha}}(\bu_k,\ k)=[{\alpha}^{(1)}(\bu_k,\ k)),\dots {\alpha}^{(N)}(\bu_k,\ k))]$, and
    \begin{equation}
        \Blambda^{(i)}(\ub_k,\ \alpha^{(i)}(\bu_k,\ k))=\LRp{\L+\A\alpha^{(i)}(\bu_k,\ k)\sC(\ub_k)\A^T}^{-1}\LRp{\db -\A\ub^{(i)}_k}.
        \label{iterative_lambda_dif}
    \end{equation}
%\krish{
    %\begin{equation}
     %   \Blambda^i_e(\ub_e,\db_e)=\LRp{\L+\A\sC(\ub_e)\A^T}^{-1}\LRp{\db_e^i -\A\ub_e^i},
      %  \label{iterative_lambda_n}
  %  \end{equation}}
We call \eqref{fixed_point_equation_dif} as EnKI-MC (II). For designing the function $\alpha^{(i)}(\bu_k,\ k)$ in \eqref{fixed_point_equation_dif} we follow the same strategy in section \ref{design}.
In particular, instead of considering a sample averaged dual function \eqref{pert_d}, one may define the perturbed sample-specific dual function  (for the $i^{th}$ ensemble) by introducing  the multiplicative factor $\alpha_k^{(i)}:=\alpha^{(i)}(\bu_k,\ k)$ as follows:
\begin{equation}
\begin{aligned}
       \DexpectSi^k\LRp{\alpha^{(i)}_k} :=& -\frac{1}{2}({\Blambda}_k^{(i)}(\alpha^{(i)}_k))^T \LRp{\L_h+ \alpha^{(i)}_k\A \sC_k\A^T}{\Blambda}_k^{(i)}(\alpha^{(i)}_k)-({\Blambda}_k^{(i)}(\alpha^{(i)}_k))^T{\bf{r}}_k^{(i)},
\end{aligned}
\label{samp_d_e_e}
\end{equation}
where
\begin{equation}
\begin{aligned}
    \sC_k=&\frac{1}{N}\sum_{i=1}^N\LRp{\ub_k^{(i)}-\bar{\ub}_k}\LRp{\ub_k^{(i)}-\bar{\ub}_k}^T,\quad  \overline{\ub}_k=\frac{1}{N}\sum_{i=1}^N\ub_k^{(i)},\\
    {\bf{r}}_k^{(i)}=& \db -\A {\ub}_k^{(i)},\ \ {\Blambda}_k^{(i)}(\alpha_k^{(i)})= \LRp{\L_h+\alpha_k^{(i)}\A \sC_k\A^T}^{-1}\LRp{\A {\ub}_k^{(i)}-\db}.
\end{aligned}   
\label{r_n_i}
\end{equation}
Note that \eqref{samp_d_e_e} is  the optimal sample averaged dual function 
for a given $i^{th}$ ensemble where the averaging is done only to estimate the covariance $\sC_k$. Similar to \eqref{inf_dual} as  $k\rightarrow \infty$ one has: 
\begin{equation}
         \lim_{k\rightarrow \infty}\DexpectSi^k(\alpha_k^{(i)})=
        \frac{1}{2}\LRp{\A {\ub}_{\infty}^{(i)}-\db}^T\L_h^{-1}\LRp{\A {\ub}_{\infty}^{(i)}-\db},
      \label{inf_ii}
\end{equation}
where we assumed $\lim_{k\rightarrow \infty}\alpha^{(i)}(\bu_k,\ k)=1$, and that the EnKI iteration \eqref{fixed_point_equation_dif} converges and all the samples collapses, i.e. one has $\lim_{k\rightarrow \infty}\sC_{k}={\bf{0}}$. Now may now consider the following optimization problem for faster convergence of each ensemble independently:
\begin{equation}
    \min_{\alpha_k^{(i)}}\sS(\alpha_k^{(i)}) := \min_{\alpha_k^{(i)}}\LRs{\LRp{\DexpectSi^{k}(\alpha_k^{(i)})-\lim_{k\rightarrow \infty}\DexpectSi^k(\alpha_k^{(i)})}^2+\delta^{(i)}(k) \times \LRp{\alpha_k^{(i)}-1}^2},\quad i=1,\dots N.
    \label{alpha_opt_new}
\end{equation}
where the regularization parameter $\delta^{(i)}(k)$ is chosen to satisfy $\lim_{k\rightarrow \infty} \delta^{(i)}(k)=\infty $. Following similar assumptions and derivations in Theorem \ref{opt_inf_the}, Proposition \ref{prop_conv_ban} and Remark \ref{approximation_remark}, the optimal solution  $\alpha_k^{(i)}$ is computed as:
\begin{equation}
    \alpha_k^{(i)}\approx  \alpha_{k-1}^{(i)}+\frac{\zeta_{\delta}^{(i)}\LRp{\alpha_{k-1}^{(i)}}-\alpha_{k-1}^{(i)}}{1-\frac{d \zeta_{\delta}^{(i)}\LRp{\alpha}}{d\alpha}\Bigg |_{\alpha=\alpha_{k-1}^{(i)}}}, \quad \alpha_0^{(i)}=1,
    \label{approx_alpha_II}
\end{equation}
where the function $\zeta^{(i)}_{\delta}\LRp{\alpha}$ is given by:
\begin{equation}
\zeta^{(i)}_{\delta}\LRp{\alpha}=1+\frac{\LRp{({\bf{r}}_k^{(i)})^T(\M(\alpha))^{-1}{\bf{r}}_k^{(i)}}\LRp{({\bf{r}}_k^{(i)})^T(\M(\alpha))^{-1}\A \sC_k\A^T(\M(\alpha))^{-1}{\bf{r}}_k^{(i)}}}{4\delta^{(i)}(k)},
\label{zeta_ii}
\end{equation}
\begin{equation}
    \delta^{(i)}(k)=\frac{3}{4q}\LRp{\lambda_{max}^2\LRp{\A \sC_k\A^T}\frac{\nor{{\bf{r}}_k^{(i)}}_2^4}{(\mu+ \lambda_{min}\LRp{\A\sC_k\A^T})^4}}+\epsilon_{\delta}\times k,
    \label{delta_ii}
\end{equation}
where $\epsilon_{\delta}\in \R^+$, $M(\alpha)$ is defined in  Theorem \ref{opt_inf_the}, ${\bf{r}}_k^{(i)},\ \sC_k$ are given in \eqref{r_n_i}.

The algorithm for EnKI-MC (I) and EnKI-MC (II) is provided in Algorithm \ref{Algo_full} and Algorithm \ref{Algo_full_II}. The EnKI iteration is terminated based on the relative change in the solution between consecutive iterations, as specified in line 2 of Algorithm \ref{Algo_full}. Another criteria that one may use in practice is based on Morozov discrepancy principle \cite{iglesias2021adaptive}. However, this requires prior knowledge on the noise level in the observation data.  In this work, we assume that the noise level in the data is unknown.
\paragraph{Choice of hyperparameters in Algorithm \ref{Algo_full} and \ref{Algo_full_II}}
Note that one of the key hyperparameter used in Algorithm \ref{Algo_full} and Algorithm \ref{Algo_full_II} is $\epsilon_{\delta}$ that appears in \eqref{delta_val}. A very high value of $\epsilon_\delta$ forces the optimal correction $\alpha_k^*$ to be small thereby affecting (reducing) the speed of convergence of the algorithm. On the other hand, a very small value of $\epsilon_\delta$ can cause $\alpha_k^*$ to be very large at some iteration of EnKI-MC thereby affecting the stability and destroying the implicit regularization in ENKI (see remark \ref{stability_remark}). In practice, it is essential to control the parameter $\epsilon_\delta$ to bound the maximum value of $\alpha_k^*$ at any iteration $k$. Line 5 of Algorithm \ref{Algo_full} takes care of this issue by dynamically adjusting $\epsilon_\delta$ such that $\alpha_k^*<\alpha_{bound}$. In practice, we start the  Algorithm \ref{Algo_full} with the choice of a very small value for $\epsilon_\delta$ ($10^{-15}$). We set $\alpha_{bound}=10000$ to provide a safety margin that prevents numerical instability, while it is rarely activated during normal algorithm operation (see, for instance, Figure~\ref{ENKF_20_samples_conv_valid}). The remaining details in Algorithm \ref{Algo_full} are self explanatory.
\begin{algorithm} 
	\caption{EnKI-MC (I) Algorithm for linear inverse problem}
	\hspace*{\algorithmicindent} \textbf{Input}: Initial list of ensembles $\ub_0=[\ub_0^{(1)},\ \ub_0^{(2)},\ \ub_0^{(3)}\dots \ub_0^{(N)}]$ (column vector) , data $\db$, linear operator $\A$, data covariance matrix $\L_h=\mu I$, function $g_m^{(I)}(.,.)$ in \eqref{iter_new}, function $\zeta_\delta\LRp{.}$ in \eqref{alpha_method_II},    function $\delta(k)$ in \eqref{delta_val}, tolerance $\epsilon_c$, parameter $\epsilon_\delta,\ q$, parameter $\alpha_{bound}$, maximum iteration $I_{max}$. \\ 
	\begin{algorithmic}[1] 
  		\State set $k =0$ set $\ub_{-1}$ such that $\LRs{\frac{\nor{\ub_{0}-\ub_{-1}}_2}{\nor{\ub_{-1}}_2}}>\epsilon_c$,  
		\While{$\LRs{\frac{\nor{\ub_{k}-\ub_{k-1}}_2}{\nor{\ub_{k-1}}_2}}>\epsilon_c$ {\bf{and}} $k\leq I_{max}$} 
        \State Compute the value of $\delta(k)$ using \eqref{delta_val} using parameters  $\epsilon_\delta,\ q$. 
 \State Compute the multiplicative correction factor $\alpha_k^*$ using \eqref{approx_alpha}.
  \State If $\alpha_k^*>\alpha_{bound}$, increase $\epsilon_{\delta}$ and recompute $\alpha_k^*$ (line 3 and line 4) until  $\alpha_k^*<\alpha_{bound}$.
   \State Update solution $\bu_k$ as follows:
   \[  \bu_{k+1}=g_m^{(I)}(\bu_k,\ \alpha_k^*)\]
   \State Set $k=k+1$
		\EndWhile
	\end{algorithmic} \label{Algo_full}
\hspace*{\algorithmicindent} \textbf{Output}: Converged solution $\ub^*_{\mathrm{EnKI-MC(I)}}$.
\end{algorithm}

In addition to details in Algorithm \ref{Algo_full},  Algorithm  \ref{Algo_full_II}, employs an additional parameter $K$ (line 3) to determine how often one recomputes the optimal multiplicative correction factor $\alpha_k^{(i)}$. Note that this strategy is introduced for computational efficiency to reduce the number of computations especially in scenarios where the number of ensembles (i.e $N$) is large.
\begin{algorithm} 
	\caption{EnKI-MC (II) Algorithm  for linear inverse problem}
	\hspace*{\algorithmicindent} \textbf{Input}: Initial list of ensembles $\ub_0=[\ub_0^{(1)},\ \ub_0^{(2)},\ \ub_0^{(3)}\dots \ub_0^{(N)}]$ (column vector) , data $\db$, linear operator $\A$, data covariance matrix $\L_h=\mu I$, function $g_m^{(II)}(.,.)$ in \eqref{fixed_point_equation_dif}, function $\zeta_\delta^{(i)}\LRp{.}$ in \eqref{zeta_ii},  function $\delta^{(i)}(k)$ in \eqref{delta_ii},  tolerance $\epsilon_c$, parameter $\epsilon_\delta,\ q$, parameter $K$, parameter $\alpha_{bound}$,  maximum iteration $I_{max}$.\\ 
	\begin{algorithmic}[1] 
  		\State set $k =0$ set $\ub_{-1}$ such that $\LRs{\frac{\nor{\ub_{0}-\ub_{-1}}_2}{\nor{\ub_{-1}}_2}}>\epsilon_c$,  
		\While{$\LRs{\frac{\nor{\ub_{k}-\ub_{k-1}}_2}{\nor{\ub_{k-1}}_2}}>\epsilon_c$ {\bf{and}} $k\leq I_{max}$} 
        \If{n $\%$ K==0}
       \State Compute the value of $\delta^{(i)}(k)$ (for each $i$) using \eqref{delta_ii} using parameters  $\epsilon_\delta,\ q$. 
 \State Compute the multiplicative correction factor $\alpha_k^{(i)}$ (for each $i$) using \eqref{approx_alpha_II}.
  \State If $\max\{\alpha_k^{(i)}\}_{i=1}^N>\alpha_{bound}$, increase $\epsilon_{\delta}$ and recompute $\alpha_k^{(i)}$ (line 4 and line 5) until  $\max\{\alpha_k^{(i)}\}_{i=1}^N<\alpha_{bound}$.
 \Else
 \State Set $\alpha_k^{(i)}=\alpha_{k-1}^{(i)}$, $\forall i$.
        \EndIf
   \State Update solution $\bu_k$ as follows:
   \[  \bu_{k+1}=g_m^{(II)}(\bu_k,\ \boldsymbol{\alpha}_k),\]
   where ${\boldsymbol{\alpha}_k}=[{\alpha}^{(1)}_k,\dots {\alpha}^{(N)}_k]$,
   \State Set $k=k+1$
		\EndWhile
	\end{algorithmic} \label{Algo_full_II}
\hspace*{\algorithmicindent} \textbf{Output}: Converged solution $\ub^*_{\mathrm{EnKI-MC(II)}}$.
\end{algorithm}

\section{Implementation of EnKI-MC (I) and EnKI-MC (II)  for non-linear inverse problem}

\label{non_linear}
The EnKI iteration \eqref{new_ss} for nonlinear inverse problem where the forward map is denoted as $\G(\ub): \R^n\mapsto \R^m$ is given by \cite{schillings2017analysis}:
\begin{equation}
    \ub_{k+1}^{(i)} =\ub_k^{(i)} +\sC_{up}(\ub_k) \LRp{\L_h+\sC_{pp}(\ub_k)}^{-1}\LRp{\db -\G(\ub_k^{(i)})},
    \label{non_lin_enki}
\end{equation}
where
 \[\sC_{up}(\bu_k)=\frac{1}{N}\sum_{i=1}^N\LRp{\ub_k^{(i)}-\bar{\ub}_k}\LRp{\G(\ub_k^{(i)})-\bar{\G}_k}^T,\ \ \sC_{pp}(\bu_k)=\frac{1}{N}\sum_{i=1}^N\LRp{\G(\ub_k^{(i)})-\bar{\G}_k}\LRp{\G(\ub_k^{(i)})-\bar{\G}_k}^T,\]
 \[ \bar{\G}_k=\frac{1}{N}\sum_{i=1}^N \G(\ub_k^{(i)}),\ \ \bar{\ub}_k=\frac{1}{N}\sum_{i=1}^N\ub_k^{(i)}. \]
Note that if one assumes  $\G(\ub)\approx \G(\bar{\ub}_k)+\A(\ub-\bar{\ub}_k)$, i.e. we consider the linearization of $\G(\ub)$ about  $\bar{\ub}_k$, and $\A$ is the Jacobian of $\G$ w.r.t $\ub$ evaluated at $\bar{\ub}_k$, then it is easy to see that $\sC(\ub_k)\A^T\approx \sC_{up}(\bu_k)$, and $\A\sC(\ub_k)\A^T \approx \sC_{pp}(\bu_k)$ in  \eqref{new_ss}. For EnKI-MC (I) and EnKI-MC (II), the same multiplicative correction factor is applied to both $\sC_{pp}(\bu_k)$ and $\sC_{up}(\bu_k)$.
The implementation of EnKI-MC (I) and EnKI-MC (II) algorithm  for nonlinear inverse problems is then straightforward by  
simply replacing $\A\sC_k\A^T$ with $\sC_{pp}(\bu_k)$ in Theorem \ref{opt_inf_the}, Proposition  \ref{prop_conv_ban}, and defining ${\bf{r}}_k^{(i)}= \db -\G ({\ub}_k^{(i)})$, $\bar{{\bf{r}}}_k= \db -\bar{\G}_k$.

\section{Numerical Verification} 
\label{numeric_section}
In this section we present numerical results concerning the EnKF and EnKI algorithm for a variety of inverse problems namely:
\begin{itemize}
    \item 1D Deconvolution Problem.
    \item Initial Condition Inversion in an Advection-Diffusion Problem.
    \item Initial condition inversion  in a Lorenz 96 model.
        \item Nonlinear parameter inversion in a steady-state heat equation.
\end{itemize}
In particular, we present two types of result: i)  Verifying the  non-asymptotic convergence result for EnKF algorithm (Theorem \ref{non_asymp}); ii) Analysing the performance of EnKI-MC(I) and EnKI-MC(II) in convergence acceleration of vanilla EnKI algorithm.

\subsection{General settings for numerical results}
In this section we present the general numerical settings used in Algorithm \ref{Algo_full} and Algorithm \ref{Algo_full_II} for all problems. We set
 $\epsilon_\delta=10^{-15}$, $q=0.99$ for all problems. The tolerance $\epsilon_c$ is set to $10^{-5}$ for linear inverse problems, and $10^{-4}$ for nonlinear  inverse problems. We choose $\L_h=(0.1)^2I$, where $I$ is the $m\times m$ identity matrix. The parameter $\alpha_{bound}$ is set to $10000$, $I_{max}$ is set to $10,000$. 
 The parameter $K$ in Algorithm \ref{Algo_full_II} is chosen as $5$. Further, before deploying EnKI-MC (II), we also considered $10$ initial iterations of EnKI-MC (I) to reduce the residuals.

For all the problems, we consider contaminating the observation $\db$ with $2\%$ additive Gaussian noise. The details on the dimensions $m$ and $n$ for each inverse problem is provided in the respective sections.  
The number of ensembles $N$
for each problem is chosen to be much smaller than the dimension $n$,  considering the practical deployability of the EnKI algorithm when the forward map $\sG$ is computationally expensive. Note that choosing a large $N$ increases the computational cost of the EnKI algorithm.

 For numerical demonstration we draw these $N$ samples from assumed distributions. We measure the accuracy in terms of the relative
error with respect to (w.r.t.) the truth defined by:
\[ \mathrm{relative\ error\ w.r.t\ truth}=\frac{\nor{\ub^*-\ub^{\dagger}}_2}{\nor{\ub^\dagger}},\]
where $\ub^*$ is the converged EnKI solution, and $\ub^\dagger$ is the ground truth solution. For each problem, we also provide comparison of our convergence acceleration strategies with other  methods. The description of these methods are provided in Appendix  \ref{other_approaches}.

\subsection{1D Deconvolution Problem}

Deconvolution, the inverse problem corresponding to  convolution process finds enormous application in signal processing and image processing. For demonstration, we consider 
the domain divided into $1000$ sub-intervals. The kernel in 1D deconvolution is constructed as \cite{mueller2012linear}:

\[\Psi(x) = C_a \LRp{x+a}^2 \LRp{x-a}^2\]
where $a = 0.235$ and the constant $C_a$ is chosen to enforce normalization condition \cite{mueller2012linear}.  
For generating a ground truth solution $\ub^\dagger$, we consider a Gaussian prior with covariance matrix of the form  $\sC=\beta \LRp{ k(X,\ X)}$ with $k(.,.)$ denoting the exponential sine squared kernel function \cite{rasmussen2006gaussian} with length scale=0.5,  periodicity=20, and $X$ denotes the $1000\times 1$ vector of equally spaced points between $[-10,\ 10]$.  The ground truth is sampled for  $\beta=0.0001$, i.e $\ub^\dagger\sim \GM{0}{\sC}$ and the observation data $\db$ is simulated based on $\ub^\dagger$. In this case the dimension $m=n=1000$. For EnKI algorithm, the number of ensembles $N$ is chosen as $20$ and each member is sampled from $\GM{0}{\sC}$.

\subsubsection{Numerical verification of Theorem \ref{non_asymp} (Non-Asymptotic Convergence of EnKF)}
\label{non_asymp_decon}
In this section, we will numerically verify Theorem \ref{non_asymp} via a sampling-based approach.  The range of $\epsilon$ where Theorem \ref{non_asymp} is valid is computed as  $0\leq \epsilon \leq 0.002$  based on the available information (we set $\eta=0.99$).

Now note that for computing the theoretical  probability in Theorem \ref{non_asymp}, one needs to know the constants $c_1,c_2$. For verifying Theorem \ref{non_asymp}, we first note the following:
\begin{enumerate}
    \item By fixing $N$, for each integer $k$ one can draw samples of  $\{\del^i_k\}_{i=1}^N$, and $\{\sigb^i_k\}_{i=1}^N$ from the respective distributions as given in Theorem \ref{non_asymp}. Thus, $(\eb_{\ub})_k$ can be computed for each $k=1,\dots K$. Here, we choose $K=100000$.
    \label{it_1}
    \item Now given $c$ and $\epsilon$, the probability that $\eb_{\ub} \le c\varepsilon$ holds can be numerically computed as:
    \[ \frac{\mathrm{No\ of\ times\ (\eb_{\ub})_k\ satisfies\ the\ inequality \ of \ all\ k=1,\dots K}}{K}.\]
    \label{it_2}
 \item In addition,  if we know constants $c_1,\ c_2,\ c$, the minimum probability  that $\eb_{\ub} \le c\varepsilon$ can be theoretically computed using the bound in Theorem \ref{non_asymp}.
  \label{it_3}
\end{enumerate}
Note that since constants $c_1,\ c_2$ are not known, numerical verification of Theorem \ref{non_asymp} is not possible in general.   However, we attempt to roughly verify Theorem \ref{non_asymp} by showing that there exists constants $c,\ c_1,\ c_2$ for the given $\sC,\ \Sigma,\ \A,\ \ub_0,\ \db$ such that the theoretically predicted probability in Theorem \ref{non_asymp} is indeed a lower bound for all values of $N,\ \epsilon$.  
For this, we consider five  choices for $N$ (i.e. $5,\ 10,\ 15,\  20,\ 25$) and, with $\epsilon=0.002$, use the sampling-based procedure outlined above  to regress approximate values for $c,\ c_1,\ c_2$\footnote{Even though Theorem \ref{non_asymp} provides an estimate for the constant $c$, this is not necessarily the best  possible value, and therefore, we regress for $c$ as well to obtain a better practical value.}  that minimize the difference between the numerically computed and theoretically computed probabilities, while satisfying the constraint that the numerically computed probability  is greater then the theoretical probability.

 The estimated constants $c,\ c_1,\ c_2$ are then used to compute the theoretical vs. numerical probability for all different cases of $N$ and  $0\leq \epsilon \leq 0.002$. The results are shown in Figure \ref{non_asymp_decon}. From Figure \ref{non_asymp_decon} it is evident that there exists constants $c_1,\ c_2,\ c$  such that the theoretically predicted probability is indeed a lower bound for all values of $N,\ \epsilon$. However, one drawback of this procedure is that  the estimated constants $c_1,\ c_2$ are dependent on $\sC,\ \Sigma, \ etc$ whereas the true underlying constants $c_1,\ c_2$ are absolute positive constants.
\begin{figure}[h!]      
  % \begin{subfigure}[b]{\textwidth}
\hspace{-0.5 cm}
  \begin{tabular}{c}

      \begin{tabular}{c}

          \centering
          \includegraphics[scale=0.38]{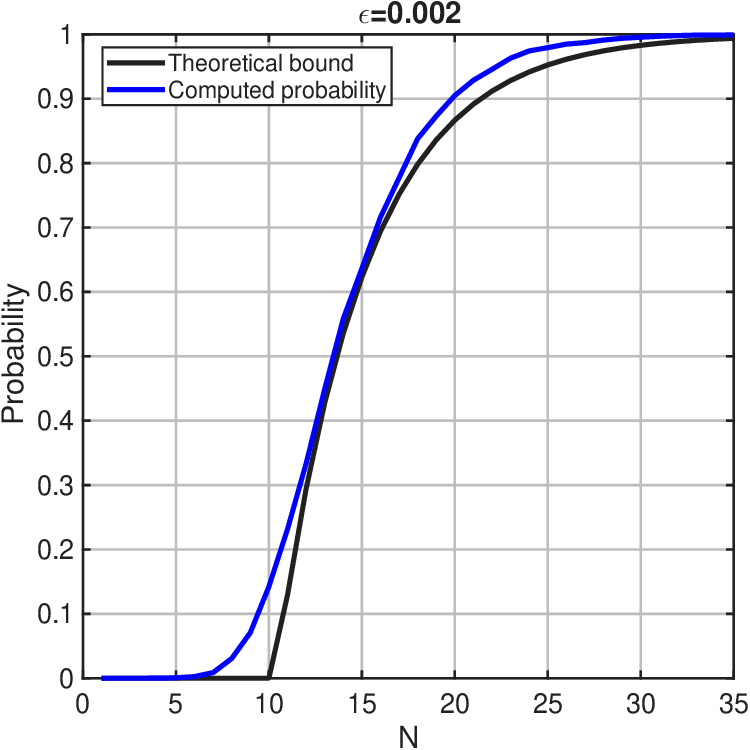}

      \end{tabular}

    \hspace{-0.3 cm}

      \begin{tabular}{c}

          \centering
          \includegraphics[scale=0.38]{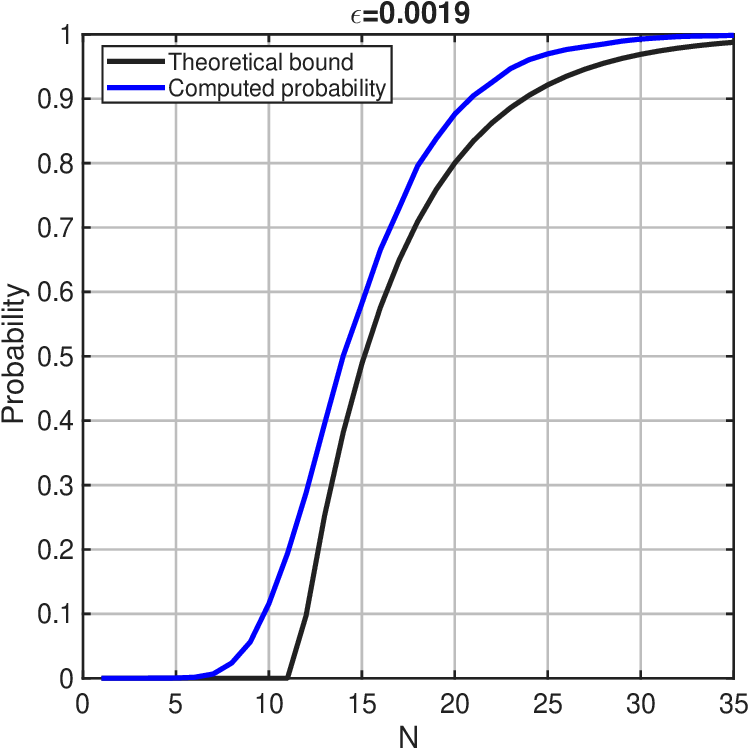}

      \end{tabular}

    \hspace{-0.3 cm}

      \begin{tabular}{c}

          \centering
          \includegraphics[scale=0.38]{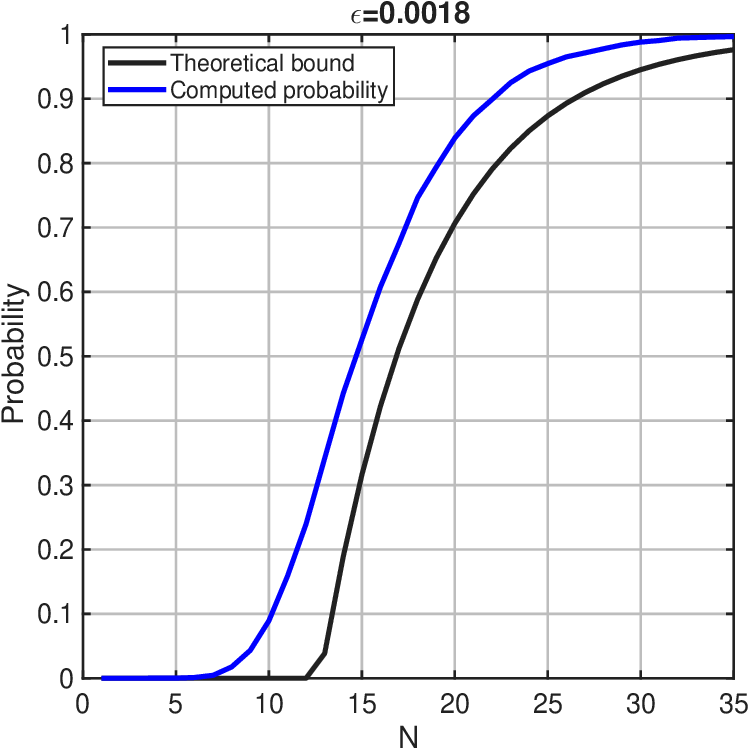}

      \end{tabular}

  \end{tabular}\\
   
    \begin{tabular}{c}
\hspace{-0.5 cm}
      \begin{tabular}{c}

          \centering
          \includegraphics[scale=0.38]{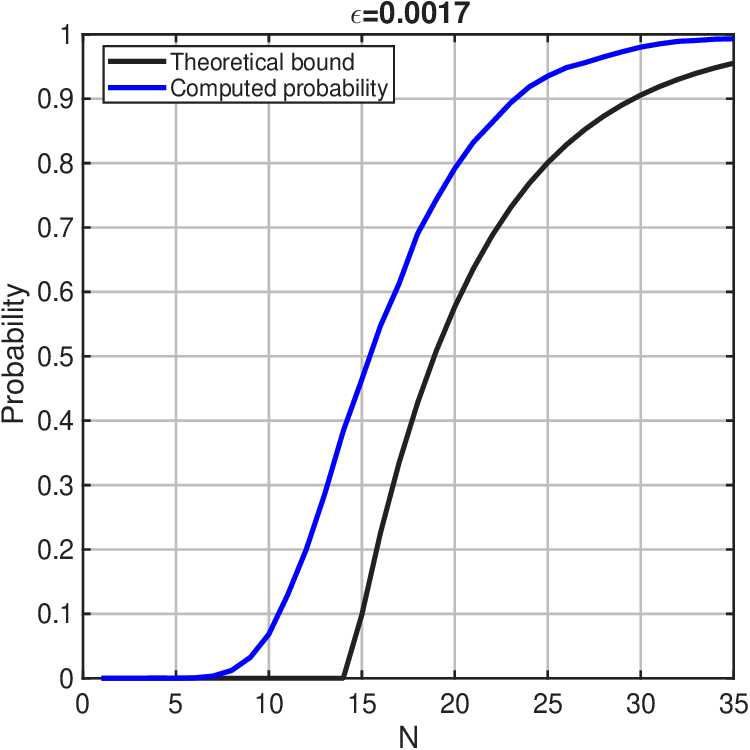}

      \end{tabular}

    \hspace{-0.3 cm}

      \begin{tabular}{c}

          \centering
          \includegraphics[scale=0.38]{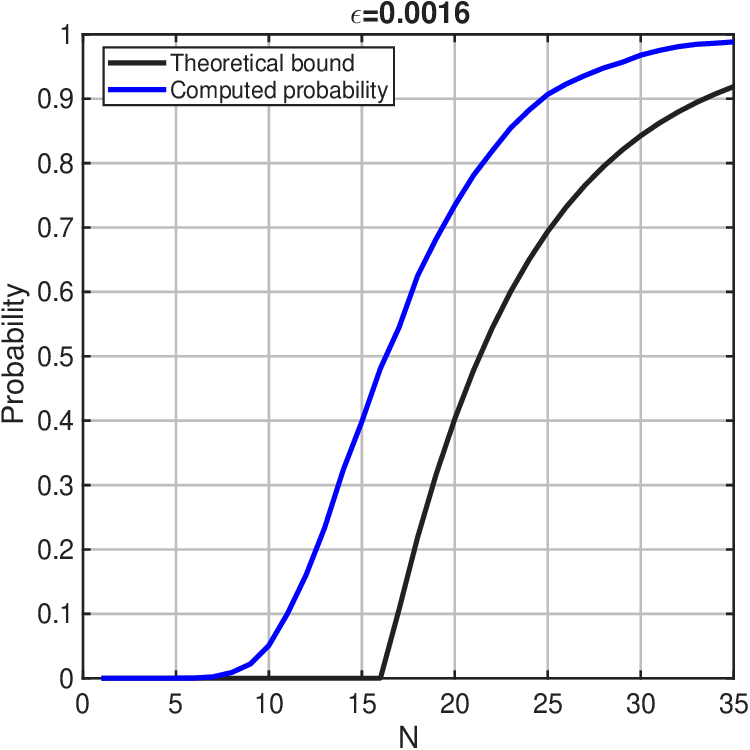}

      \end{tabular}

    \hspace{-0.3 cm}

      \begin{tabular}{c}

          \centering
          \includegraphics[scale=0.38]{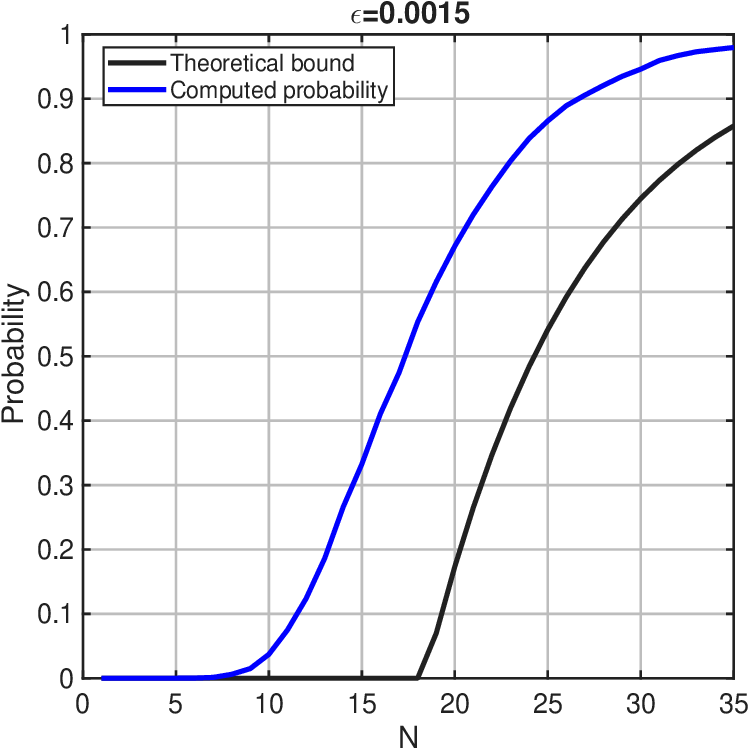}

      \end{tabular}

  \end{tabular}
    
    \caption{Numerical verification of Theorem \ref{non_asymp} (non-asymptotic convergence): Theoretically predicted probability is lower than the numerically computed probability for different values of $\epsilon$ and $N$. }
  \label{non_asymp_decon}
  % \end{subfigure} \\
\end{figure}  

\subsubsection{Convergence acceleration with   EnKI-MC (I) and EnKI-MC (II)}

In this section, we demonstrate the effectiveness of EnKI-MC (I) and EnKI-MC (II) in accelerating the convergence of Vanilla EnKI algorithm.  
Figure \ref{ENKF_20_samples_conv_valid} (left) shows the evolution of the optimal $\alpha_k$ with respect to iteration $k$ for EnKI-MC (I). The reason for the increasing-decreasing nature of $\alpha_k$ has been explained in detail in remark \ref{curve_nature}. Further, Figure \ref{ENKF_20_samples_conv_valid} (right) shows the evolution of the optimal $\alpha_k^{(i)}$ computed using \eqref{approx_alpha_II} for EnKI-MC (II). We see the step pattern in Figure \ref{ENKF_20_samples_conv_valid} (right) since $\alpha_k^{(i)}$ is only recomputed every $K=5$ iterations (refer  Algorithm  \ref{Algo_full_II}) for computational efficiency.

\begin{figure}[h!]
\hspace{-0.5 cm}
\begin{minipage}{.5\textwidth}
 \includegraphics[scale=0.38]{Figures/1D_Deconv/alpha_EnKI_MC_I.pdf}
\end{minipage}%
\begin{minipage}{.5\textwidth}
    \includegraphics[scale=0.38]{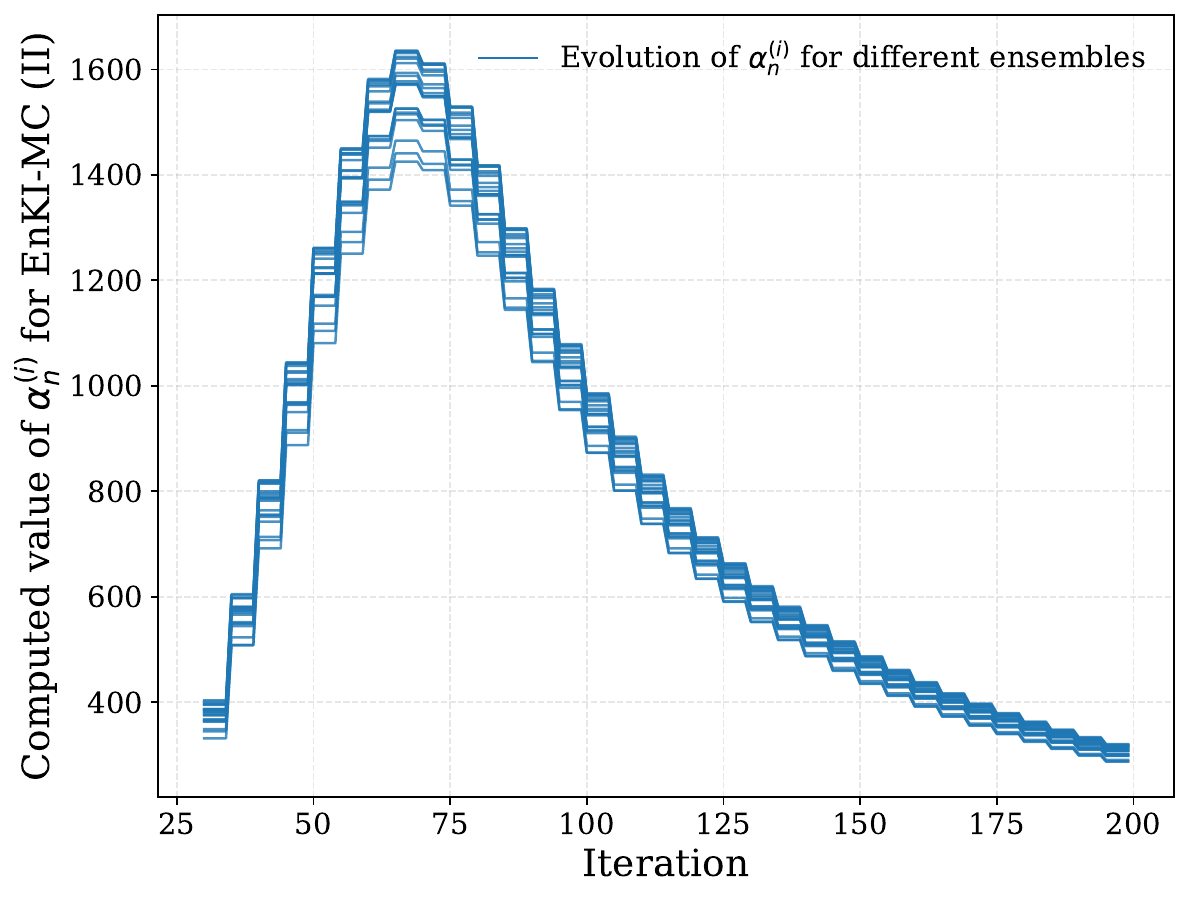}
\end{minipage}
\caption{Left to right: Optimal correction factor computed at each iteration of EnKI-MC (I) using \eqref{iterative_conver} and the corresponding approximation \eqref{approx_alpha}; Optimal correction factor computed at each iteration of EnKI-MC (II) using \eqref{approx_alpha_II} for all ensembles (The step pattern is due to  $\alpha_n^i$ being recomputed only every $K=5$ iterations).}
\label{ENKF_20_samples_conv_valid}
\end{figure}

Figure \ref{ENKF_dec_log} shows the convergence of different variants of EnKI algorithm. Figure \ref{ENKF_dec_log} (left) shows the convergence of relative error with respect to truth where we observe that EnKI-MC (I) and EnKI-MC (II) converges the fastest out of all the approaches.  Figure \ref{ENKF_dec_log} (right) shows a similar trend for the convergence of the loss $\sJ_{EnKI}$ defined in \eqref{enki_loss_approx}. Note that all the approaches use the same criteria and tolerance (line 2 in Algorithm \ref{Algo_full}) for termination of the algorithm. From Table \ref{num_result1} we see that EnKI-MC (II) is the best performing method providing a speedup of about $2.7$ times (in terms of computational time) over Chada et al. \cite{chada2019convergence} (the next best method in terms of convergence speed). We also see that EnKI-MC (II) provides the best relative error  of $0.100$ with respect to the truth. In addition to convergence improvement, we also observed that even when continuing the iterations way beyond the exit criteria  (line 2 in Algorithm \ref{Algo_full}), none of the approaches were able to produce the best relative error produced by EnKI-MC (II) at any iterations. The reason for this accuracy improvement in EnKI-MC (II) needs further theoretical investigation and will be conducted in the future. Further, we also observed that while methods such as Chada et al. \cite{chada2019convergence} and Nesterov Acceleration \cite{vernon2025nesterov} provides significant speedup over vanilla EnKI, these methods were prone to overfitting issue, i.e the relative error w.r.t truth started increasing significantly after some iterations while the loss continued to decrease. Especially, Chada et al. \cite{chada2019convergence} with it's monotonically increasing nature of the correction factor (i.e $\alpha_k=k^{0.8}$) led to significant loss of regularization inside the EnKI framework as the iteration proceeds (refer remark \ref{stability_remark}, and remark \ref{curve_nature} for more explanation). Whereas, EnKI-MC (I) and EnKI-MC (II) due to it's inbuilt nature of $\alpha_k$ decreasing after some iteration, led to reimposing strong regularization inside the EnKI iterations and less prone to overfitting leading to  better quality solution upon termination of the algorithm. In section \ref{init_cond_inv} (Initial Condition Inversion in an Advection-Diffusion Problem), these effects are much more pronounced and are clearly visible in Figure \ref{ENKF_IC_log}.

Figure \ref{ENKF_20_samples_inflation_I} also shows the inverse solution achieved by different methods where we see that all the methods did produce good solutions for this problem.

\begin{table}[h!]
\caption{Performance comparison of different Ensemble Kalman Inversion Algorithms for 1D Deconvolution problem.}
\centering
\begin{tabular}{|c | c | c |c|c|c| }
        \hline
     & Time to reach  & No of &    Relative error   \\ 
      & ($\epsilon$ tolerance) & iterations & w.r.t ground truth    \\  \hline
   EnKI-MC (I)  & $11$ min& 319 & 0.105\\ \hline
         EnKI-MC (II)  &$9$ min& 291 & 0.100   \\ \hline
Chada et al. \cite{chada2019convergence} &$25$ min &$1897$ & 0.107  \\  \hline
Vanilla EnKI, \eqref{new_ss} &$39$ min  &3087 &0.111  \\  \hline
Nesterov Acceleration \cite{vernon2025nesterov}  &$28$ min &2223 & 0.111 \\  \hline
Anderson et al. \cite{anderson2007adaptive} &$83$ min & 3086 &0.111  \\  \hline
\end{tabular} 
 \label{num_result1}
\end{table}
\begin{figure}[h!]
\centering
\begin{minipage}{.5\textwidth}
 \includegraphics[scale=0.38]{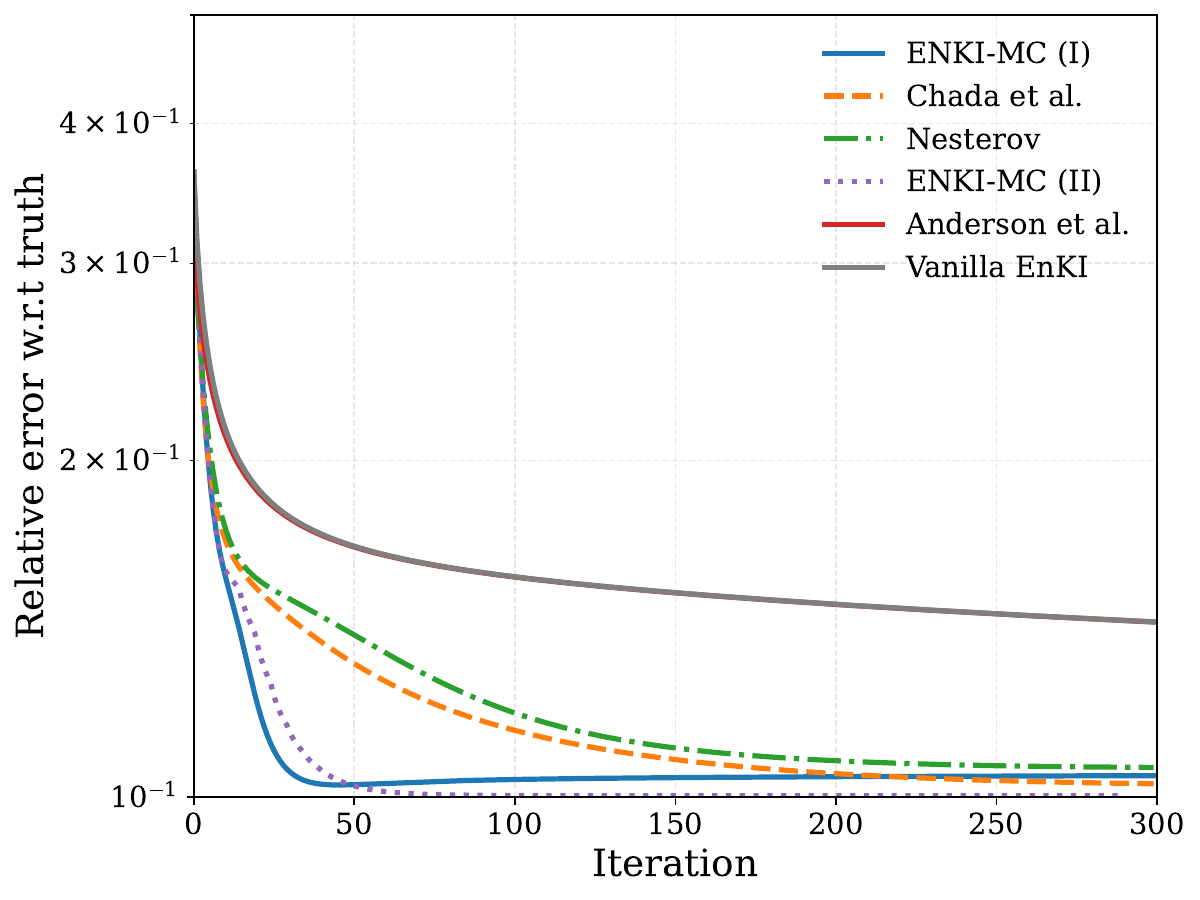}
\end{minipage}%
\begin{minipage}{.5\textwidth}
\hspace{-0.1 cm}
 \includegraphics[scale=0.38]{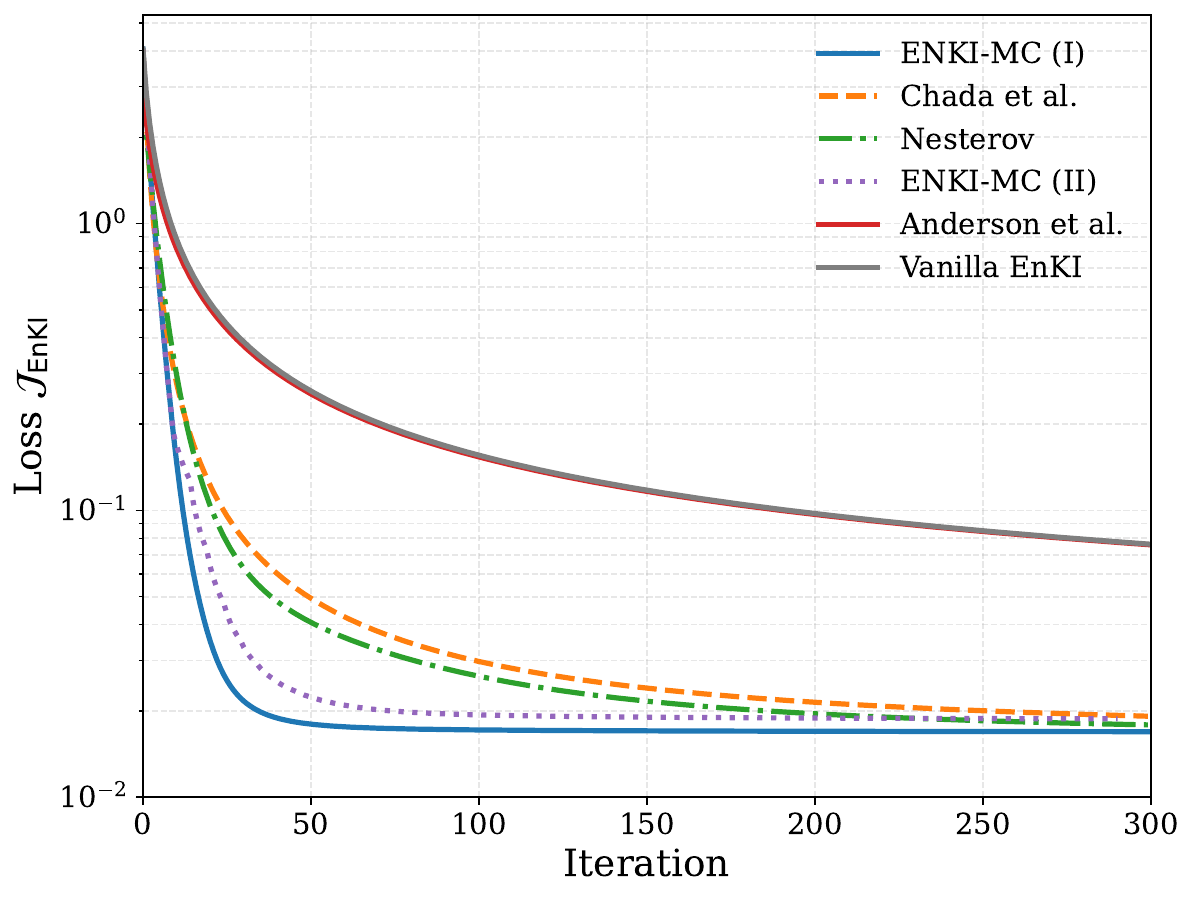}
\end{minipage}
\caption{Performance of different variants of EnKI  for the deconvolution Problem. Left to right: Convergence of relative error w.r.t truth for different methods;  Convergence of loss defined in \eqref{enki_loss_approx}.}
\label{ENKF_dec_log}
\end{figure}

\begin{figure}[h!]
\centering
 \includegraphics[scale=0.38]{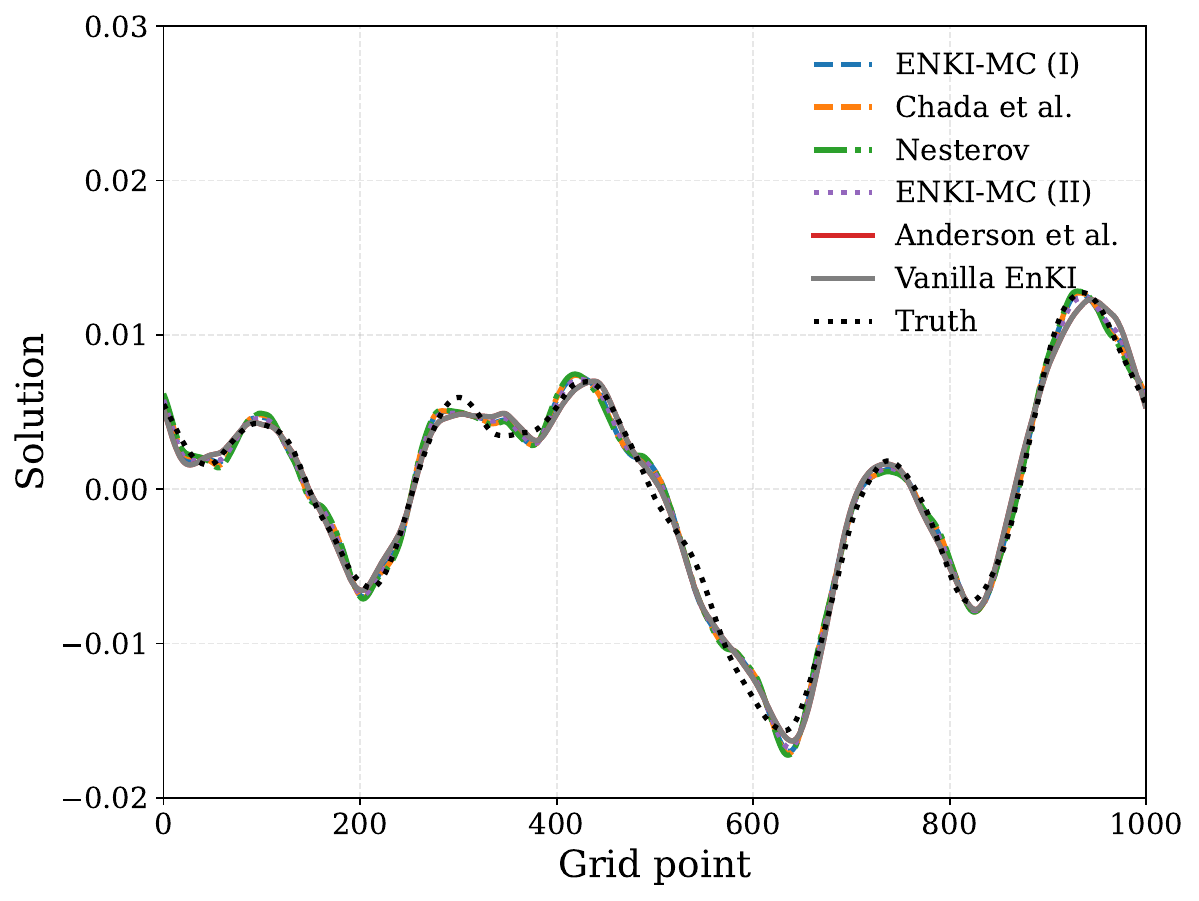}
\caption{Inverse solution achieved by different EnKI variants. }
\label{ENKF_20_samples_inflation_I}
\end{figure}

\subsection{Initial Condition Inversion in an Advection-Diffusion Problem}
\label{init_cond_inv}
In this section we consider a linear inverse problem governed by a parabolic PDE \cite{AustinMerced2017}.
The parameter to observable map   maps an initial condition $u \in L^2(\Omega)$ to pointwise spatiotemporal observations of the concentration field $y({\bf{x}},t)$, where ${\bf{x}}=(x_1,\ x_2)$. The solution of the advection-diffusion equation is given by:

\begin{equation}
    \begin{aligned}
        y_t - \kappa\Delta y + {{\bf{v}}} \cdot \nabla y &= 0 \quad {in } \quad \Omega \times(0,T),\\
 y(.,0)&=u\quad {in } \quad \Omega,\\
 \kappa \nabla y \cdot \bf{n}&=0\quad {on } \quad \partial \Omega \times (0,T). 
 \label{advection_diffusion}  
    \end{aligned}
\end{equation}
where, $\Omega\subset \R^2$ is a bounded domain, $\kappa>0$ is the diffusion coefficient, $T>0$ is the final time. The velocity field ${\bf{v}}$ is computed by solving the following steady-state Navier-Stokes equation with the side walls driving the flow:
\begin{equation}
    \begin{aligned}
         -\frac{1}{\operatorname{Re}} \Delta {{\bf{v}}} + \nabla q + {{\bf{v}}} \cdot \nabla {\bf{v}} &= 0 \quad{ in }\quad \Omega,\\
 \nabla \cdot {\bf{v}} &= 0 \quad { in } \quad \Omega,\\
 {\bf{v}} &= {\bg} \quad { on } \quad  \partial\Omega. 
    \label{velocity_field}
    \end{aligned}
\end{equation}
where, $q$ is the pressure, and $\operatorname{Re}$ is the Reynolds number.
The Dirichlet boundary condition ${\bg}\in \R^2$ is prescribed as ${\bg}={\bf{e}}_2$ (canonical basis vector) on the left side of the domain,
and $\bg={\bf{0}}$ elsewhere. Velocity boundary conditions are not prescribed on the right side of the boundary. 
\begin{figure}[h!]      
  \begin{tabular}{c}
  \hspace{-0.5 cm}
      \begin{tabular}{c}
          \centering
          \includegraphics[scale=0.35]{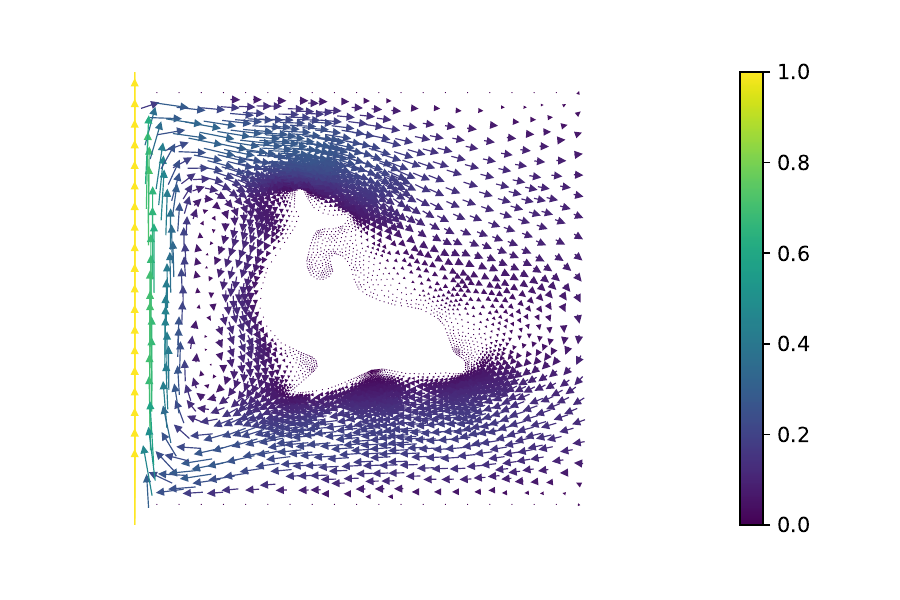}
      \end{tabular}
    \hspace{-1.2 cm}
      \begin{tabular}{c}
          \centering
          \includegraphics[scale=0.3]{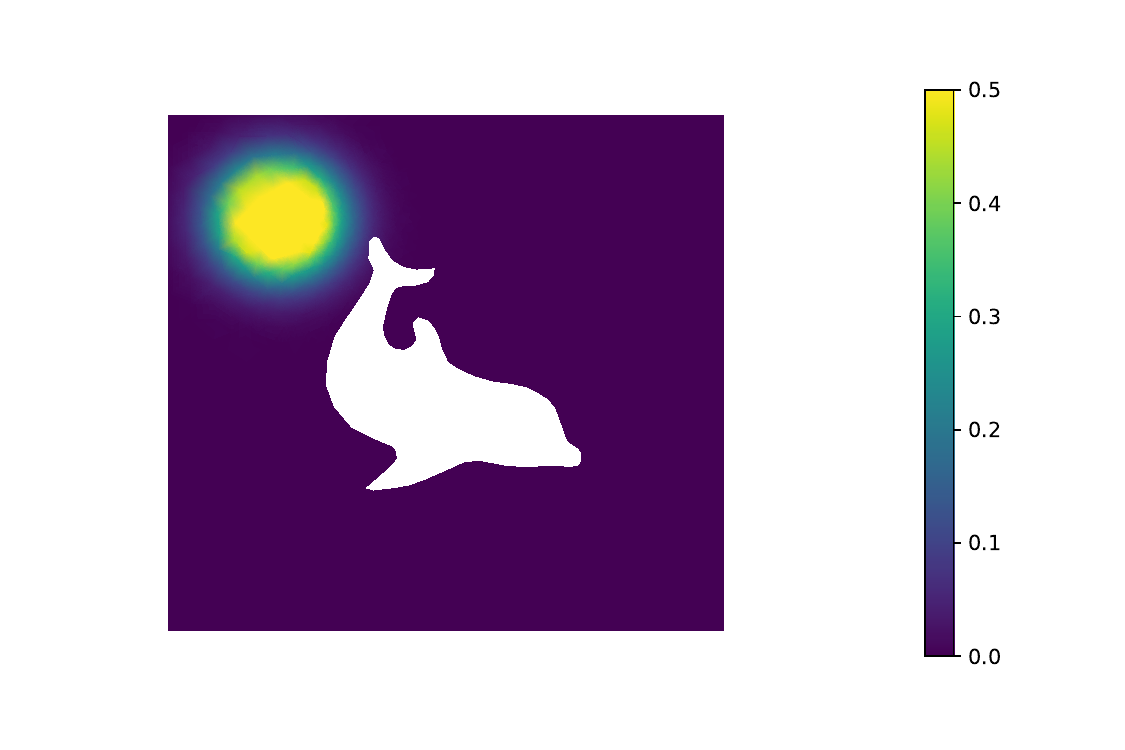}
      \end{tabular}
    \hspace{-1.2 cm}
      \begin{tabular}{c}
          \centering
          \includegraphics[scale=0.3]{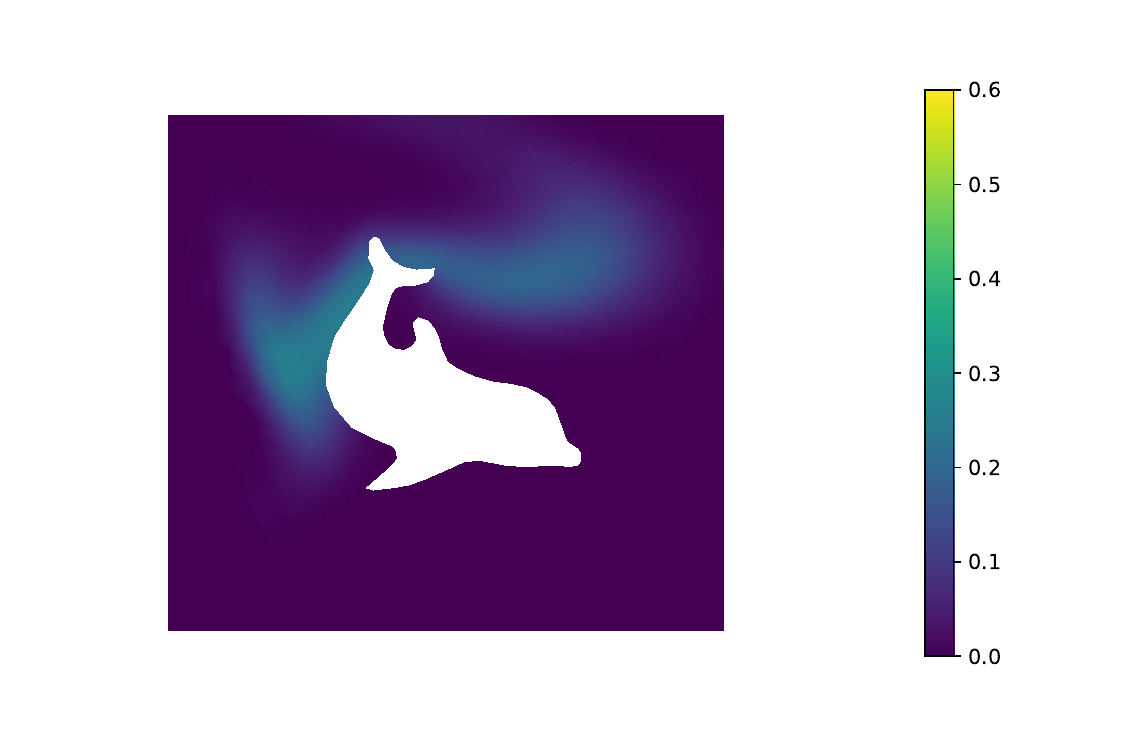}
      \end{tabular}
  \end{tabular}
\caption{ From left to right: The velocity profile; Initial condition;  Observation data used for inversion.  }
\label{velocity_profile}
  % \end{subfigure} \\
\end{figure}     
The values of the forward solution $y({\bf{x}},t)$ on a set of locations $\{ {\bf{x}}_1, {\bf{x}}_2, ..., {\bf{x}}_m\}$ at the final time $T$ are extracted
and used as the observation vector $\bd\in \R^m$ for solving the initial condition inverse problem.
The  velocity profile, initial condition and the observation data used in the study are shown in Figure \ref{velocity_profile}.
The forward operator $\A$ maps the initial condition  $\ub \in \R^n$ ($n$ is the number of points on the grid) to the observation $\bd\in \R^m$. The objective is to infer the unknown initial condition  $\ub \in \R^n$ given noise contaminated observation field $\bd\in \R^m$.

In \eqref{advection_diffusion}, the diffusion coefficient is set as  $\kappa=0.001$. Following \cite{AustinMerced2017}, a mixed formulation employing $P2$ Lagrange elements for approximating the velocity field and $P1$ elements for pressure is adopted
for solving (\ref{velocity_field}) to get the velocity field.
The computed velocity field is then used to solve the advection-diffusion equation, (\ref{advection_diffusion}).
$P1$ Lagrange elements are used for the variational formulation of the advection-diffusion equation. We generate ensembles for the EnKI algorithm from a Gaussian prior, $ \GM{0.25}{\sC}$ where $\sC$ is the BiLaplacian prior given by:
\begin{equation}
  \sC=\LRp{\zeta I+\gamma \nabla \cdot (\theta \nabla)}^{-2}  
  \label{Bi_Laplacian}
\end{equation}
where $\zeta$ governs the variance of the samples,  the ratio $\frac{\gamma}{\zeta}$ governs the correlation length, $\theta$ is a symmetric positive definite tensor to introduce anisotropy in the correlation length. We choose  $\zeta=8$ and $\gamma=1$.  The ground truth is sampled from $ \GM{0.25}{\sC}$. The observation vector $\db$ is computed at the end of time $t=3s$ with $m=200$ observation points (random points in the domain).
For the present problem, we fix $n = 2868$. The number of ensembles $N$ is chosen as $80$ and each member is sampled from $ \GM{0.25}{\sC}$.

\subsubsection{Numerical verification of Theorem \ref{non_asymp} (Non-Asymptotic Convergence of EnKF)}

For this problem, we see that Theorem \ref{non_asymp} for non-asymptotic convergence is valid only for $0\leq \epsilon \leq 5.042 \times 10^{-16}$ computed based on the available information, and therefore a numerical verification of Theorem \ref{non_asymp} is not conducted for this example.

\subsubsection{Convergence acceleration with  EnKI-MC (I) and EnKI-MC (II)}
In this section, we demonstrate the effectiveness of EnKI-MC (I) and EnKI-MC (II) in accelerating the convergence of Vanilla EnKI algorithm. 

\begin{figure}[h!]
\centering
\begin{minipage}{.5\textwidth}
 \includegraphics[scale=0.38]{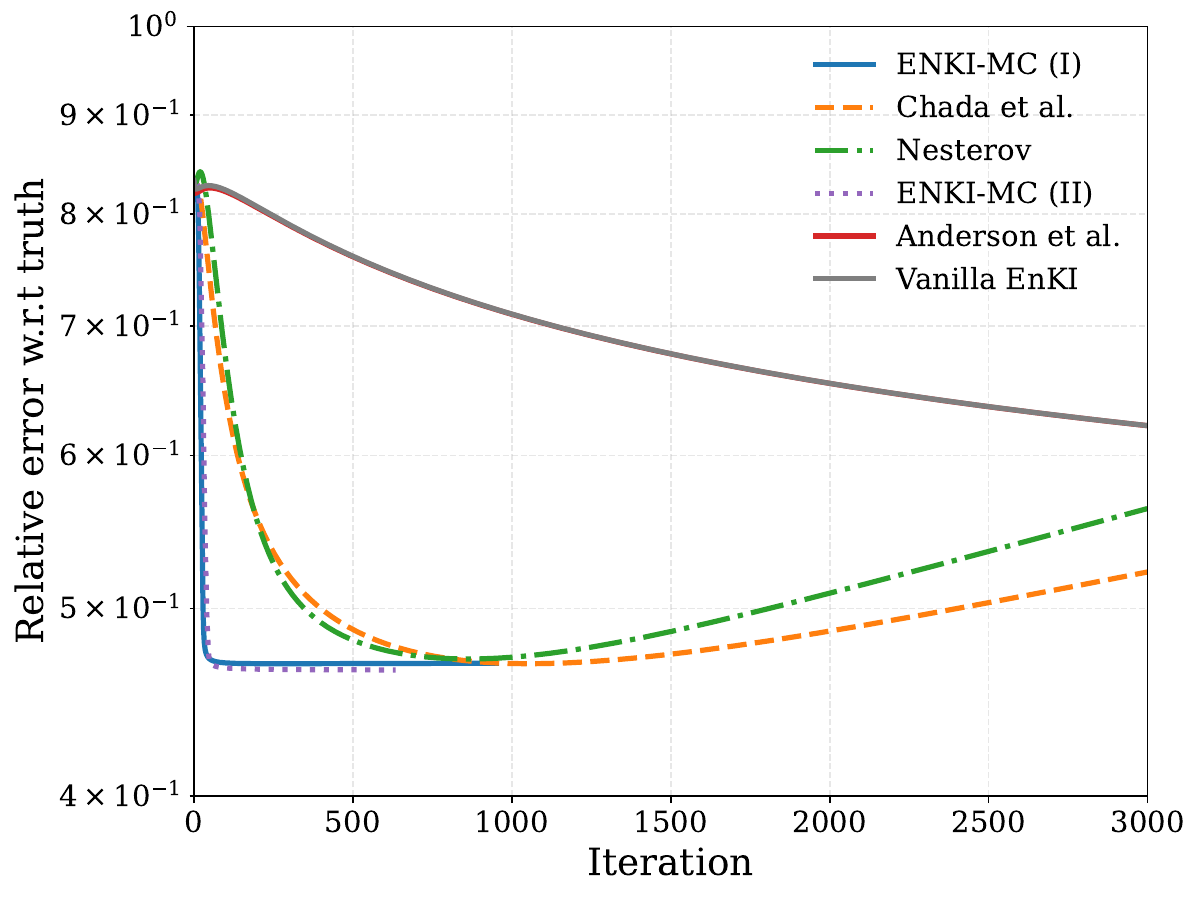}
\end{minipage}%
\begin{minipage}{.5\textwidth}
\hspace{-0.1 cm}
 \includegraphics[scale=0.38]{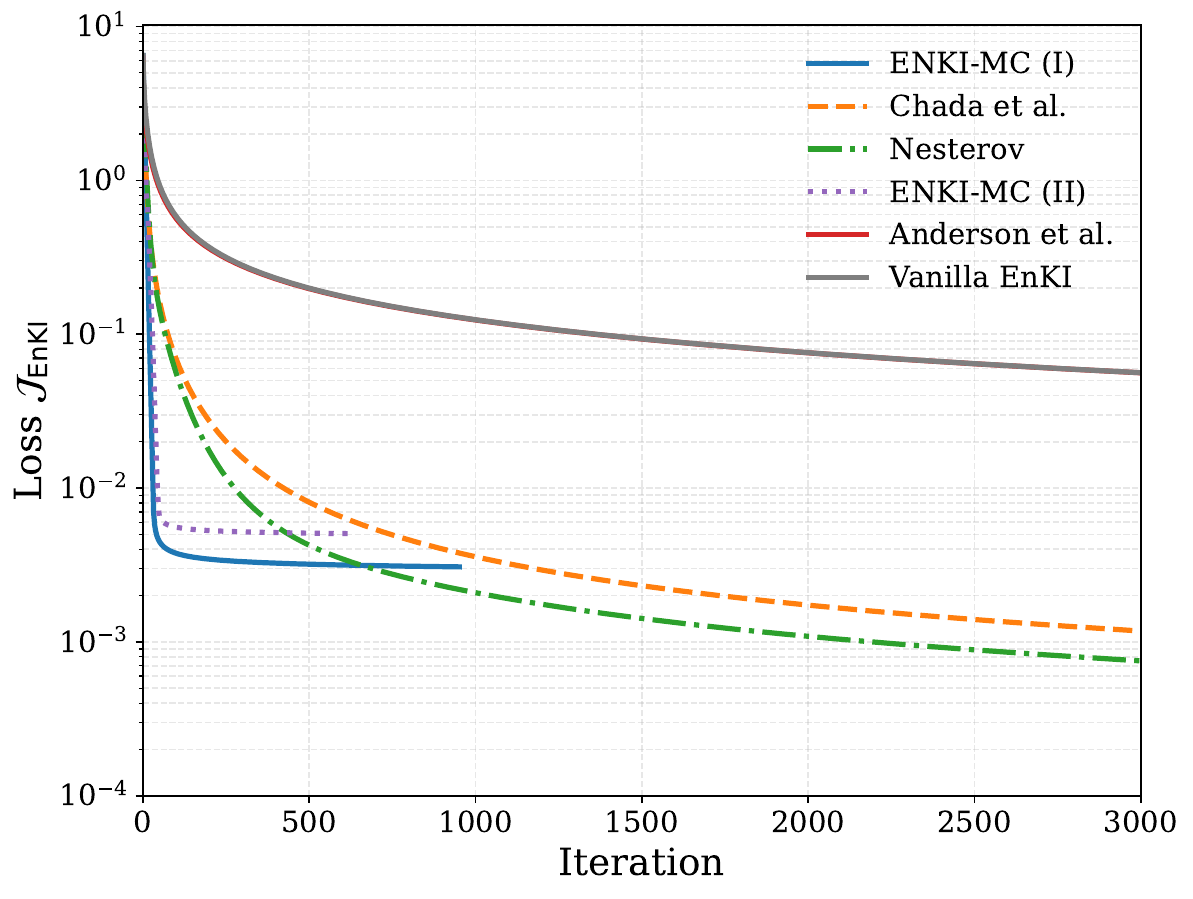}
\end{minipage}
\caption{Performance of different variants of EnKI  for the Advection-Diffusion Problem. Left to right: Convergence of relative error w.r.t truth for different methods;  Convergence of loss defined in \eqref{enki_loss_approx}.}
\label{ENKF_IC_log}
\end{figure}

Figure \ref{ENKF_IC_log} (left) shows the convergence of relative error with respect to truth where we observe that EnKI-MC (I) and EnKI-MC (II) converges the fastest out of all the approaches.  Figure \ref{ENKF_IC_log} (right) shows a similar trend for the convergence of the loss $\sJ_{EnKI}$ defined in \eqref{enki_loss_approx}. 

Moreover, we observe that methods such as those of Chada et al.~\cite{chada2019convergence} and Nesterov acceleration~\cite{vernon2025nesterov} suffer from significant overfitting: the relative error with respect to the truth begins to increase after approximately $1000$ iterations, even though the loss continues to decrease. Experiments with noise-free data $\db$ show that this divergence from the true solution persists for both Chada et al.~\cite{chada2019convergence} and Nesterov acceleration~\cite{vernon2025nesterov}, indicating that the observed overfitting is not caused by data noise.
Instead, this behavior arises because there exists a solution $\ub^* \in \sD$  (subspace spanned by the initial ensembles) that fits the misfit \eqref{enki_loss_approx} better than the true solution $\ub^\dagger$, leading to divergence from the latter. One possible remedy is to increase the number of initial ensemble members so that the subspace $\sD$ contains the true solution $\ub^\dagger$, although this significantly increases computational cost. Alternatively, the EnKI algorithm \eqref{new_ss} can be modified to incorporate additional regularization, such as the variants proposed in \cite{chada2020tikhonov,chada2019convergence}. Applying our proposed convergence improvement strategies to such variants will be explored in future work.

On the other hand, our approach seems to be immune to the overfitting issue  as evident from Figure \ref{ENKF_IC_log}. This is attributed to the fact that collapse of ensembles causes a rapid decrease in the computed $\alpha_n$ which leads to reimposing strong regularization inside the EnKI iterations (remark \ref{curve_nature}). Note that all the methods uses the same termination criteria to exit the algorithm (line 2 in Algorithm \ref{Algo_full}). 
\begin{table}[h!]
\caption{Performance comparison of different Ensemble Kalman Inversion Algorithms for initial condition inversion in an Advection-Diffusion Problem.}
\centering
\begin{tabular}{|c | c | c |c|c| }
        \hline
     & Time to reach  & No of &    Relative error \\ 
      & ($\epsilon$ tolerance) & iterations & w.r.t ground truth  \\  \hline
   EnKI-MC (I)  & $5$ min& 952 & 0.468\\ \hline
         EnKI-MC (II)  &$9$ min & 635& 0.465\\ \hline
Chada et al. \cite{chada2019convergence}  &$31$ min &10000 & 0.767   \\  \hline
Vanilla EnKI, \eqref{new_ss} &$31$  min & 10000  &0.542 \\  \hline
Nesterov Acceleration \cite{vernon2025nesterov}  & $32$ min &10000 & 0.936  \\  \hline
Anderson et al. \cite{anderson2007adaptive} & $59$ min & 10000&0.542 \\  \hline
\end{tabular} 
 \label{num_result2}
\end{table}
Table \ref{num_result2} shows that EnKI-MC (I) is the best performing method in terms of convergence time providing a speedup of about $6$ times  over Chada et al. \cite{chada2019convergence} (the next best method in terms of convergence speed). 
Table \ref{num_result2} also shows that EnKI-MC (II) provides the best relative error  of $0.465$ with respect to the truth. We saw that none of the approaches were able to produce the best relative error produced by EnKI-MC (II) at any iteration of the algorithm (even after continuing way beyond the termination criteria). In addition, we also considered an experiment where we decreased the tolerance $\epsilon_c$ in line 1 of Algorithm \ref{Algo_full} from $10^{-5}$ to $10^{-6}$.
We observed that while EnKI-MC (I) and EnKI-MC (II) showed negligibly small increase in the relative error w.r.t truth, methods such as 
Nesterov Acceleration \cite{vernon2025nesterov}  and Chada et al. \cite{chada2019convergence} showed significant divergence from the truth.

\begin{figure}[h!]
  \hspace{-0.8 cm}
    % ===== Row 1 =====
    \begin{tabular}{ccc}
        \includegraphics[scale=0.3]{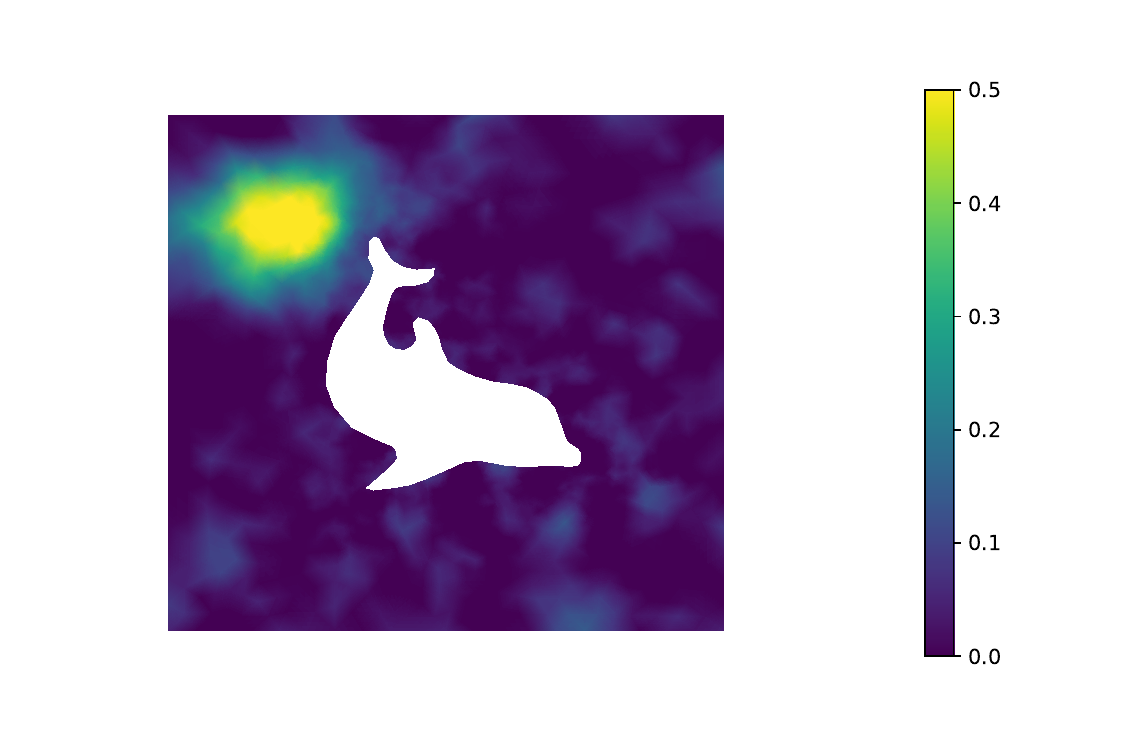} &
      \hspace{-1 cm}  \includegraphics[scale=0.3]{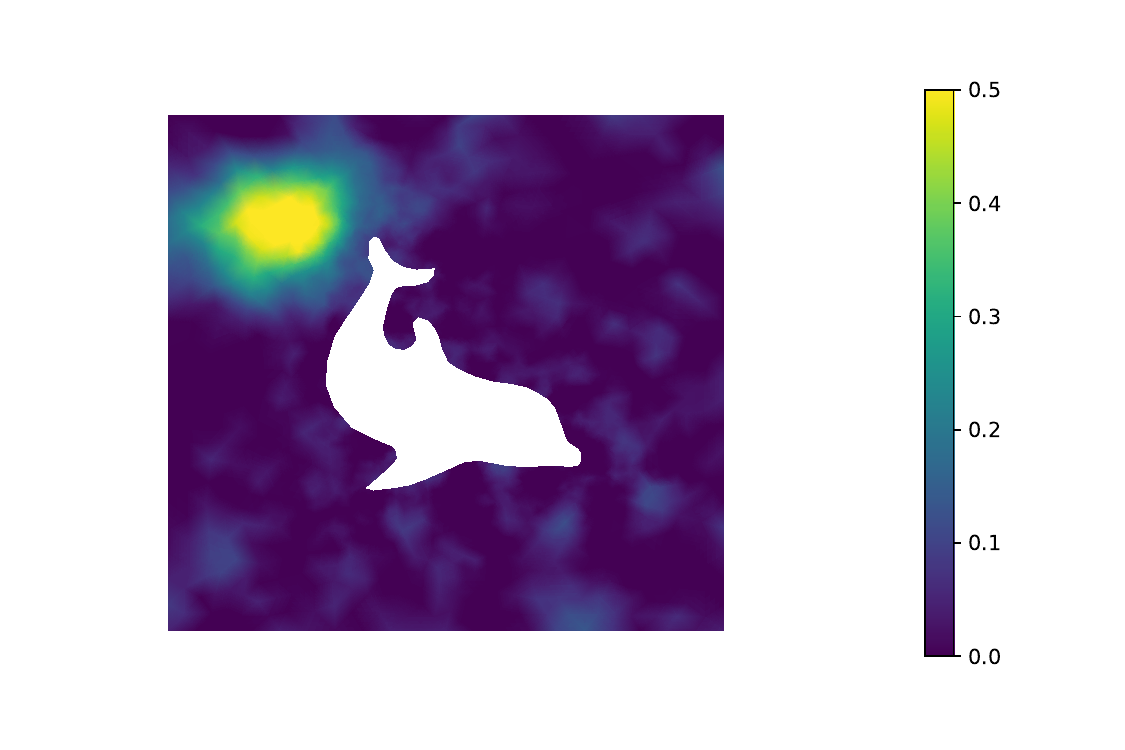} &
        \hspace{-1 cm}  \includegraphics[scale=0.3]{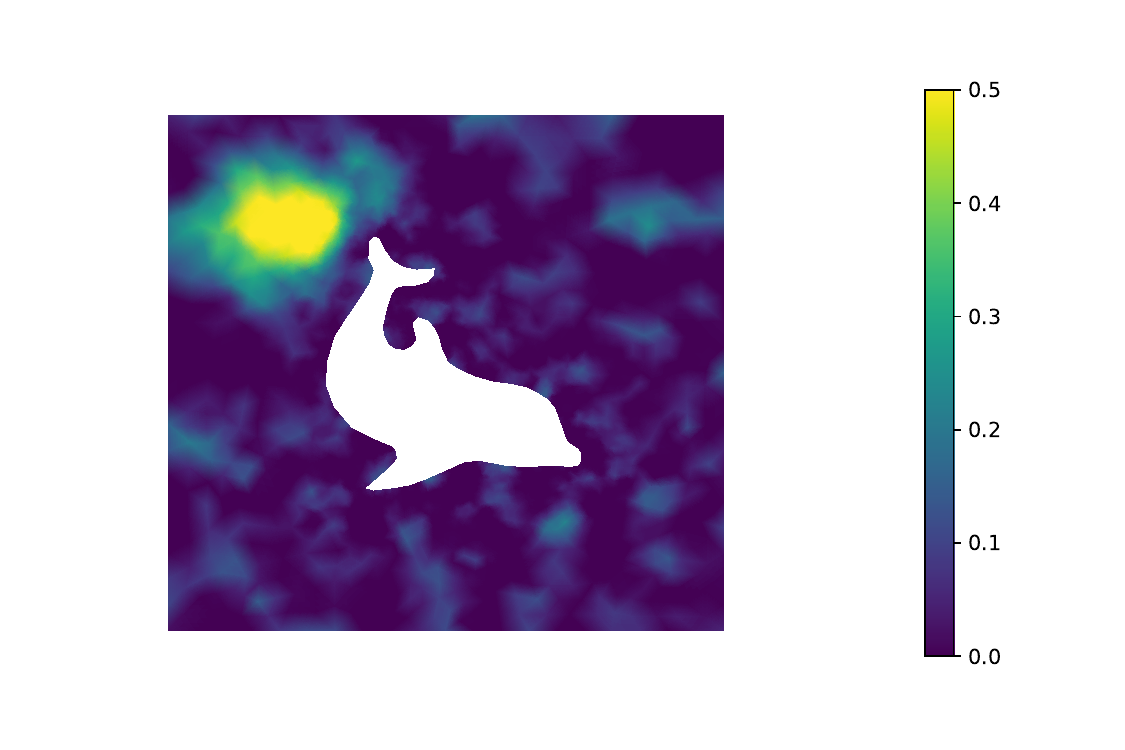} \\
    \end{tabular}

    \vspace{0.3cm} % vertical space between rows

    % ===== Row 2 =====
    \hspace{-0.8 cm}
    \begin{tabular}{ccc}
         \includegraphics[scale=0.3]{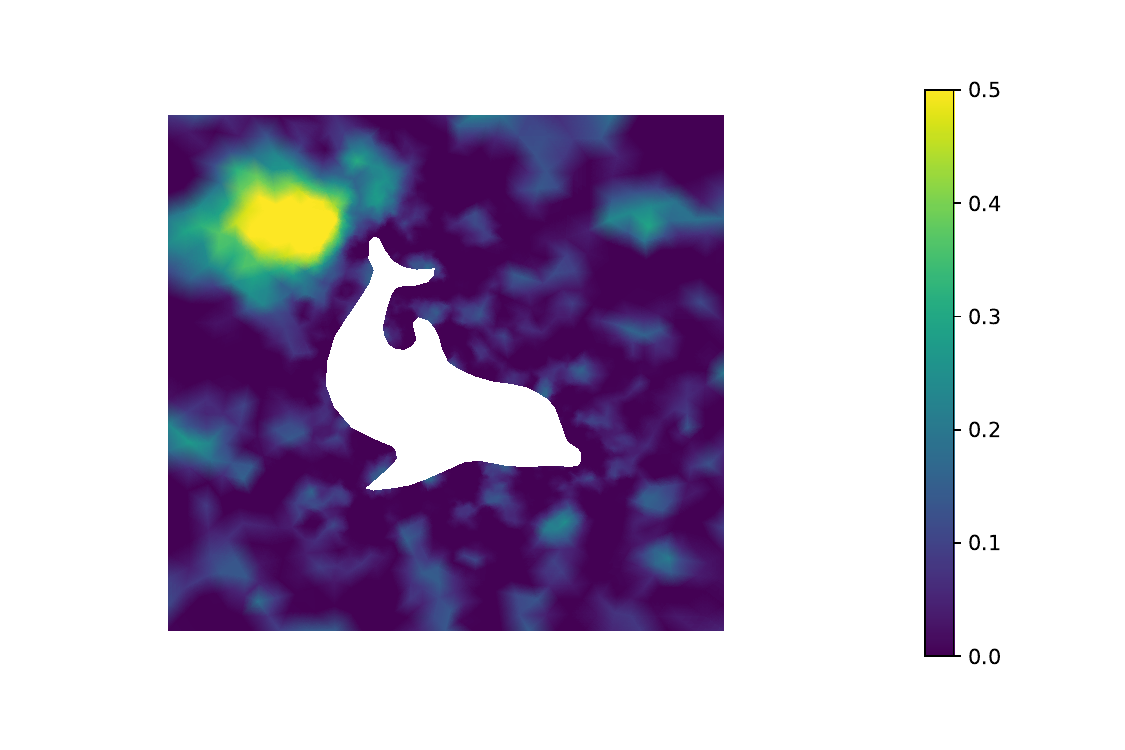}&
          \hspace{-1 cm}\includegraphics[scale=0.3]{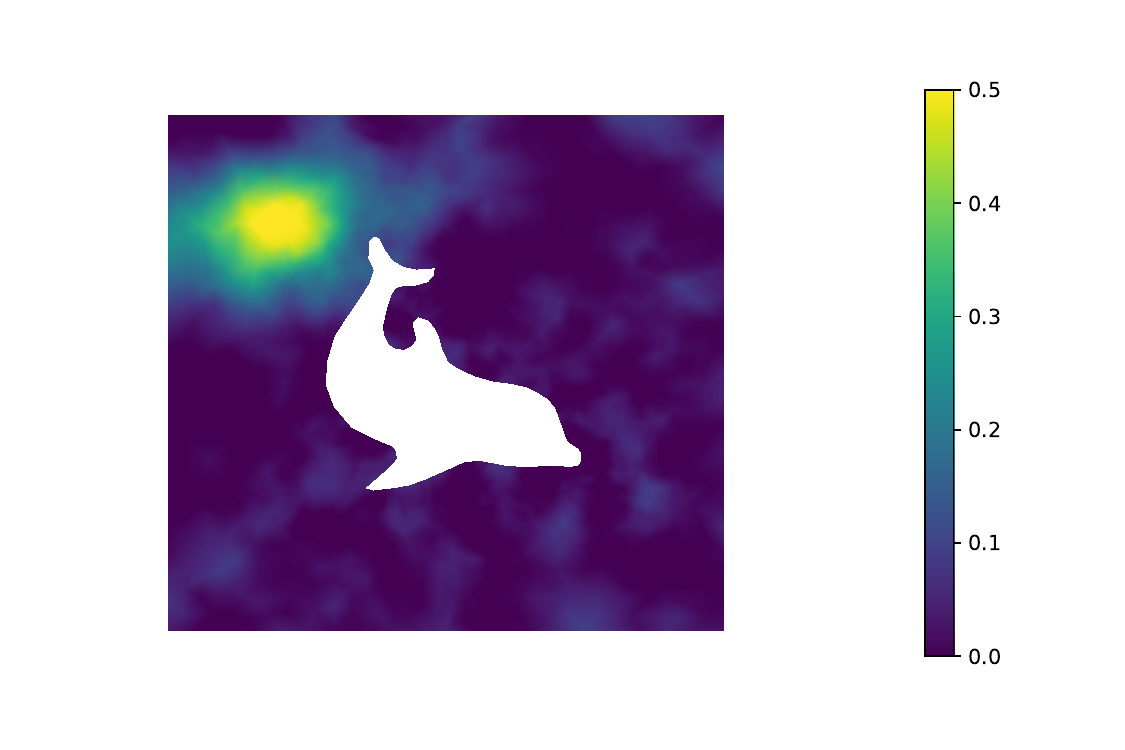} &
       \hspace{-1 cm}  \includegraphics[scale=0.3]{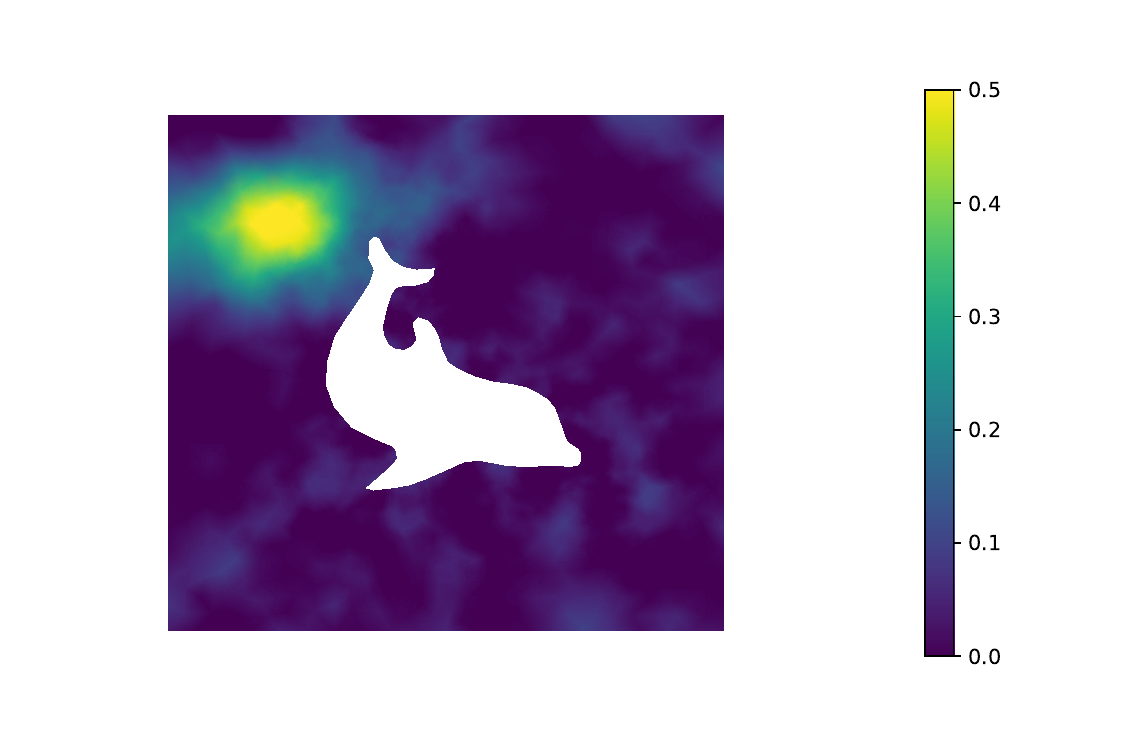} \\
    \end{tabular}

    \vspace{0.3cm} % vertical space between rows

    % ===== Row 3 =====
    \hspace{4.42 cm}
    \begin{tabular}{c}
       \includegraphics[scale=0.3]{Figures/Initial_Condition/Dophin_true.pdf} \\
    \end{tabular}

    \caption{Solutions computed using different variants of EnKI for the Advection-Diffusion Problem. 
    Top row (left to right): EnKI-MC (I); EnKI-MC (II); Chada et al. \cite{chada2019convergence}. \\
    Middle row: Nesterov Acceleration \cite{vernon2025nesterov}; Anderson et al. \cite{anderson2007adaptive}; Vanilla EnKI. \\
   Third row: True solution.}
   \label{EnKI_MC_IC}
\end{figure}
\begin{figure}[h!]
  \hspace{-0.8 cm}
    % ===== Row 1 =====
    \begin{tabular}{ccc}
        \includegraphics[scale=0.3]{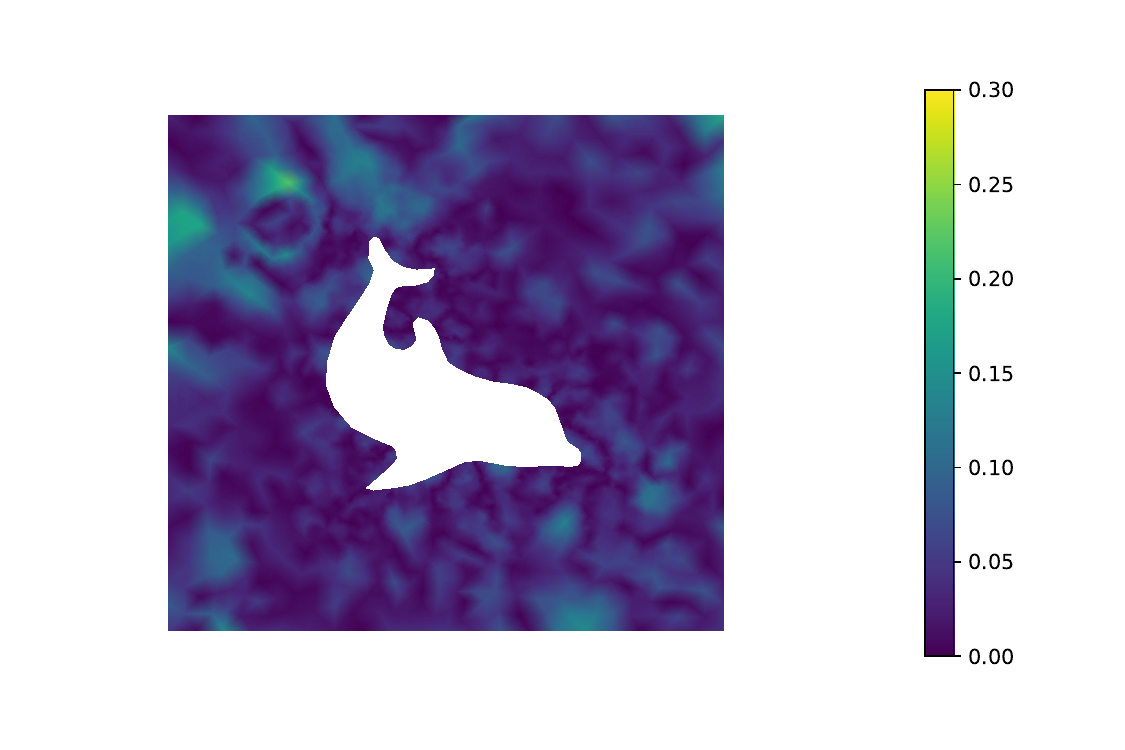} &
      \hspace{-1 cm}  \includegraphics[scale=0.3]{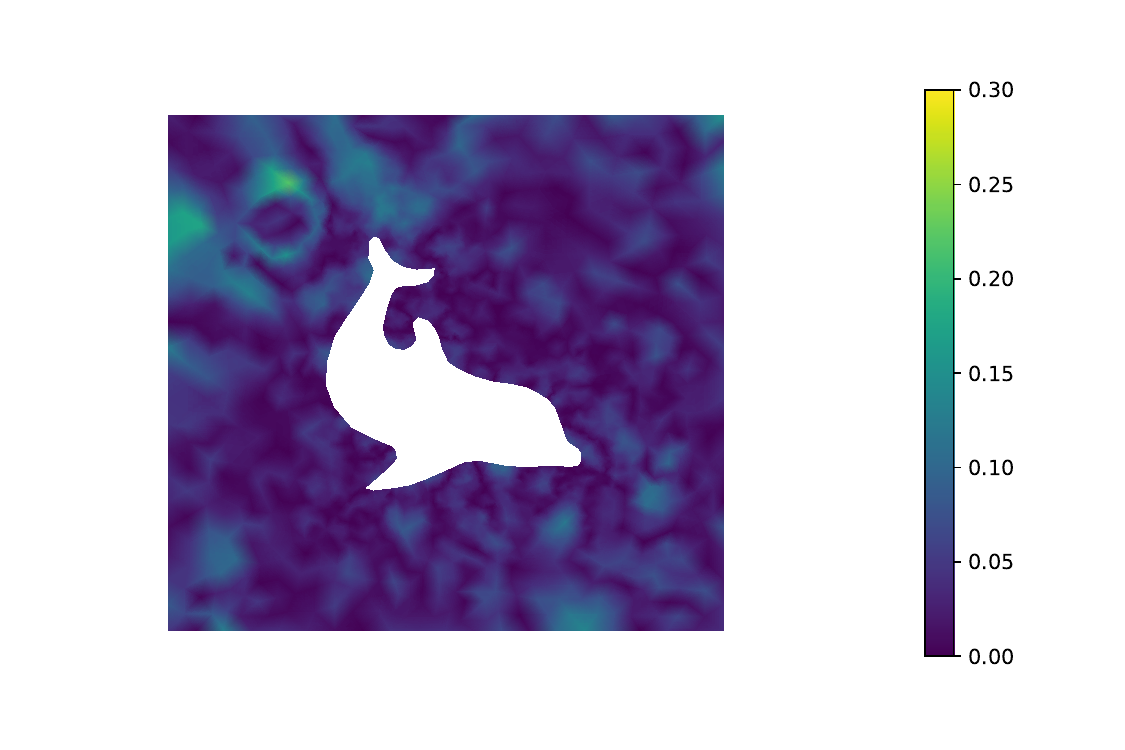} &
        \hspace{-1 cm}  \includegraphics[scale=0.3]{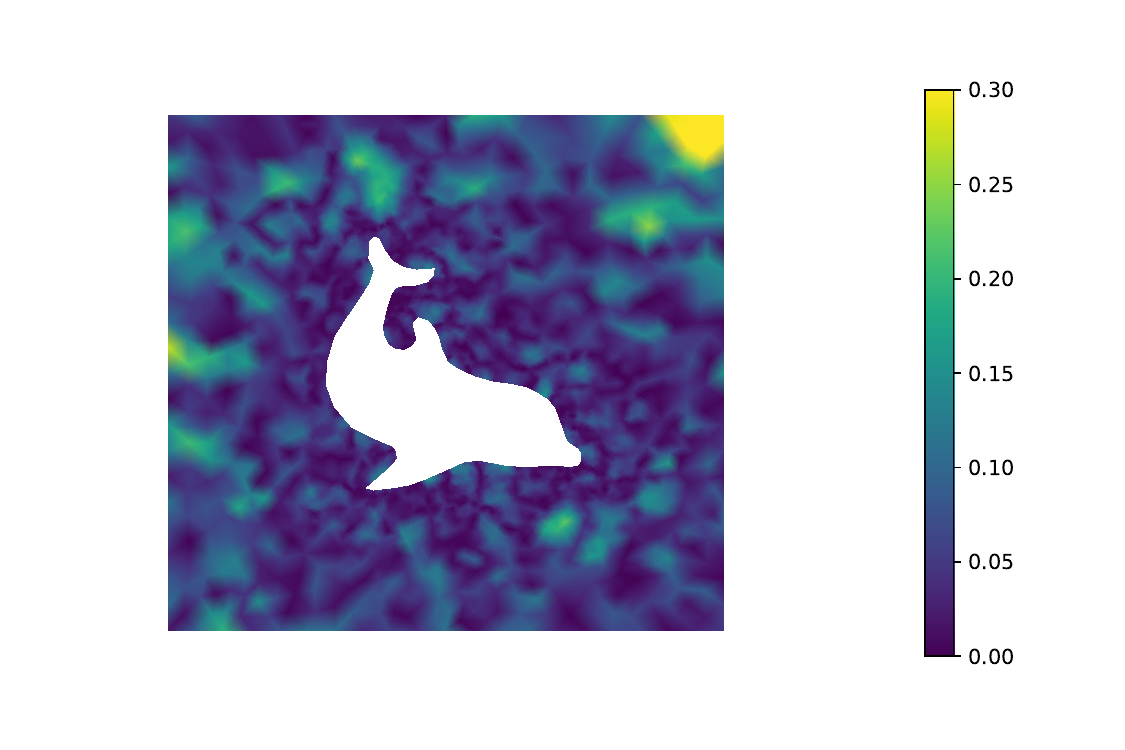} \\
    \end{tabular}

    \vspace{0.3cm} % vertical space between rows

    % ===== Row 2 =====
    \hspace{-0.8 cm}
    \begin{tabular}{ccc}
         \includegraphics[scale=0.3]{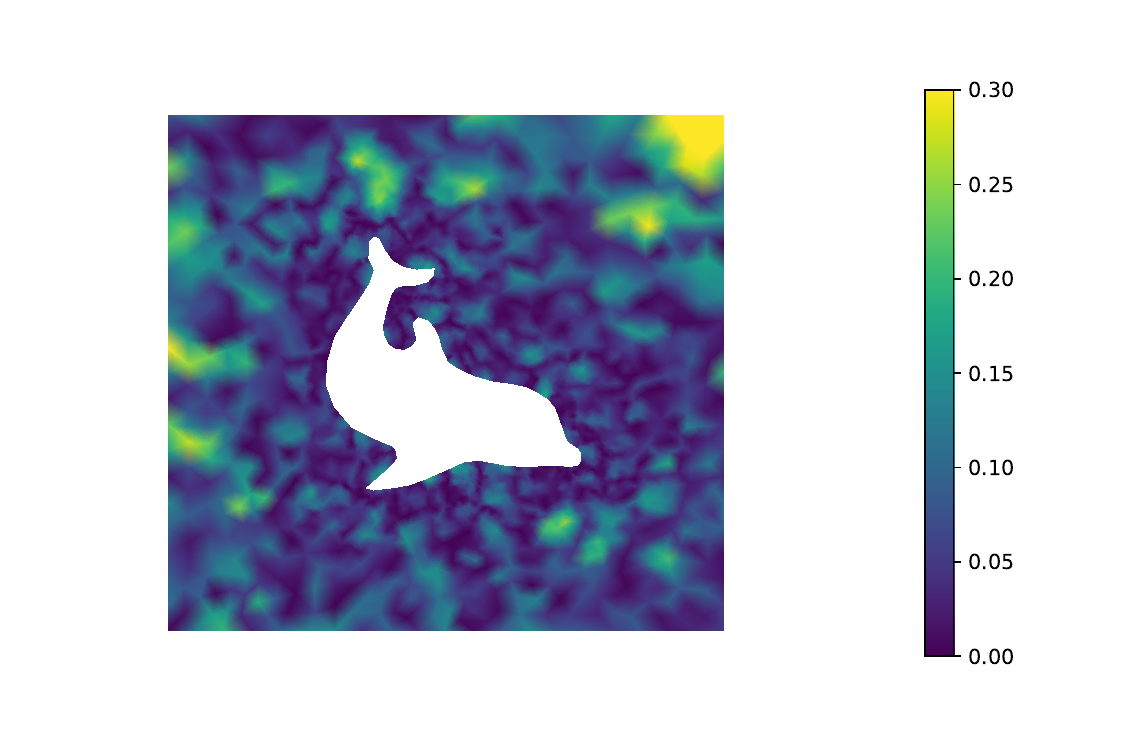}&
          \hspace{-1 cm}\includegraphics[scale=0.3]{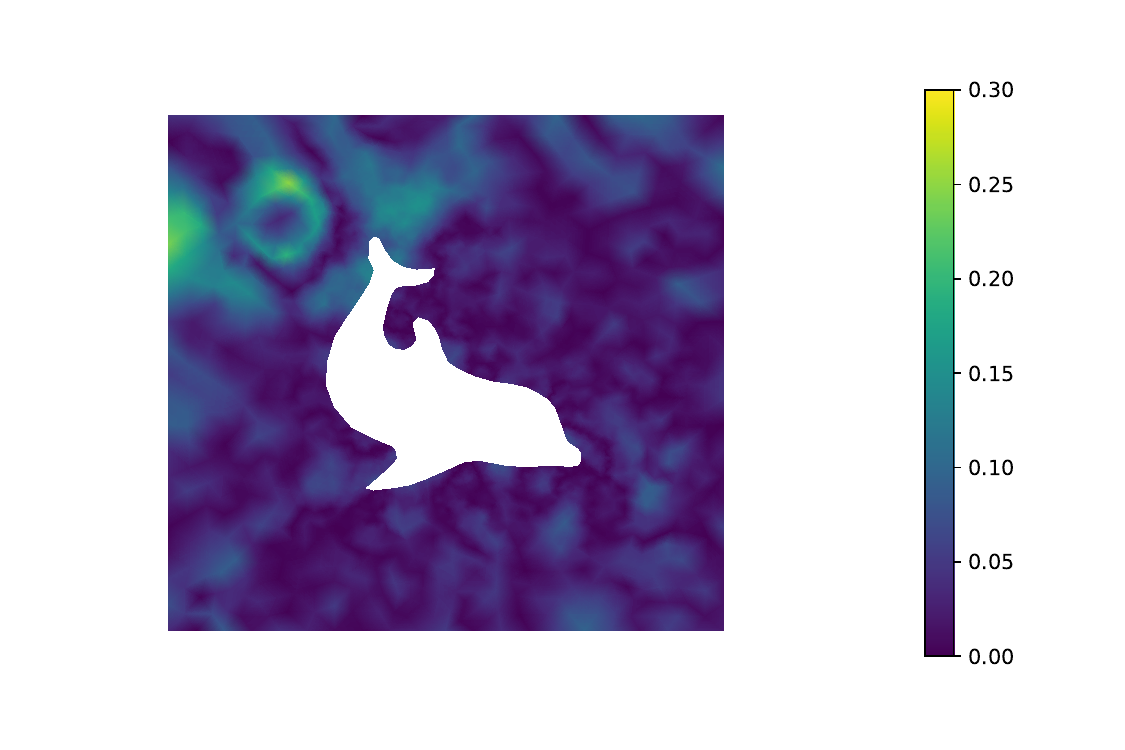} &
       \hspace{-1 cm}  \includegraphics[scale=0.3]{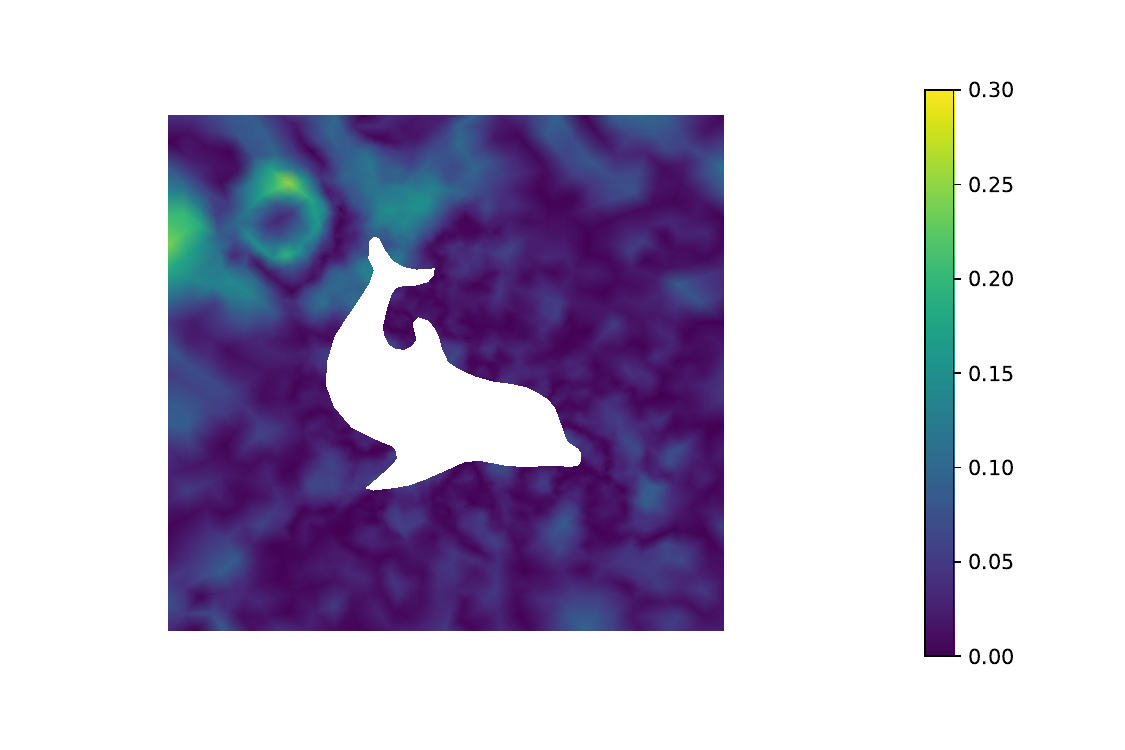} \\
    \end{tabular}

    \caption{Absolute error in solutions with respect to truth computed for different variants of EnKI for the Advection-Diffusion Problem. 
    Top row (left to right): EnKI-MC (I); EnKI-MC (II); Chada et al. \cite{chada2019convergence}. \\
    Middle row: Nesterov Acceleration \cite{vernon2025nesterov}; Anderson et al. \cite{anderson2007adaptive}; Vanilla EnKI.}
   \label{EnKI_MC_IC_error}
\end{figure}
  The inverse solution obtained by different variants of EnKI are shown in Figure \ref{EnKI_MC_IC}. Further, the absolute error  in computed solutions with respect to truth is also shown in Figure \ref{EnKI_MC_IC_error}. It is clear from Figure \ref{EnKI_MC_IC_error} that EnKI-MC (I) and EnKI-MC (II) produces the best solution out of all approaches.

\subsection{Initial condition inversion  in a Lorenz 96 model}
In this section we consider a nonlinear inverse problem involving the Lorenz 96 (L96) model \cite{chada2019convergence}.  The L96 model is described by the following system of coupled ordinary differential equations:
\begin{equation}
\begin{aligned}
      \frac{d v_k}{dt}=&v_{k-1}\LRp{v_{k+1}-v_{k-2}}-v_k+F,\quad k=1,\dots n\\
      &v_0=v_n,\ \ v_{n+1}=v_1,\ \ v_{-1}=v_{n-1},
\end{aligned}
\label{lorenz}
\end{equation}
where $v_k$ denotes the current state of the system at the $k-$th grid point. $F$ is a forcing constant with default value $8$. The initial condition to the system is denoted as $\vb(0)=\LRp{v_1(0),\dots, v_n(0)}^T$. The inverse problem is to recover the initial condition $\vb(0)$ given noise contaminated measurements $\{v_{k_r}(t)\}_{r=1}^m$, where $k_r$ is an integer between $1$ and $n$. That is we measure $m$ random components of $\vb(t)$. The measurement is made at time $t=0.3$.

In \eqref{lorenz}, we set $n=500$. The number of unique measurements locations on the grid is set as $315$. We also considered repeated measurements (with different noise) for some of the grid points giving rise to the final value of $m$ as $500$.
For generating a ground truth initial condition $\vb(0)$, we considered a  Gaussian process with mean$=2$ and an exponential sine squared kernel function  i.e $\vb(0)\sim \GM{{\bf{2}}}{\sC}$, where $\sC$ is the covariance matrix associated with the exponential sine squared kernel function with length scale=0.5,  periodicity=20. A small zero mean Gaussian noise (with standard deviation $0.01$) is added to the generated $\vb(0)$ to introduce minor fluctuations, as the L96 system is sensitive to perturbations and the inverse problem of recovering $\vb(0)$ becomes challenging.
The ensembles are drawn from $\GM{{\bf{2}}}{\sC}$ and we set the number of ensembles $N$ to be $500$. To solve the L96 model \eqref{lorenz} we
use a fourth order Runge–Kutta method with step size $h_{L96}=0.01$.
\subsubsection{Convergence acceleration with  EnKI-MC (I) and EnKI-MC (II)}
Figure \ref{lorentz} (left) shows the convergence of relative error with respect to truth where we observe that EnKI-MC (I) and EnKI-MC (II) converges the fastest out of all the approaches.  Figure \ref{lorentz} (right) shows a similar trend for the convergence of the loss $\sJ_{EnKI}$ defined in \eqref{enki_loss_approx}.  
\begin{figure}[h!]
\centering
\begin{minipage}{.5\textwidth}
 \includegraphics[scale=0.38]{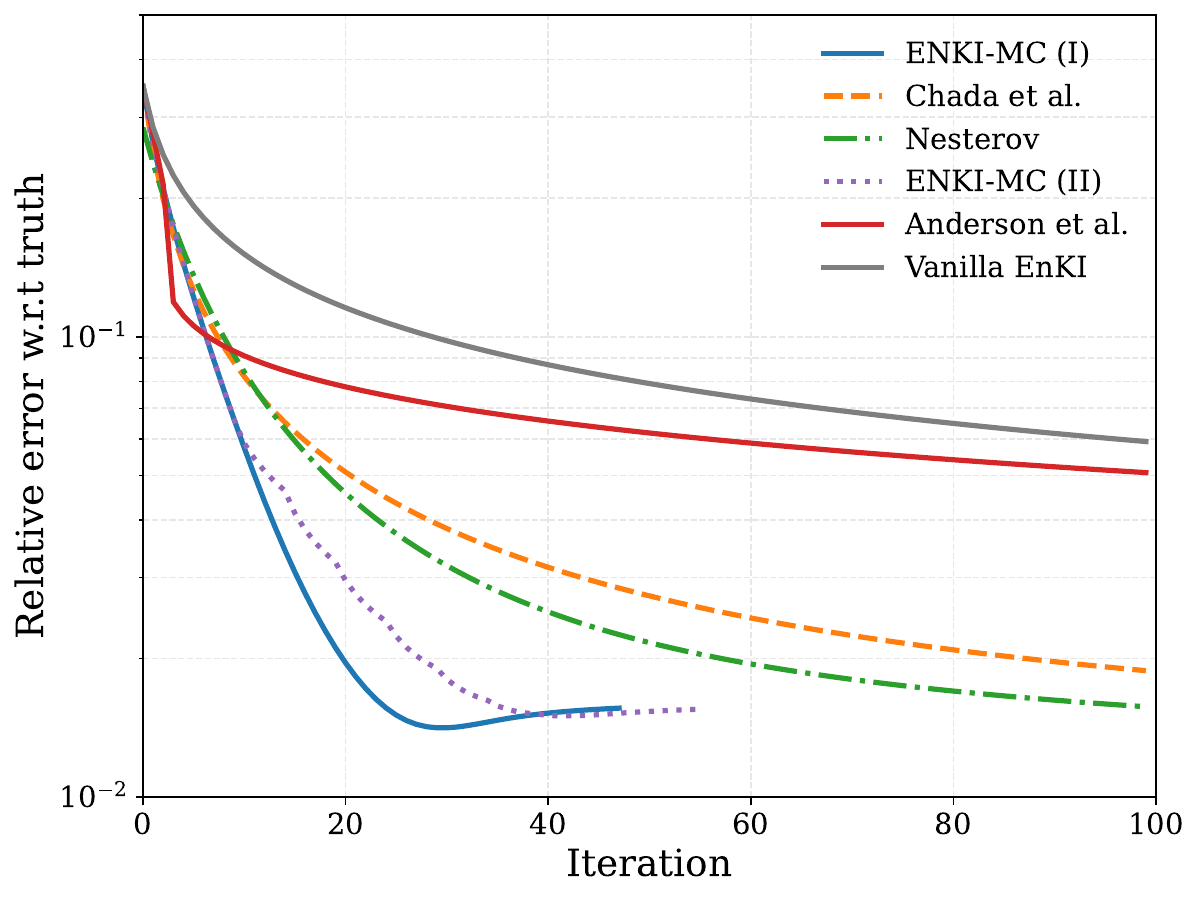}
\end{minipage}%
\begin{minipage}{.5\textwidth}
\hspace{-0.1 cm}
 \includegraphics[scale=0.38]{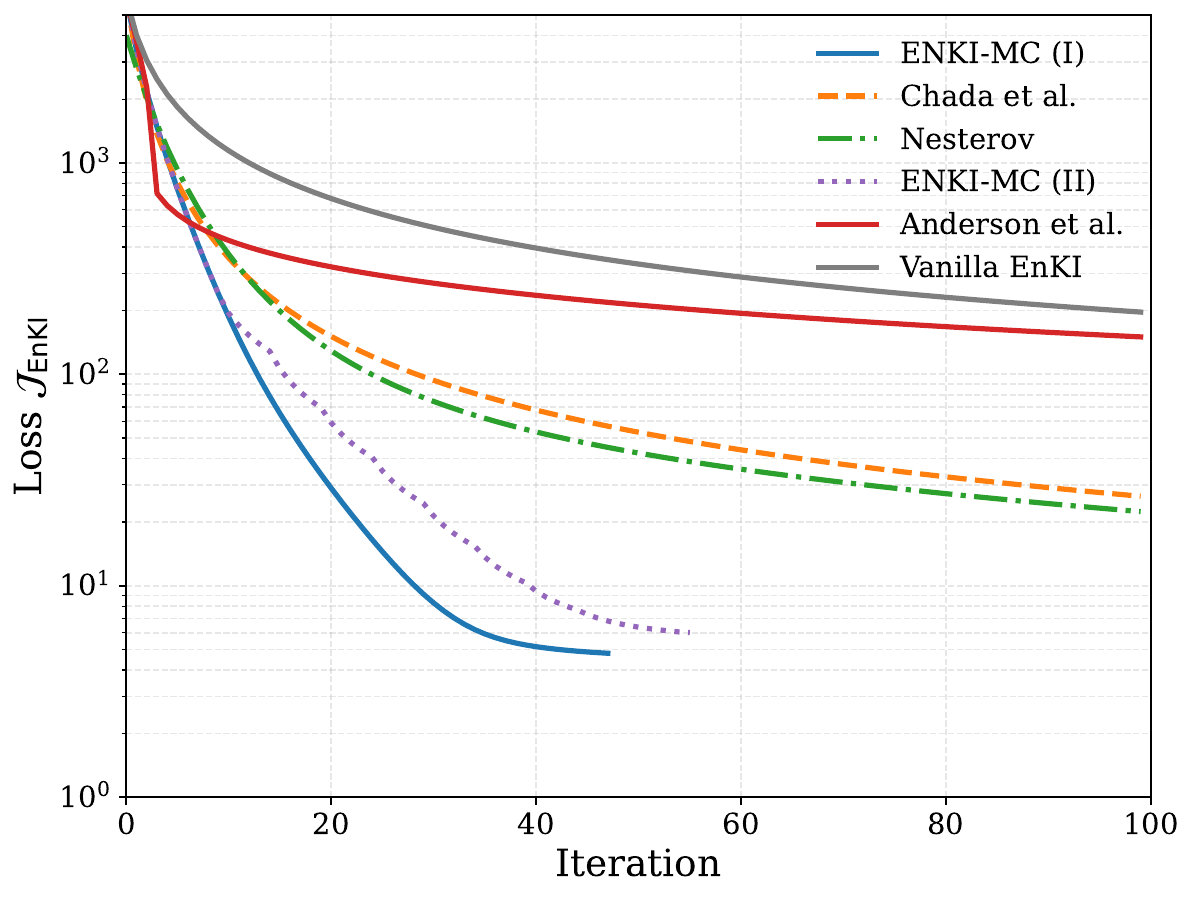}
\end{minipage}
\caption{Performance of different variants of EnKI for the  Lorenz 96 model. Left to right: Convergence of relative error w.r.t truth for different methods;  Convergence of loss defined in \eqref{enki_loss_approx}.}
\label{lorentz}
\end{figure}
\begin{figure}[h!]
\centering
 \includegraphics[scale=0.5]{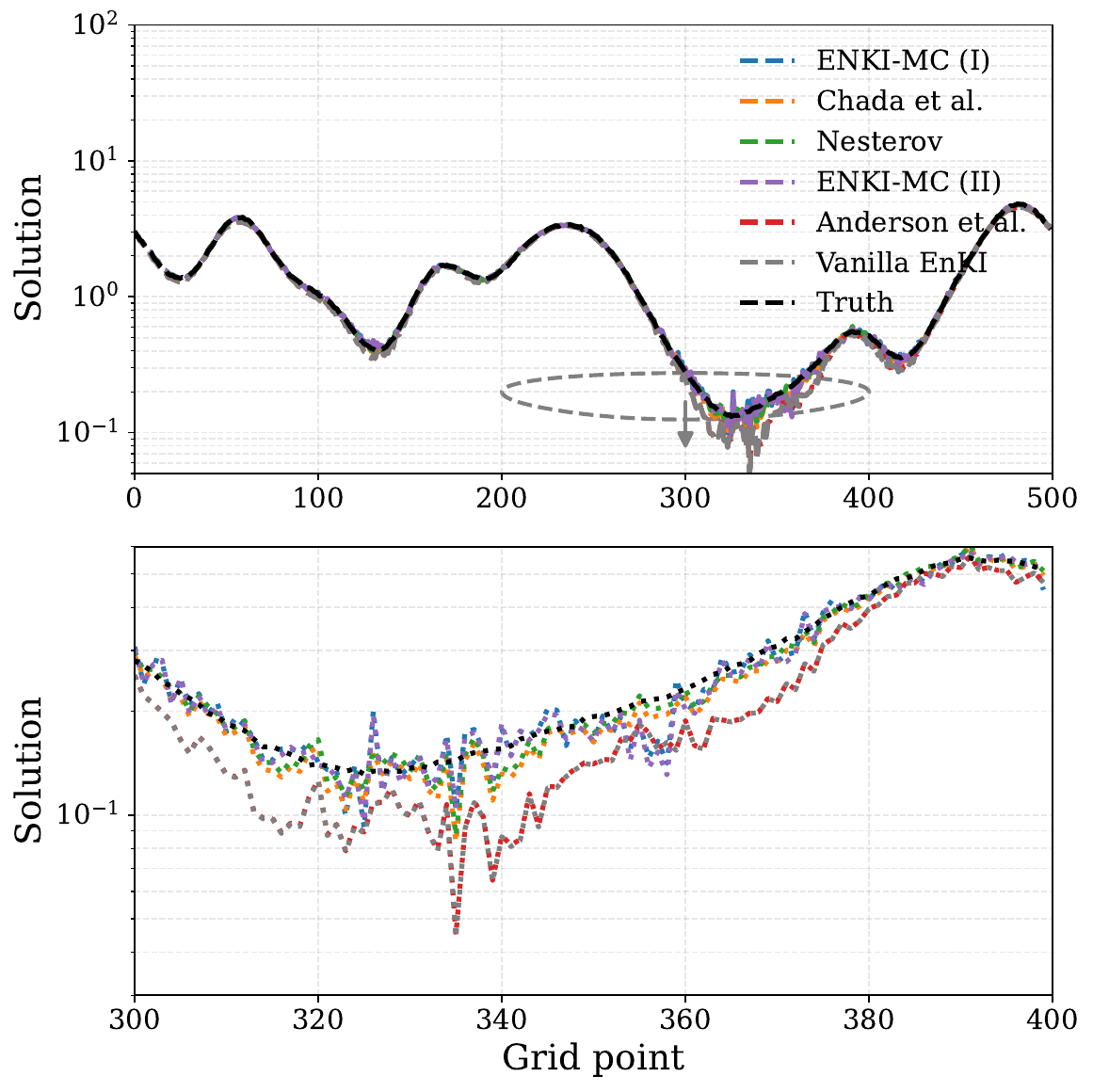}
\caption{Inverse solution achieved by different EnKI variants for Lorenz 96 model. }
\label{lorent}
\end{figure}
Table \ref{num_result4} shows that EnKI-MC (I) is the best performing method in terms of convergence time providing a speedup of about $2$ times  over Chada et al. \cite{chada2019convergence} (the next best method in terms of convergence speed). One interesting observation is that even though  Nesterov Acceleration \cite{vernon2025nesterov} produced  relative error w.r.t truth comparable to 
EnKI-MC (I)  and EnKI-MC (II), the corresponding loss  $\sJ_{EnKI}$ achieved by  Nesterov Acceleration \cite{vernon2025nesterov}  is considerably higher than our approach as evident from Figure \ref{lorentz}. Therefore, in this case our approach produced a solution which minimized the data misfit function  \eqref{enki_loss_approx} much more effectively than other methods. Further theoretical investigation is necessary to understand why this occurs. 

The inverse solution obtained by different variants of EnKI are shown in Figure \ref{lorent}.  From Figure \ref{lorent} we see that EnKI-MC (I), EnKI-MC (II), Chada et al. \cite{chada2019convergence} and Nesterov Acceleration \cite{vernon2025nesterov}  produced the best quality solutions out of all approaches.

\begin{table}[h!]
\caption{Performance comparison of different Ensemble Kalman Inversion Algorithms for the Lorenz 96 model.}
\centering
\begin{tabular}{|c | c | c |c|c|c|}
        \hline
   Method  & Time to reach  & No of & No of&   Relative error   \\ 
      & ($\epsilon$ tolerance) & iterations & $\sG$ evaluations & w.r.t ground truth   \\  \hline
   EnKI-MC (I)  & $4$ min& 47 & 23500 &0.0156 \\ \hline
         EnKI-MC (II)  &$8$ min&55& 27500 &0.0155 \\ \hline
Chada et al. \cite{chada2019convergence} &$10$ min & 113 &56500 & 0.0177\\  \hline
Vanilla EnKI (\eqref{non_lin_enki}) &$16$ min  & 202  & 101000 & 0.0438  \\  \hline
Nesterov Acceleration \cite{vernon2025nesterov}  &$9$ min& 100 & 50000 & 0.0157\\  \hline
Anderson et al. \cite{anderson2007adaptive} &$14$ min &159 & 79500 &0.0436   \\  \hline
\end{tabular} 
 \label{num_result4}
\end{table}

 \subsection{Nonlinear parameter inversion in a steady-state heat equation}

In this section, we consider a 2D nonlinear inverse problem. 
The heat equation is a two-dimensional elliptic PDE of the form:
\begin{equation}
  \grad \cdot \LRp{e^u \grad {p}}= f \quad {in }\quad  \Omega:= (0,\ 1)^2, \quad \quad {p}=0 \quad {in } \quad \partial \Omega.    
  \label{heat_cont}
\end{equation}
where $f$ is the forcing function, $u(x_1,\ x_2)$ is the conductivity field (parameter field), and $p(x_1,\ x_2)$ is the temperature field.  \eqref{heat_cont} is discretized with finite difference method (5-point stencil) resulting in 
a discrete parameter field $\bu$ of dimension $n=2304$. Our aim is to infer $\bu$ given measurement of temperature $p(x_1,\ x_2)$ at $500$ random grid  points. Therefore, in this case we have $m=500$. We also choose $f=1$ in this experiment. 
For demonstration we draw a ground truth parameter field $\ub^\dagger$ from  a  zero mean Gaussian prior $\GM{{\bf{0}}}{\sC}$ with the covariance $\sC$  defined as $\sC=(\Delta)^{-2}$, where 
homogeneous Dirichlet boundary conditions is used to define the inverse of $\Delta$. 
The ensembles for EnKI are drawn from $\GM{{\bf{0}}}{\sC}$ and we set the number of ensembles $N$ to be $50$.

\subsubsection{Convergence acceleration with   EnKI-MC (I) and EnKI-MC (II)}
Similar convergence acceleration observed in previous examples is seen here as well. In particular,  Figure \ref{ENKF_100_samples_poisson_relative} (left) shows the convergence of relative error with respect to truth where we observe that EnKI-MC (I) and EnKI-MC (II) converges the fastest out of all the approaches.  Figure \ref{ENKF_100_samples_poisson_relative} (right) shows a similar trend for the convergence of the loss $\sJ_{EnKI}$ defined in \eqref{enki_loss_approx}.  
Table \ref{num_result3} shows that EnKI-MC (I) is the best performing method in terms of convergence time providing a speedup of about $4.5$ times  over Chada et al. \cite{chada2019convergence} (the next best method in terms of convergence speed). We also observe that our approach has the minimum number of function evaluations for convergence leading to tremendous speedup over other approaches. Table \ref{num_result3} also shows that EnKI-MC (II) achieves the best relative error w.r.t truth out of all approaches. 

The inverse solution obtained by different variants of EnKI are shown in Figure \ref{fig:EnKI_combined}. Further, the absolute error  in computed solutions with respect to truth is also shown in Figure \ref{fig:heat_error}. It is clear from Figure \ref{fig:heat_error} that EnKI-MC (I) and EnKI-MC (II) produced the best solution out of all approaches.
 \begin{figure}[h!]
\centering
\begin{minipage}{.5\textwidth}
 \includegraphics[scale=0.38]{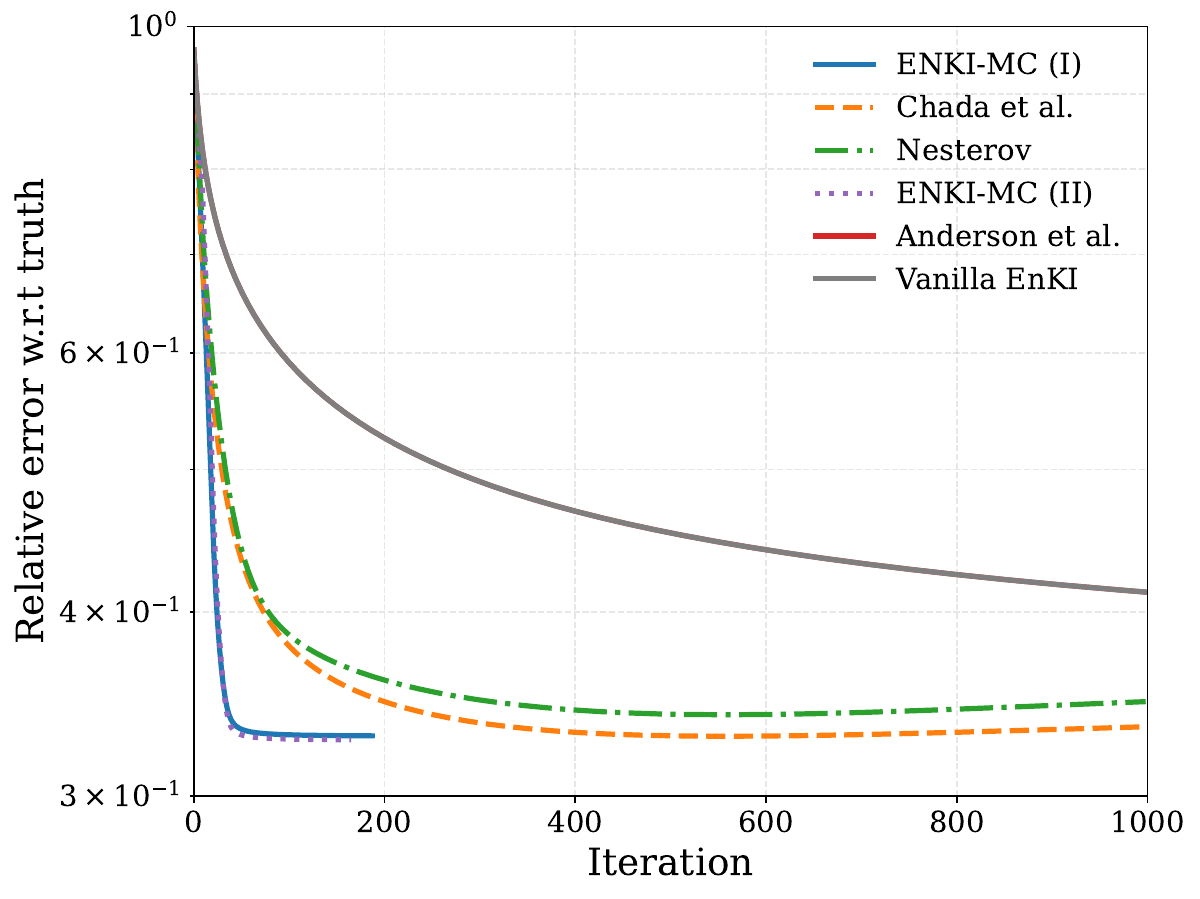}
\end{minipage}%
\begin{minipage}{.5\textwidth}
\hspace{-0.1 cm}
 \includegraphics[scale=0.38]{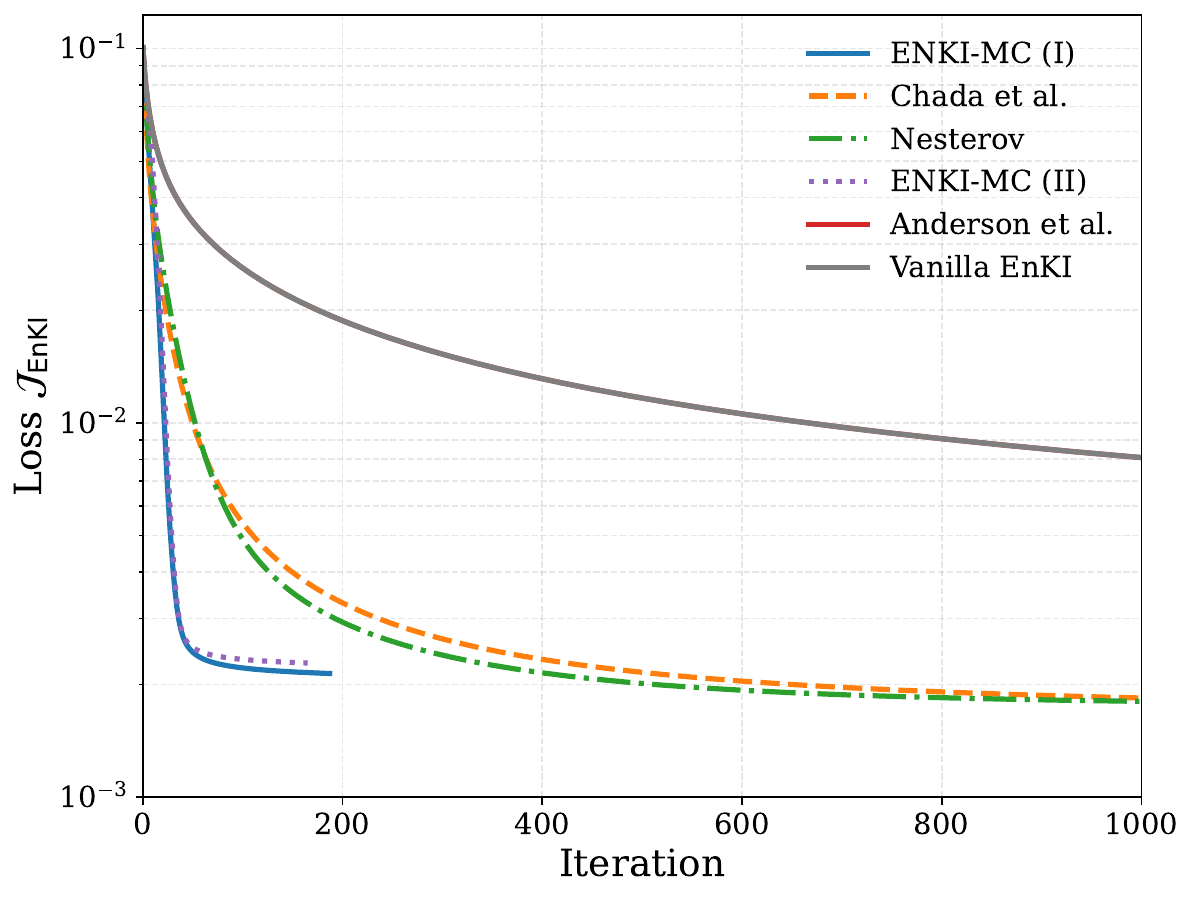}
\end{minipage}
\caption{Performance of different variants of EnKI for the steady-state heat equation. Left to right: Convergence of relative error w.r.t truth for different methods;  Convergence of loss defined in \eqref{enki_loss_approx}.}
\label{ENKF_100_samples_poisson_relative}
\end{figure}

\begin{figure}[h!]
  \hspace{-0.8 cm}
    % ===== Row 1 =====
    \begin{tabular}{ccc}
       \hspace{0.3 cm} \includegraphics[scale=0.20]{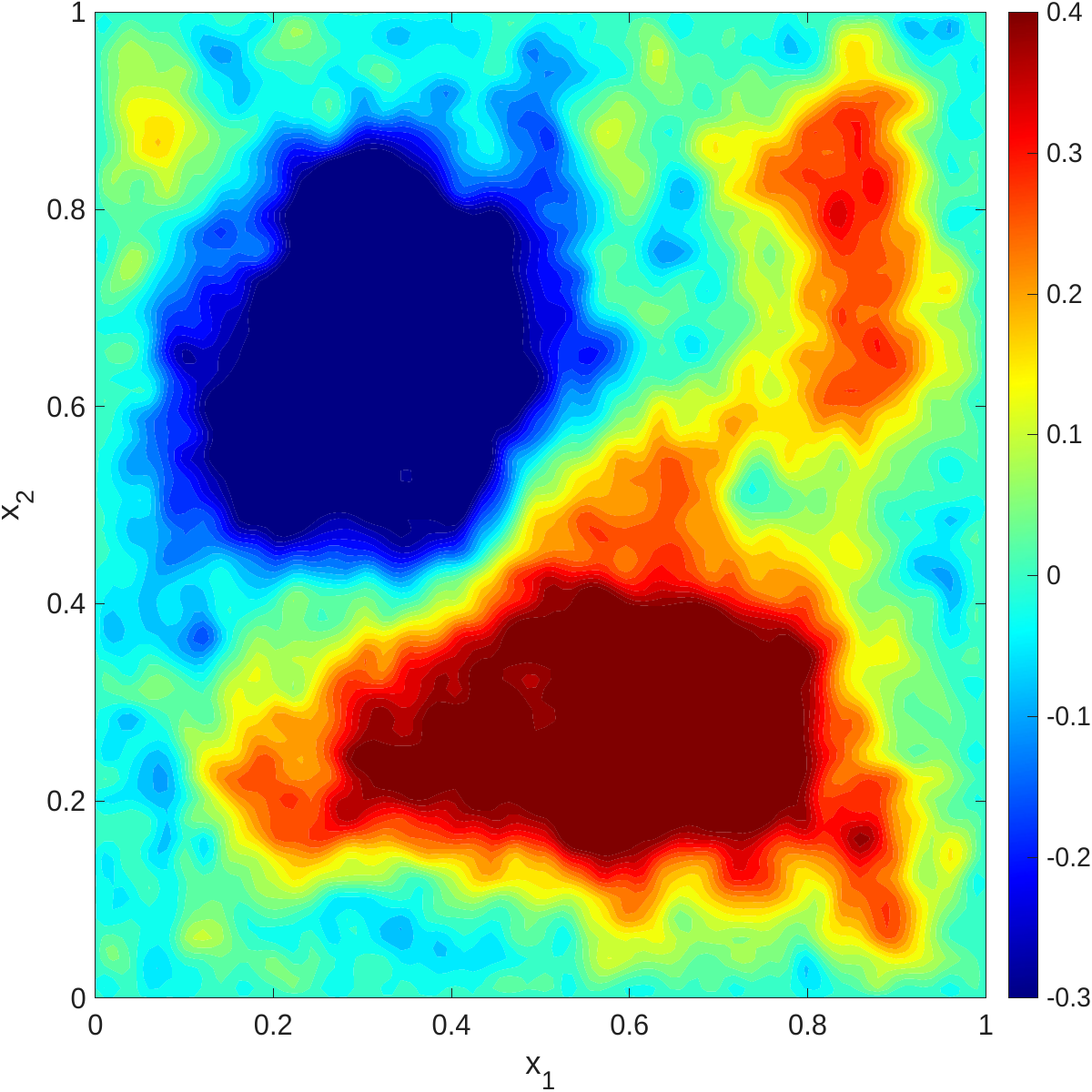} &
       \hspace{0.3 cm}  \includegraphics[scale=0.20]{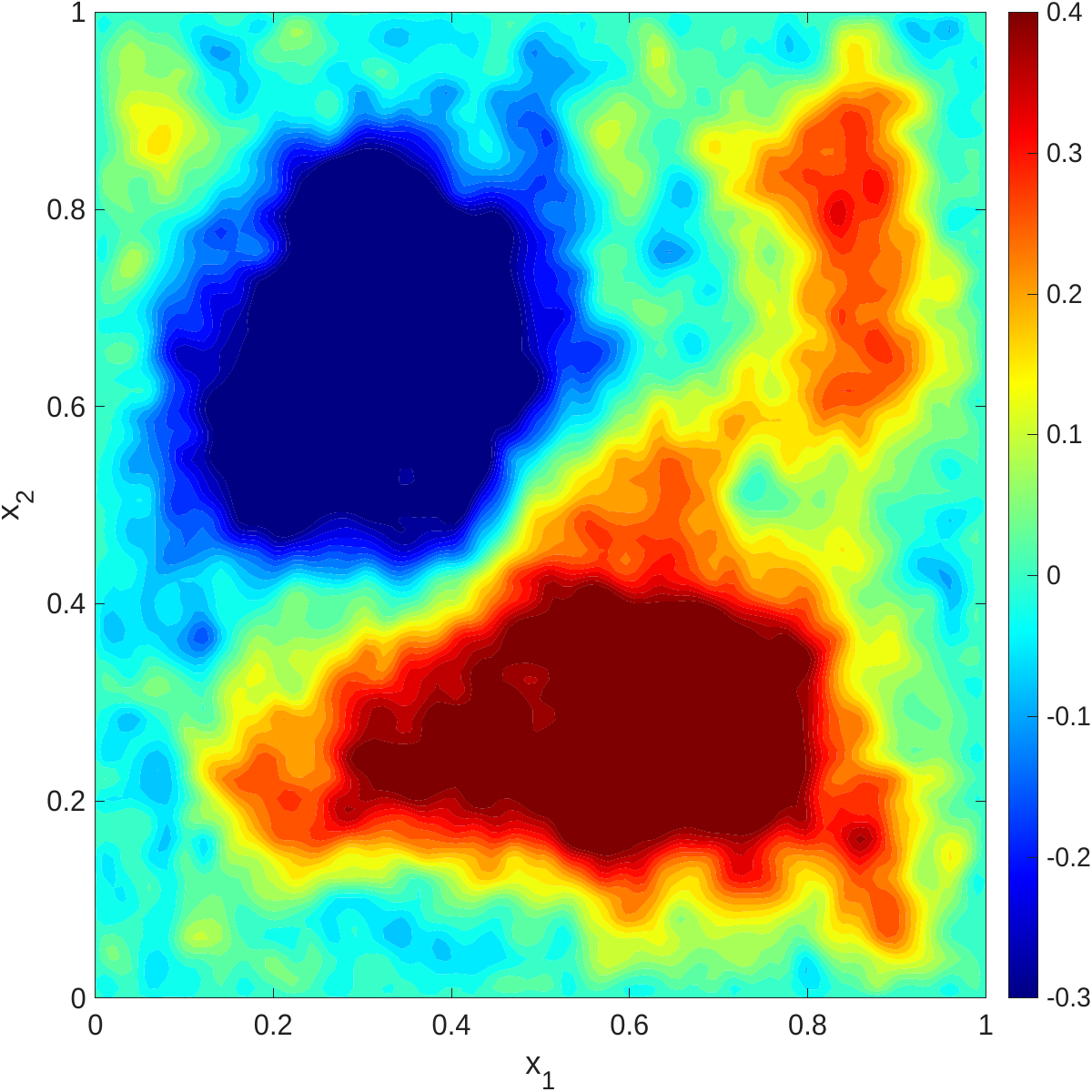} &
         \hspace{0.3 cm}\includegraphics[scale=0.20]{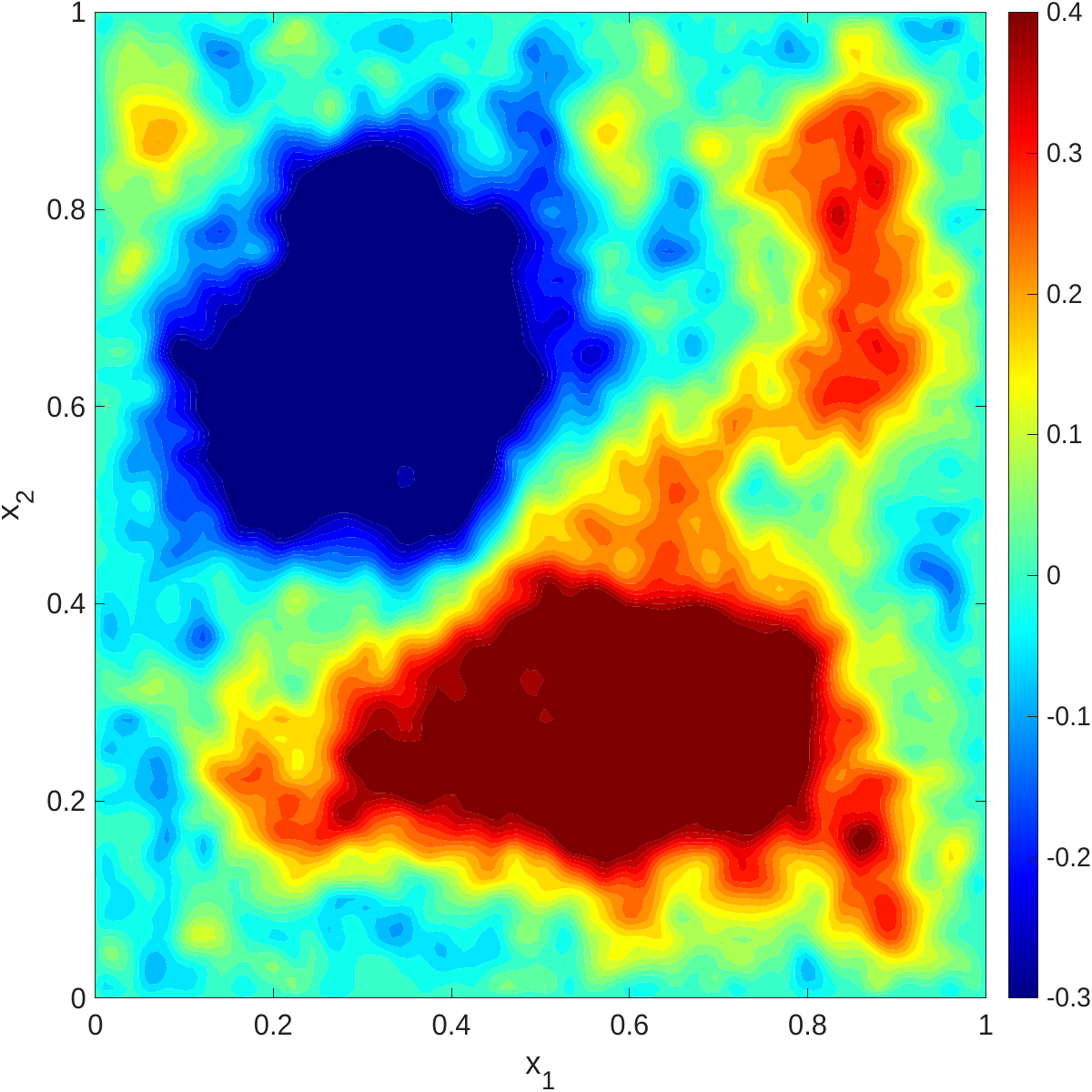} \\
    \end{tabular}

    \vspace{0.3cm} % vertical space between rows

    % ===== Row 2 =====
    \hspace{-0.8 cm}
    \begin{tabular}{ccc}
      \hspace{0.3 cm}  \includegraphics[scale=0.20]{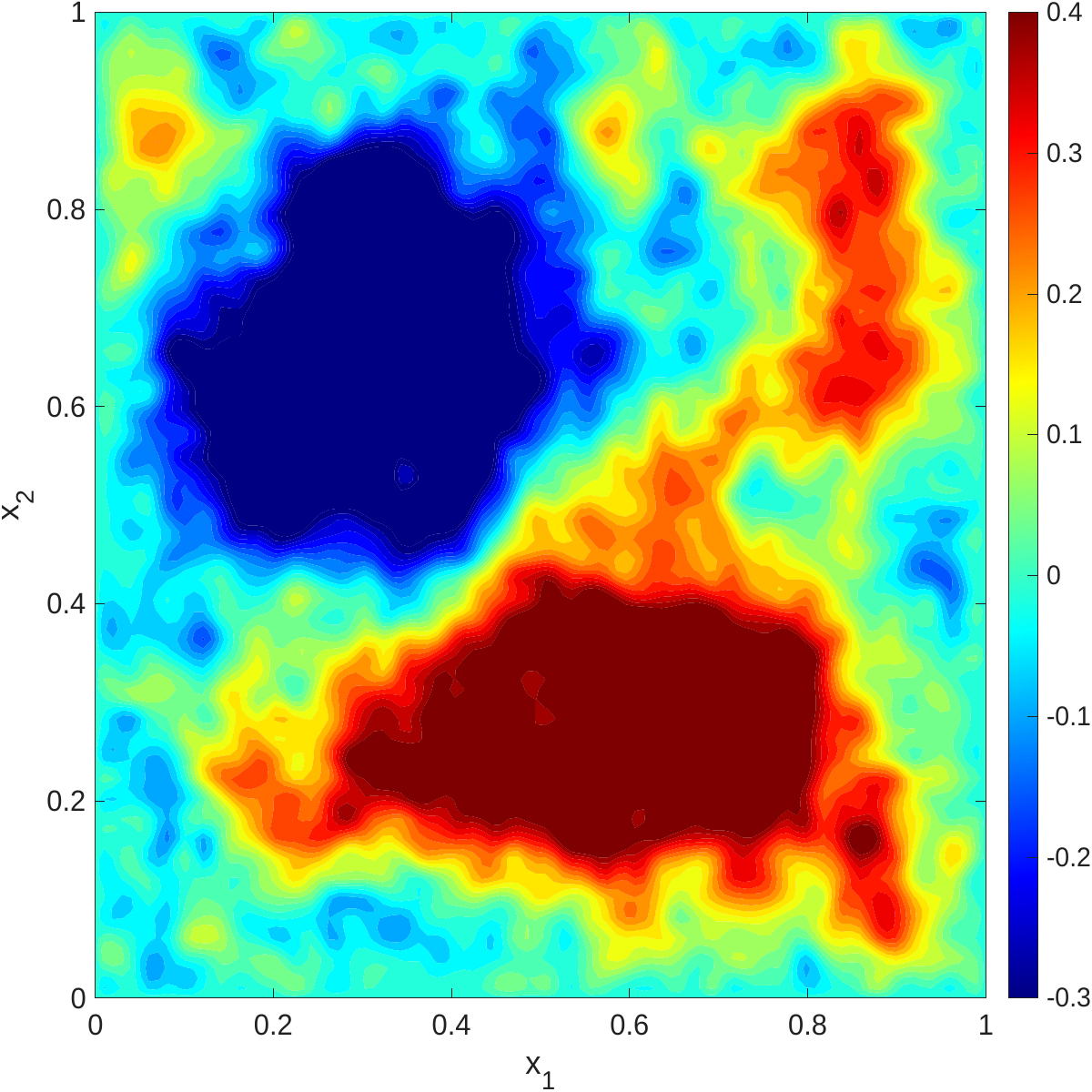} &
       \hspace{0.3 cm} \includegraphics[scale=0.20]{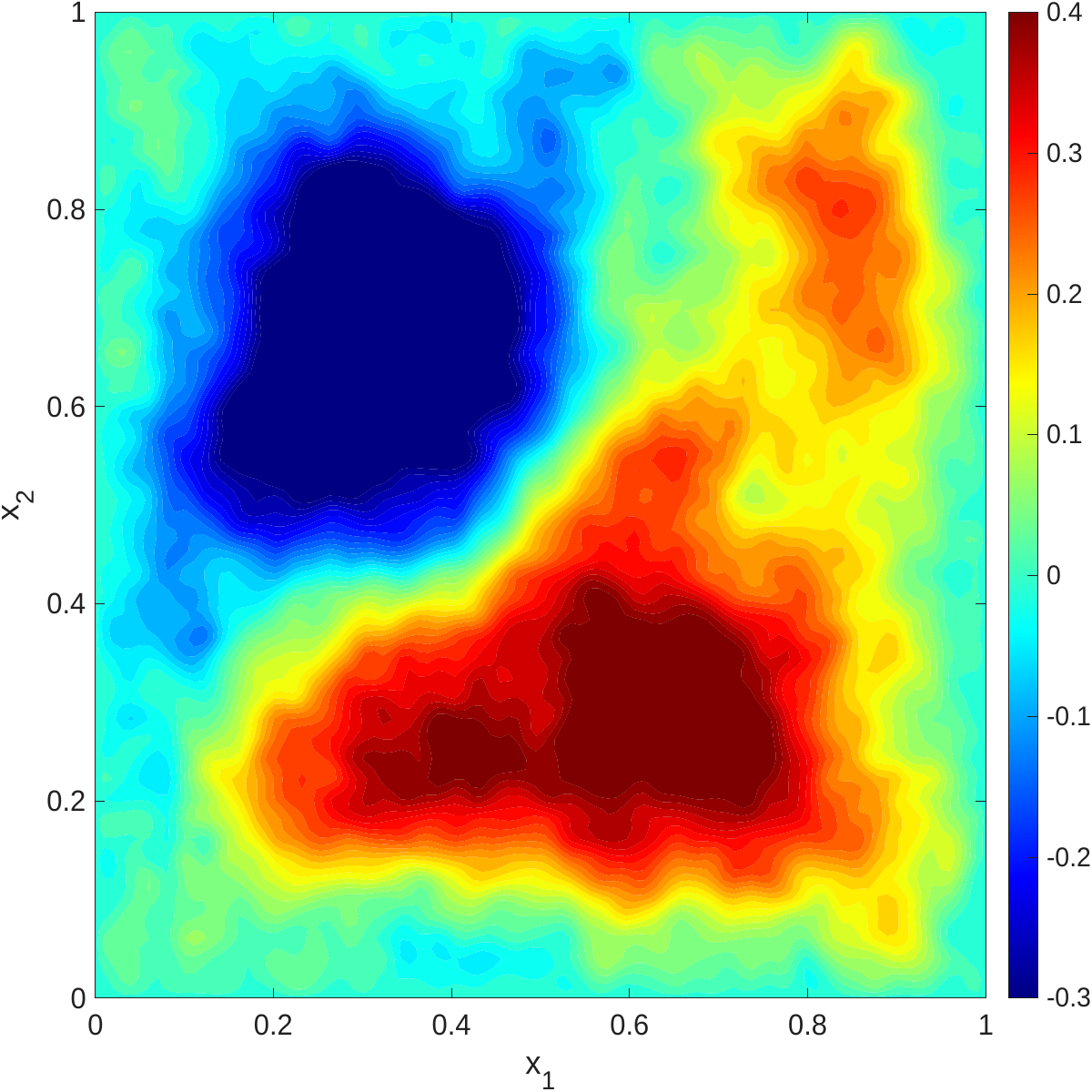} &
     \hspace{0.3 cm}   \includegraphics[scale=0.20]{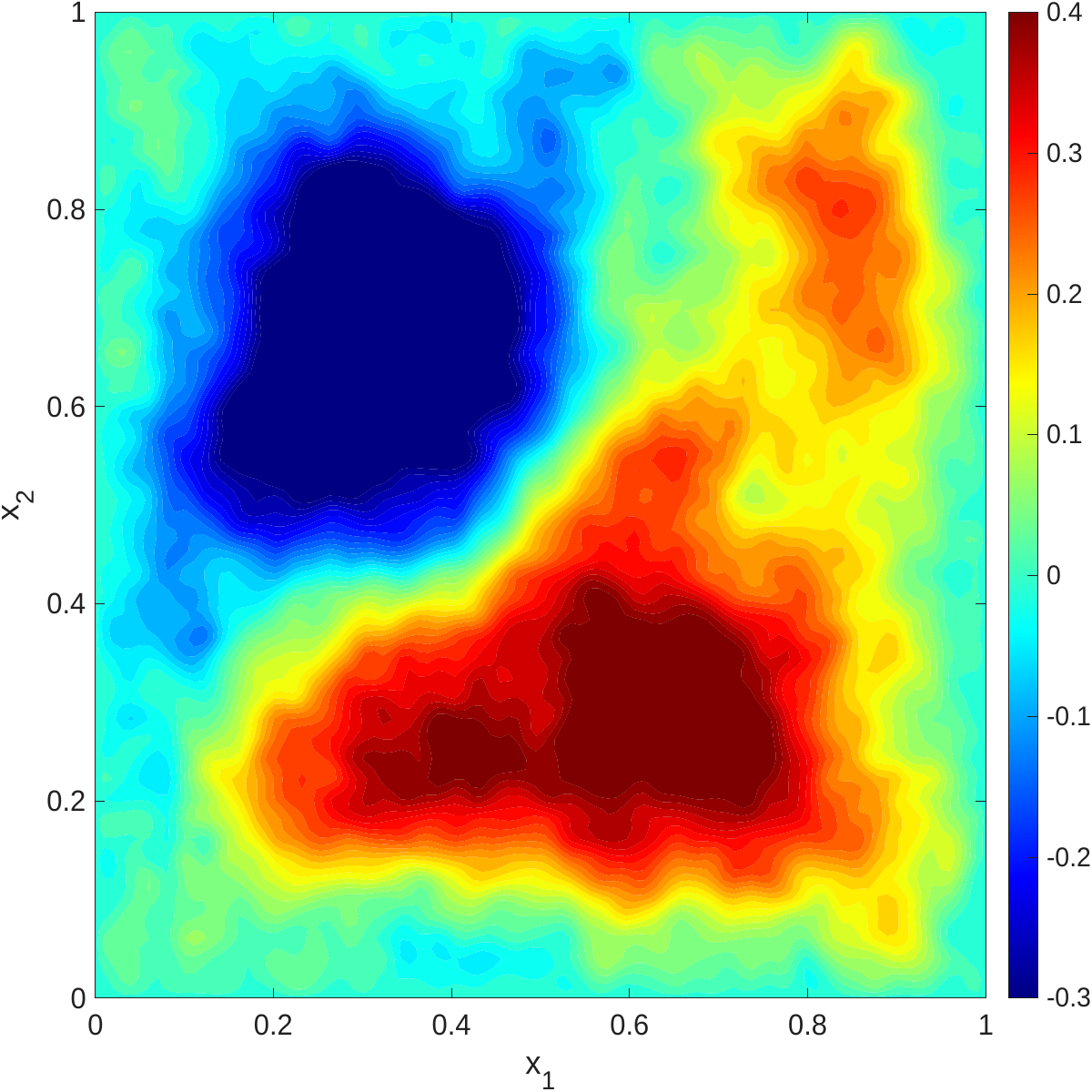} \\
    \end{tabular}

    \vspace{0.3cm} % vertical space between rows

    % ===== Row 3 =====
    \hspace{4.5 cm}
    \begin{tabular}{c}
        \includegraphics[scale=0.20]{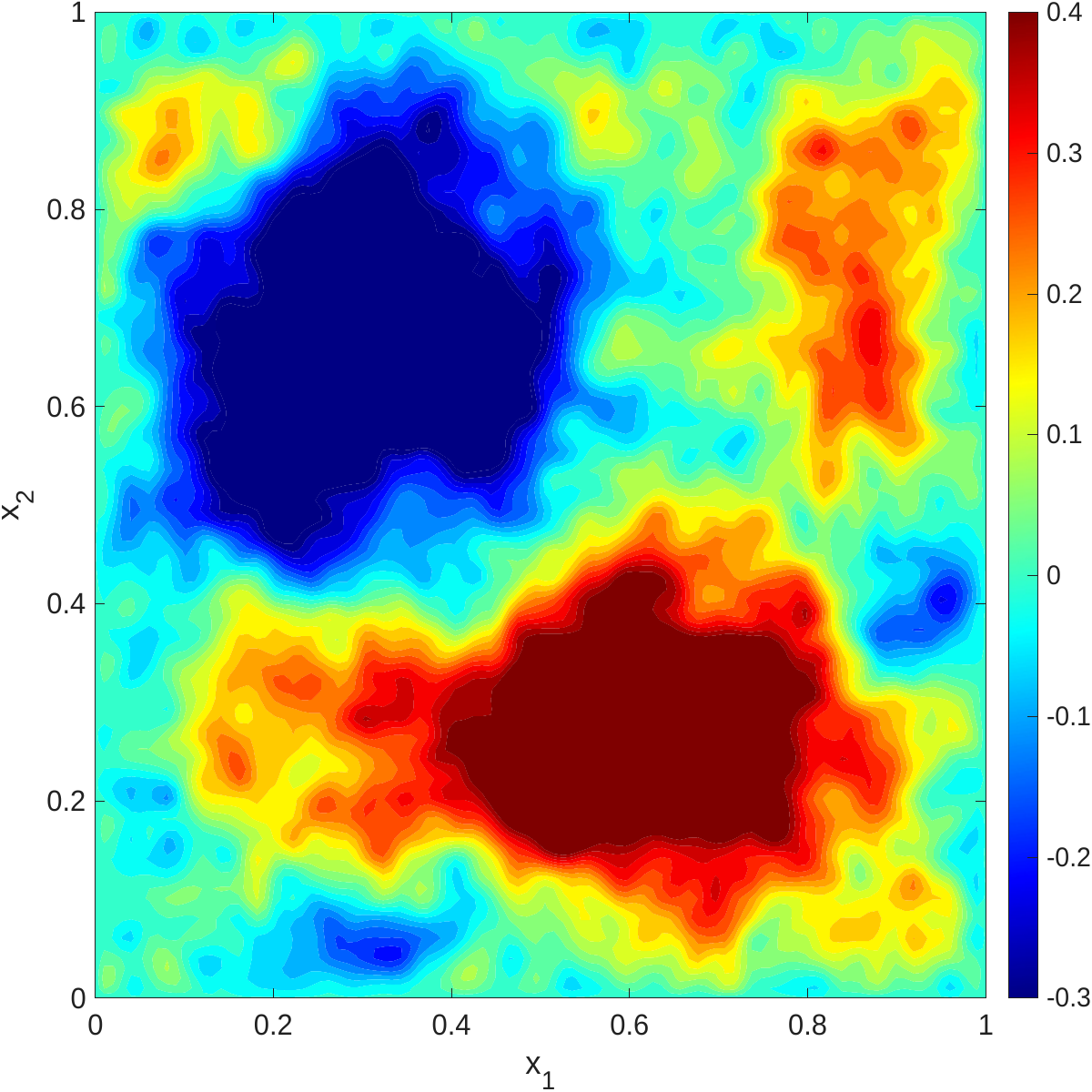} \\
    \end{tabular}

    \caption{Solutions computed using different variants of EnKI for the steady-state heat equation.
    Top row (left to right): EnKI-MC (I); EnKI-MC (II); Chada et al. \cite{chada2019convergence}. \\
    Middle row: Nesterov Acceleration \cite{vernon2025nesterov}; Anderson et al. \cite{anderson2007adaptive}; Vanilla EnKI. \\
   Third row: True solution.}
    \label{fig:EnKI_combined}
\end{figure}
  
\begin{figure}[h!]
  \hspace{-0.8 cm}
    % ===== Row 1 =====
    \begin{tabular}{ccc}
        \includegraphics[scale=0.24]{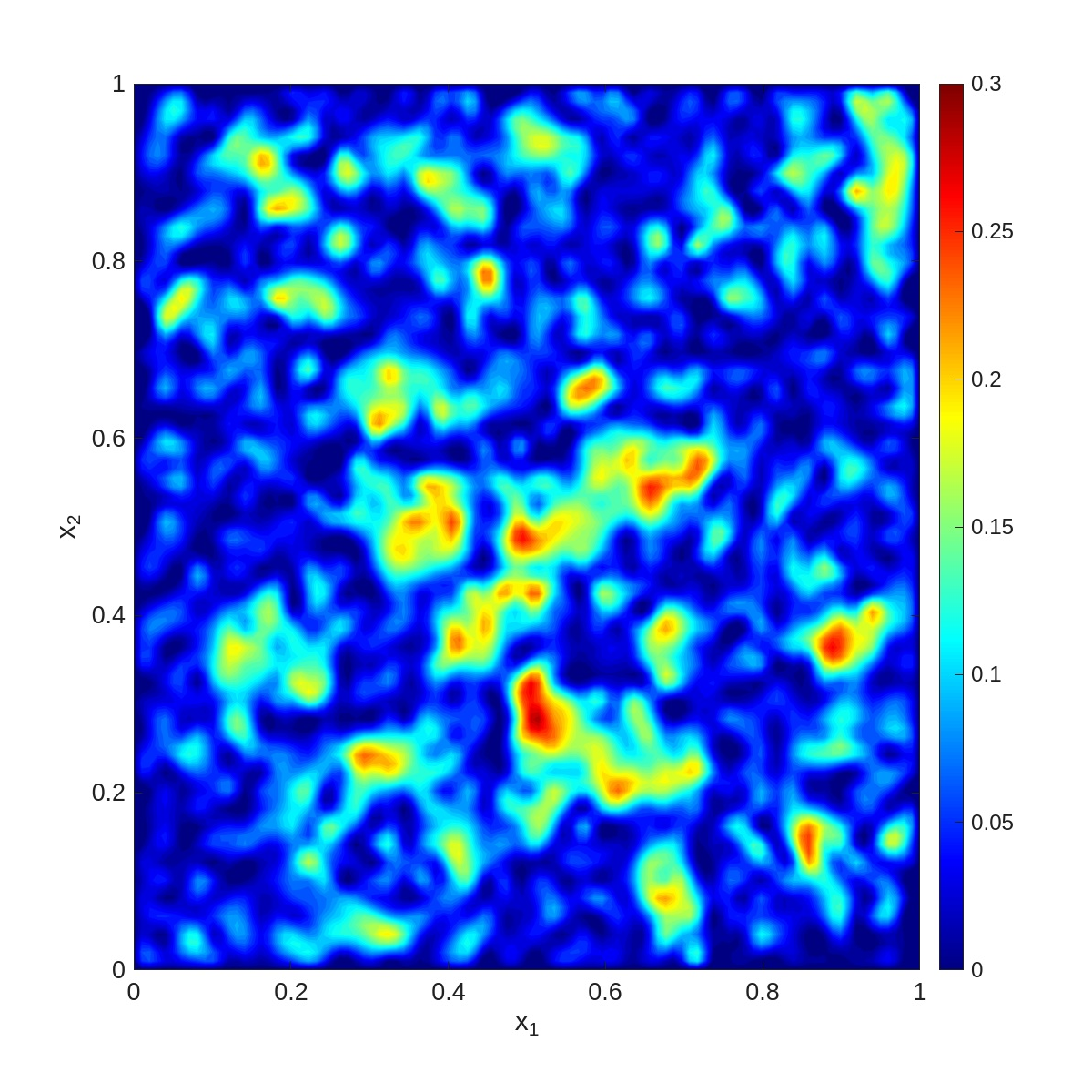} &
        \includegraphics[scale=0.22]{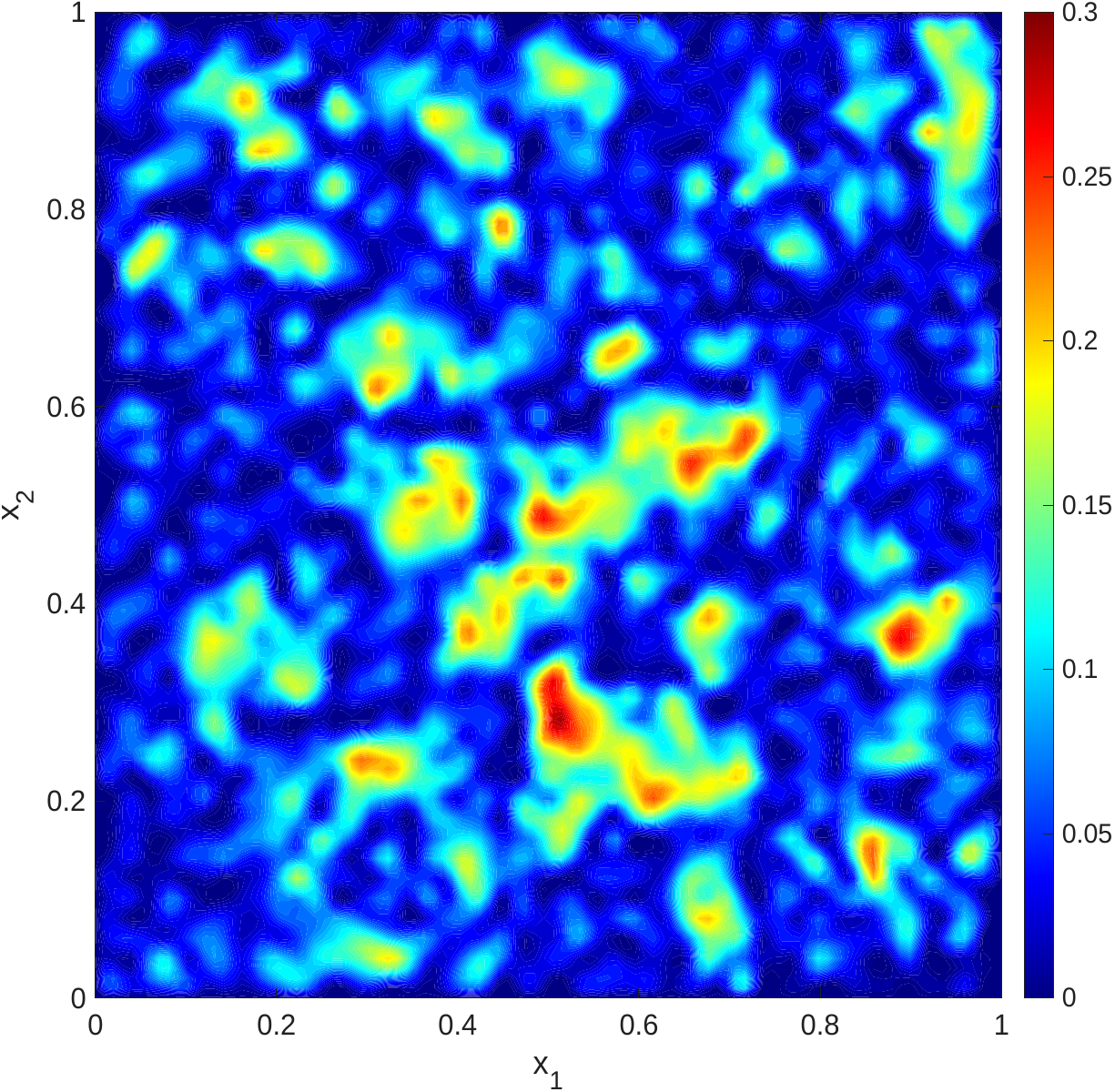} &
        \includegraphics[scale=0.24]{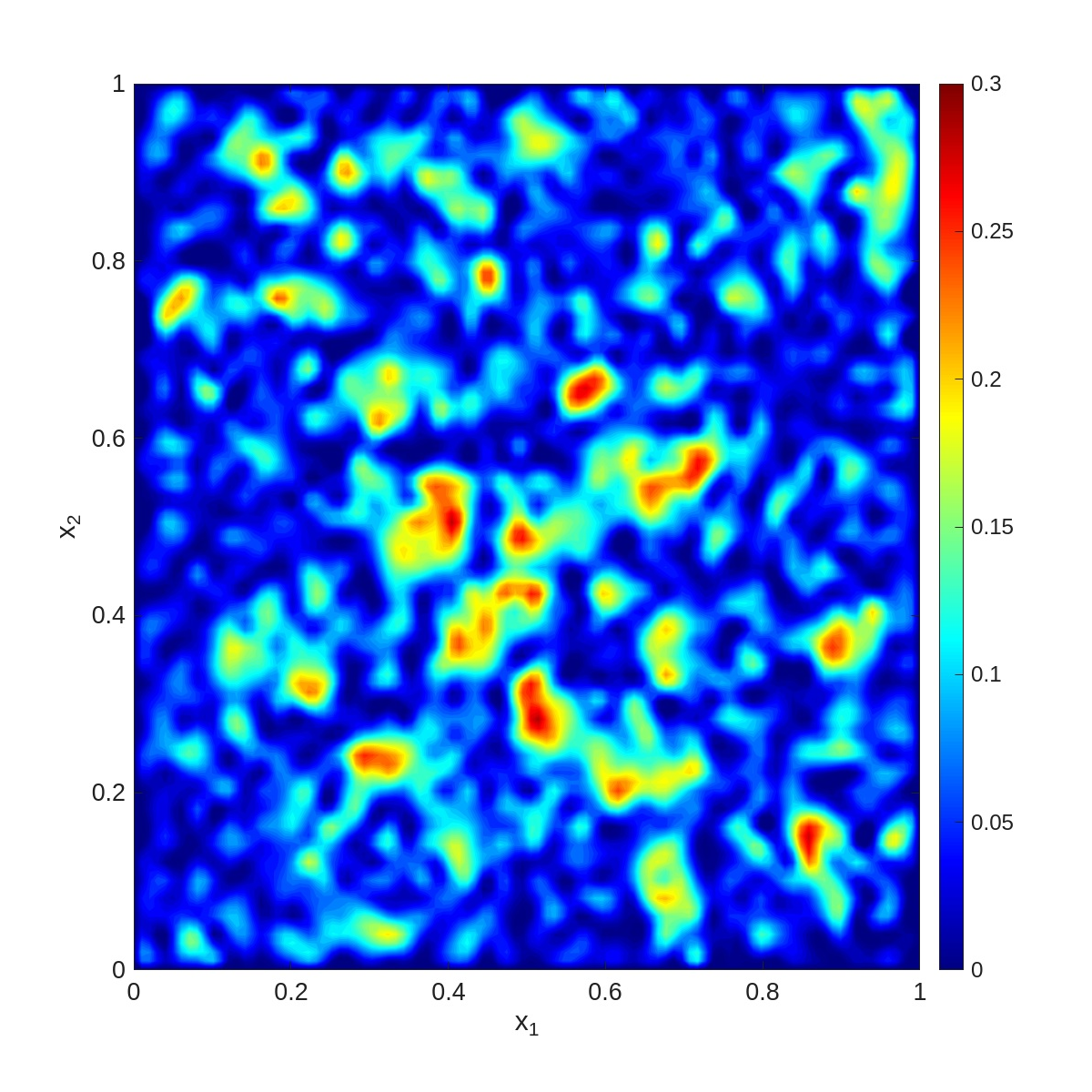} \\
    \end{tabular}

    \vspace{0.3cm} % vertical space between rows

    % ===== Row 2 =====
    \hspace{-0.8 cm}
    \begin{tabular}{ccc}
       \hspace{0.1 cm}  \includegraphics[scale=0.22]{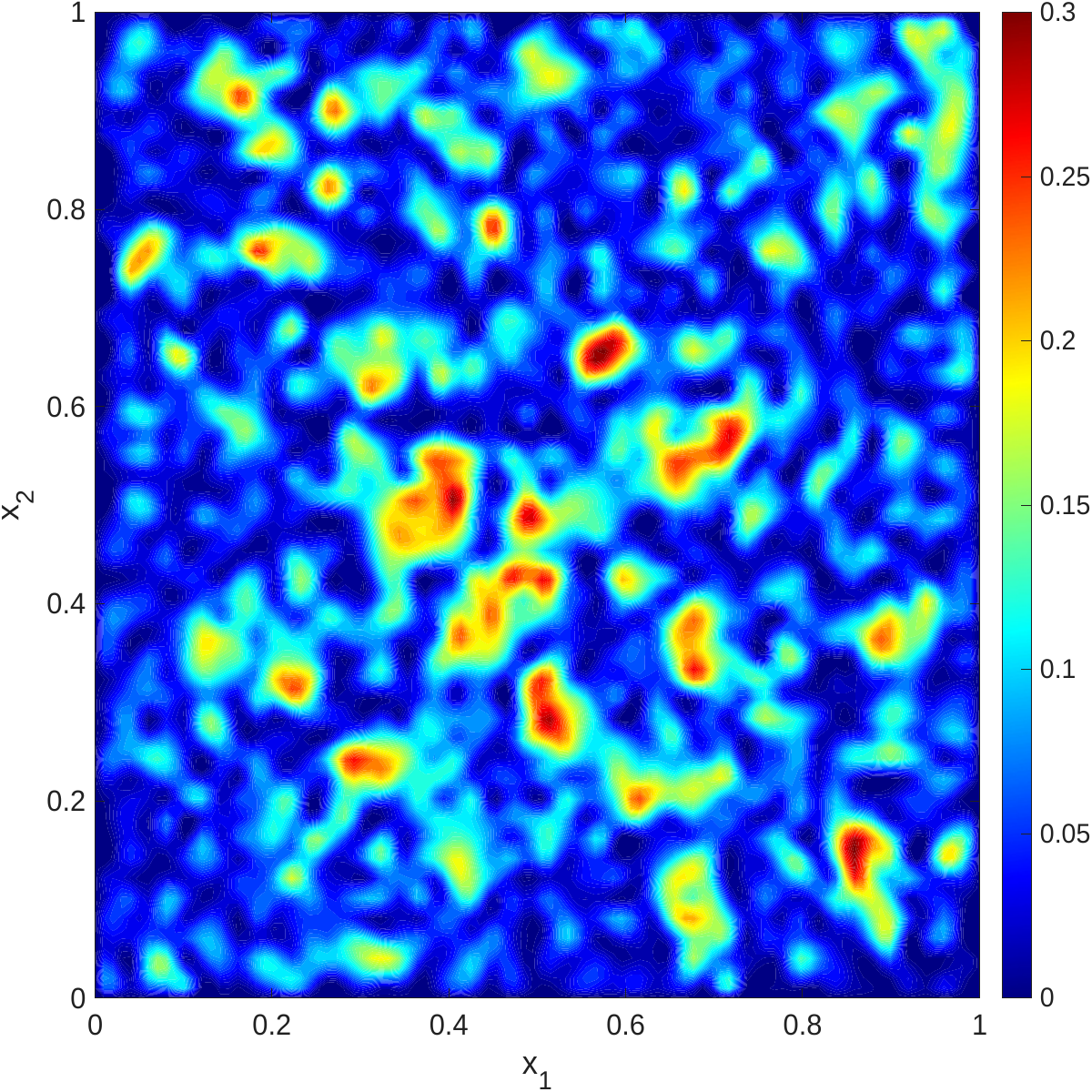} &
         \hspace{0.1 cm}  \includegraphics[scale=0.22]{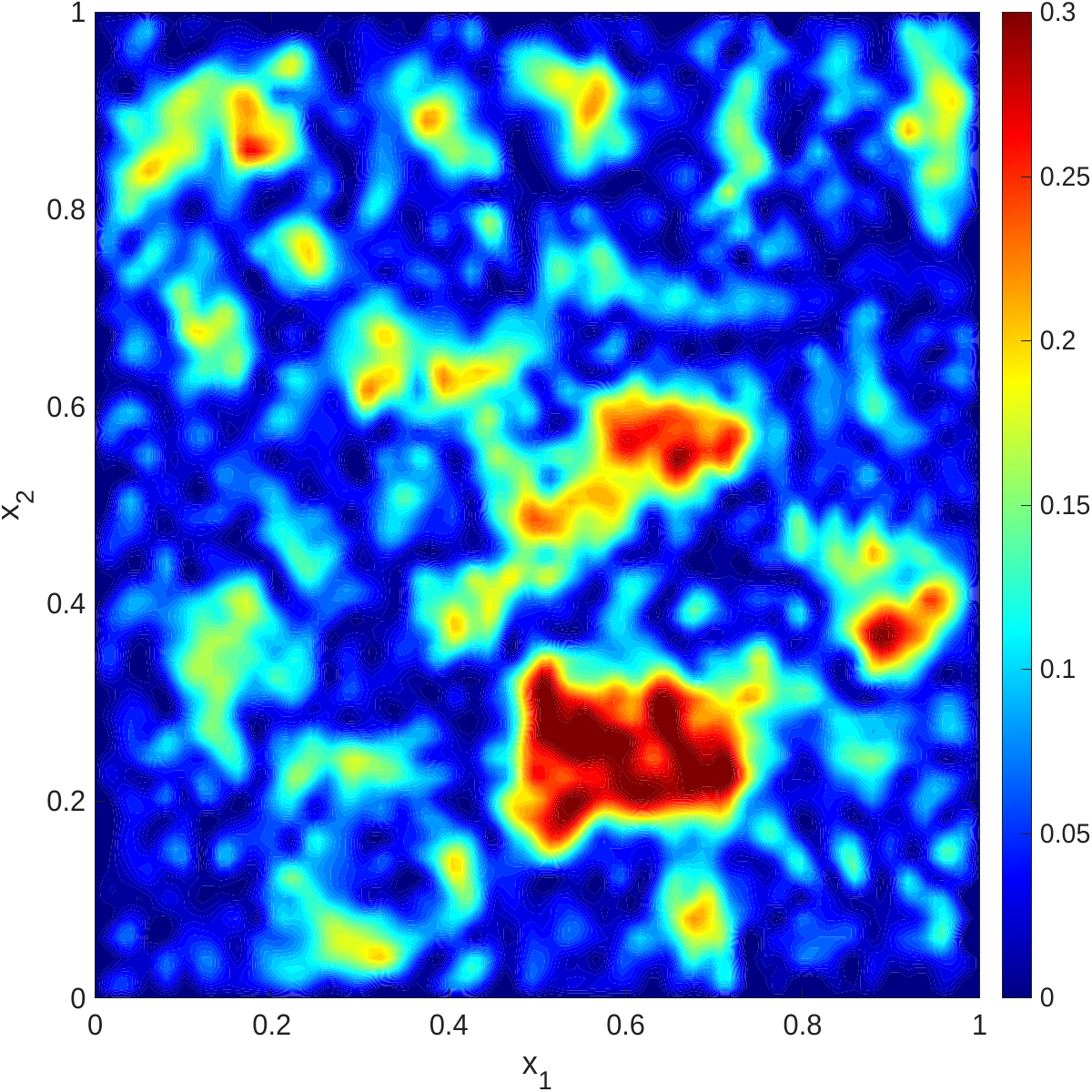} &
      \hspace{0.1 cm}    \includegraphics[scale=0.22]{Figures/Poisson/VANILLA_ERROR_HEAT.pdf} \\
    \end{tabular}

    \caption{Absolute error in solutions with respect to truth computed using different variants of EnKI for the steady-state heat equation.
    Top row (left to right): EnKI-MC (I); EnKI-MC (II); Chada et al. \cite{chada2019convergence}. \\
    Middle row: Nesterov Acceleration \cite{vernon2025nesterov}; Anderson et al. \cite{anderson2007adaptive}; Vanilla EnKI. \\
   Third row: True solution.}
    \label{fig:heat_error}
\end{figure}

\begin{table}[h!]
\caption{Performance comparison of different Ensemble Kalman Inversion Algorithms for parameter inversion in a steady-state heat equation.}
\centering
\begin{tabular}{|c | c | c |c|c|c|}
        \hline
   Method  & Time to reach  & No of & No of &   Relative error  \\ 
      & ($\epsilon$ tolerance) & iterations & $\sG$ evaluations& w.r.t ground truth   \\  \hline
   EnKI-MC (I)  & $266$ min& 188 & 9400 &0.329 \\ \hline
         EnKI-MC (II)  &$360$ min& 166 & 8300 & 0.327 \\ \hline
Chada et al. \cite{chada2019convergence} &$1242$ min & 1049 & 52450 & 0.335 \\  \hline
Vanilla EnKI (\eqref{non_lin_enki}) &$2038$ min  & 1446  & 72300 & 0.395  \\  \hline
Nesterov Acceleration \cite{vernon2025nesterov}  &$1250$ min& 1050 & 52500 & 0.347\\  \hline
Anderson et al. \cite{anderson2007adaptive} &$2038$ min & 1446 & 72300 & 0.395   \\  \hline
\end{tabular} 
 \label{num_result3}
\end{table}

\section{Conclusion}

This works presented a new perspective on the understanding of the Ensemble Kalman filter (EnKF)
for inverse problems. We showed that the Kalman filter equation can be derived as the solution to a randomized dual problem. The perspective allowed us to derive a new non-asymptotic result for EnKF which we  verified numerically for the 1D Deconvolution problem. Theorem \ref{non_asymp} also allowed us to rigorously explain the need to increase
the sample size $N$ in the small noise limit, a phenomenon commonly observed in practice (see remark \ref{import_theo}).
Further, based on the duality perspective we also developed EnKI-MC (I) and EnKI-MC (II)-two variants of the Ensemble Kalman Inversion Algorithm (EnKI) for convergence acceleration. For all the numerical examples, we observed that the developed variants led to significant convergence acceleration in comparison to other methods while also producing better quality inverse solution at termination of the algorithm. Even though EnKI-MC (I) and EnKI-MC (II), presented in Algorithm \ref{Algo_full} and Algorithm \ref{Algo_full_II}, are variants or extensions of the vanilla EnKI algorithm, the same methodology can be incorporated into other  EnKI variants currently available. For example, the method can be  adapted to the TEnKI (Tikhonov EnKI) approach presented in \cite{chada2020tikhonov}. This extension will be explored in the future.

\appendix

\section{Details of other approaches used in comparison study}
\label{other_approaches}

The description of methods used in our comparison study are provided here.
\begin{enumerate}
\item Vanilla EnKI: Iterative scheme given by \eqref{non_lin_enki}.
    \item Chada et al. \cite{chada2019convergence}: This method considers an adaptive multiplicative covariance correction of the form $\alpha(\bu_k,\ k)=k^{\beta}$ (where $0\leq \beta\leq 0.8$) suggested in \eqref{mul_cov} for convergence acceleration.  Chada et al. \cite{chada2019convergence} also considered an additive covariance inflation and Tikhonov regularization in their framework which is not considered in our implementation.   We choose $\beta=0.8$ in our experiments for maximum speed of convergence. Higher values of $\beta>1$ is not recommended as it affects the numerical stability of the algorithm and causes significant divergence from the true solution, as illustrated in Figure \ref{ENKF_IC_log}.
    \item Nesterov Acceleration \cite{vernon2025nesterov}: Implementation of Algorithm 1 in \cite{vernon2025nesterov}.
    \item Anderson et al. \cite{anderson2007adaptive}: In this  approach one models the covariance of the innovation (residual) as $S(\alpha_k)=\L_h+\alpha_k C_{pp}(\ub_k)$, where $C_{pp}(\ub_k)$ is defined in \eqref{non_lin_enki}. Assuming $\bar{\bold{r}}_k| \alpha_k\sim \GM{{\bf{0}}}{S(\alpha_k)}$, and assuming a prior for $\alpha_k$ as $\alpha_k \sim \GM{{{1}}}{1}$, the Maximum A Posteriori estimate for $\alpha_k$ is given as:
    \begin{equation}
        \alpha_k=\arg \min_{\alpha_k} \frac{1}{2}\LRp{{\bold{r}}_k^T(S(\alpha_k))^{-1}{\bold{r}}_k}+\frac{1}{2} \log \LRp{\mathrm{det}(S(\alpha_k))}+\frac{1}{2} \LRp{{\bold{r}}_k-1}^2.
        \label{estimated_ander}
    \end{equation}
    The estimated $\alpha_k$ from solving \eqref{estimated_ander} is used as the multiplicative covariance correction factor in \eqref{mul_cov} (that is we set $\alpha(\bu_k,\ k)=\alpha_k$.
\end{enumerate}

\bibliographystyle{unsrt}
\bibliography{ceo}

\end{document}